\documentclass[11pt]{amsart}
\usepackage[colorlinks = true,
            linkcolor = red,
            urlcolor  = blue,
            citecolor = blue,
            anchorcolor = blue]{hyperref}

\makeatletter
\def\@tocline#1#2#3#4#5#6#7{\relax
  \ifnum #1>\c@tocdepth 
  \else
    \par \addpenalty\@secpenalty\addvspace{#2}%
    \begingroup \hyphenpenalty\@M
    \@ifempty{#4}{%
      \@tempdima\csname r@tocindent\number#1\endcsname\relax
    }{%
      \@tempdima#4\relax
    }%
    \parindent\z@ \leftskip#3\relax \advance\leftskip\@tempdima\relax
    \rightskip\@pnumwidth plus4em \parfillskip-\@pnumwidth
    #5\leavevmode\hskip-\@tempdima
      \ifcase #1
       \or\or \hskip 1em \or \hskip 2em \else \hskip 3em \fi%
      #6\nobreak\relax
    \dotfill\hbox to\@pnumwidth{\@tocpagenum{#7}}\par
    \nobreak
    \endgroup
  \fi}
\makeatother

\usepackage{tikz,amsthm,amsmath,amstext,amssymb,amscd,epsfig,euscript, mathrsfs, dsfont,pspicture,multicol, mathtools,graphpap,graphics,graphicx,times,enumerate,subfig,sidecap,wrapfig,color}

\usepackage{ hyperref}
\usepackage{enumerate}
\usepackage{esint}
\usepackage{verbatim} 

 \numberwithin{equation}{section}

\newcommand{\R}{{\mathbb R}}       
\newcommand{\N}{{\mathbb N}}       %
\newcommand{\Z}{{\mathbb Z}}       

\newcommand{\DD}{{\mathcal D}}

\newcommand{\LL}{{\mathcal L}}

\newcommand{\om}{{\Omega}}
\newcommand{\hm}{{\omega}}

\newcommand{\dist}{ \operatorname{dist}}
\newcommand{\diam}{ \operatorname{diam}}
\newcommand{\bmo}{ \operatorname{BMO}}

\newcommand{\ran}{ \operatorname{Ran}}

\newcommand{\supp}{\operatorname{supp}}

\newcommand{\vphi}{{\varphi}}
\newcommand{\ve}{{\varepsilon}}
\newcommand{\vv}{{\vspace{2mm}}}
\newcommand{\vvv}{\vspace{4mm}}
\newcommand{\wt}[1]{{\widetilde{#1}}}

\newcommand{\khalf}{{\mathcal{K}_{\textup{Dini}}}}

\newcommand{\bu}{{\bar{u}}}
\newcommand{\ump}{{u_M^+}}
\newcommand{\umm}{{u^-_{m}}}

\newcommand{\osc}{\mathop{\textup{osc}}}
\newcommand{\ess}{\mathop{\textup{ess}}}

\def\Xint#1{\mathchoice
{\XXint\displaystyle\textstyle{#1}}%
{\XXint\textstyle\scriptstyle{#1}}%
{\XXint\scriptstyle\scriptscriptstyle{#1}}%
{\XXint\scriptscriptstyle\scriptscriptstyle{#1}}%
\!\int}
\def\XXint#1#2#3{{\setbox0=\hbox{$#1{#2#3}{\int}$ }
\vcenter{\hbox{$#2#3$ }}\kern-.58\wd0}}

\def\avint{\Xint-}

\newtheorem{theorem}{Theorem}[section]
\newtheorem{lemma}[theorem]{Lemma}
\newtheorem{corollary}[theorem]{Corollary}

\newtheorem*{lemma*}{Lemma}
\newtheorem*{theorem*}{Theorem}

\theoremstyle{definition}
\newtheorem{definition}[theorem]{Definition}
\newtheorem{example}[theorem]{Example}

\theoremstyle{remark}
\newtheorem{remark}[theorem]{\bf Remark}

\numberwithin{equation}{section}

\newcommand{\RRem}{\begin{rem}}
\newcommand{\erem}{\end{rem}}

\def\d{\partial}

\makeatletter
\def\@tocline#1#2#3#4#5#6#7{\relax
  \ifnum #1>\c@tocdepth 
  \else
    \par \addpenalty\@secpenalty\addvspace{#2}%
    \begingroup \hyphenpenalty\@M
    \@ifempty{#4}{%
      \@tempdima\csname r@tocindent\number#1\endcsname\relax
    }{%
      \@tempdima#4\relax
    }%
    \parindent\z@ \leftskip#3\relax \advance\leftskip\@tempdima\relax
    \rightskip\@pnumwidth plus4em \parfillskip-\@pnumwidth
    #5\leavevmode\hskip-\@tempdima
      \ifcase #1
       \or\or \hskip 1em \or \hskip 2em \else \hskip 3em \fi%
      #6\nobreak\relax
    \dotfill\hbox to\@pnumwidth{\@tocpagenum{#7}}\par
    \nobreak
    \endgroup
  \fi}
\makeatother

\def\Cap{\mathop\mathrm{Cap}}

\def\bN{{\mathbb{N}}}

\def\bK{{\mathbb{K}}}

\newcommand{\divv}{{\text{{\rm div}}}}

\def\lec{\lesssim}

\def\loc{\textup{loc}}

\def\lec{\lesssim}
\def\Lip{\textup{Lip}}

\begin{document}

\title[Elliptic equations with lower order terms]{Regularity theory and Green's function for elliptic equations with lower order terms in unbounded domains}


\author[Mourgoglou]{Mihalis Mourgoglou}

\address{Mihalis Mourgoglou\\
Departamento de Matem\'aticas, Universidad del Pa\' is Vasco, Barrio Sarriena s/n 48940 Leioa, Spain and\\
Ikerbasque, Basque Foundation for Science, Bilbao, Spain.
}
\email{michail.mourgoglou@ehu.eus}

\subjclass[2010]{35A08, 35B50, 35B51, 35B65, 35J08, 35J15, 35J20, 35J25, 35J67, 35J86.}
\thanks{The author was supported   by IKERBASQUE and partially supported by the grants MTM-2017-82160-C2-2-P and PID2020-118986GB-I00 of the Ministerio de Econom\'ia y Competitividad (Spain),  and by the grants IT-1247-19 and IT-1615-22 (Basque Government).
}

\begin{abstract}
We consider elliptic operators in divergence form with lower order terms of the form $Lu =  -\divv (A \cdot \nabla  u + b u ) - c \cdot \nabla u - du$, in an open set $\Omega \subset \mathbb{R}^n$, $n \geq 3$, with possibly infinite Lebesgue measure. We assume that the $n \times n$ matrix $A$ is uniformly elliptic with real, merely bounded and possibly non-symmetric coefficients,  and either $b, c \in L^{n,\infty}_{\textup{loc}}(\om)$ and $d \in L_{\textup{loc}}^{\frac{n}{2}, \infty}(\Omega)$, or $|b|^2, |c|^2, |d| \in  \mathcal{K}_{\textup{loc}}(\Omega)$, where  $\mathcal{K}_{\textup{loc}}(\Omega)$ stands for the local Stummel-Kato class. Let ${\mathcal{K}_{\textup{Dini}}}(\Omega)$ be a variant of $\mathcal{K}(\Omega)$ satisfying a Carleson-Dini-type condition. We develop a De Giorgi/Nash/Moser theory for solutions of $Lu = f - \divv g$, where $ f$ and $|g|^2  \in  {\mathcal{K}_{\textup{Dini}}}(\Omega)$ if, for $q \in [n, \infty)$,  any of the following assumptions holds: (i) $|b|^2, |d| \in {\mathcal{K}_{\textup{Dini}}}(\Omega)$ and either $c \in L^{n,q}_{\textup{loc}}(\Omega)$ or $|c|^2  \in  \mathcal{K}_{\textup{loc}}(\Omega)$; (ii) $\divv b +d \leq 0$ and either $b+c \in L^{n,q}_{\textup{loc}}(\Omega)$ or $|b+c|^2  \in  \mathcal{K}_{\textup{loc}}(\Omega)$; (iii) $-\divv c + d \leq 0$ and  $|b+c|^2 \in {\mathcal{K}_{\textup{Dini}}}(\Omega)$.   We also prove a Wiener-type criterion for boundary regularity. Assuming global conditions on the coefficients,  we show that the variational Dirichlet problem is well-posed and, assuming $-\divv c +d \leq 0$, we  construct the Green's function associated with $L$ satisfying quantitative estimates. Under the additional hypothesis $|b+c|^2  \in  \mathcal{K}'(\Omega)$, we show that it satisfies global pointwise bounds and also construct the Green's function associated with the formal adjoint  operator of $L$. An important feature of our results is that all the estimates are scale invariant and independent of $\Omega$, while we do not assume  smallness of the norms of the coefficients or coercivity of the associated bilinear form.
\end{abstract}

\maketitle

\tableofcontents

\section{Introduction}

In the present paper we will deal with elliptic equations of the form
 \begin{equation}\label{operator}
 L u= -\divv (A \cdot \nabla  u + b u ) - c\cdot \nabla u - du=0
\end{equation}
in an open set   $\Omega \subset \R^{n}$, $n \geq  3$, where  $A(x)= (a_{ij}(x))_{i, j =1}^{n}$ is a matrix with entries $a_{ij}: \om \to \R$, for $i,j \in \{1, 2, \dots, n\}$,  $b,c: \Omega \to \R^{n}$ are vector fields, and $d: \Omega \to \R$ a real-valued function. Our standing assumptions are the following:

There exist $0 < \lambda < \Lambda < \infty$, so that
\begin{align}\label{eqelliptic1}
&\lambda|\xi|^2\leq \langle A(x) \xi,\xi\rangle,\quad \mbox{ for all $\xi \in \R^n$ and a.e. $x\in\Omega$,}\\
&\langle A(x) \xi,\eta \rangle  \leq \Lambda |\xi| |\eta|, \quad \mbox{ for all $\xi, \eta \in \R^n$ and a.e. $x\in\Omega$,} \label{eqelliptic2}\\
&|b|^2, |c|^2, |d| \in  \mathcal{K}_\loc(\om) \quad \textup{or} \quad b, c \in L^{n,\infty}_\loc(\om), d \in L^{\frac{n}{2}, \infty}_\loc(\om),\label{eq:elliptic3}
\end{align}
where $ \mathcal{K}_\loc(\om)$ and $L^{n,\infty}_\loc(\om)$ stand for the local Stummel-Kato class and the local weak-$L^n$ space respectively (see Definitions \ref{def:Kato} and \ref{def:Lor-space})\footnote{Our original assumptions were $b , c \in L^{n}(\om)$ and $d \in L^{\frac{n}{2}}(\om)$.  The extension to weak Lebesgue spaces is due to an observation of G. Sakellaris in \cite{Sak}; a more detailed discussion can be found at the end of the introduction.}. In several cases, we will also need to assume one of the following negativity conditions:
\begin{equation}
 \int_{\Omega} (d\,\vphi -  b\cdot \nabla \vphi) \leq 0,\quad \textup{for all}\,\,0\leq  \vphi \in C^\infty_0(\Omega), \label{negativity}
  \end{equation}
 or
 \begin{equation}
\int_{\Omega}  (d\,\vphi +  c\cdot \nabla \vphi)  \leq 0,\quad  \textup{for all}\,\,0\leq  \vphi \in C^\infty_0(\Omega).\label{negativity2}
  \end{equation}
If \eqref{negativity} (resp. \eqref{negativity2}) holds we will  say that the $bd$ (resp. $cd$) negativity condition  is satisfied. If we reverse the inequality signs we will say that the $bd$ or $cd$ positivity condition  is satisfied.

\vv

The objective of the current manuscript is to generalize the standard theory of elliptic PDE of the form $-\divv A \nabla u =0$ in  open sets $\om \subset \R^n$, $n \geq 3$, with  possibly  infinite Lebesgue measure, to equations of the form \eqref{operator} under the aforementioned standing assumptions. In particular, we aim to show {\bf scale invariant} {\it a priori} local estimates (Caccioppoli inequality, local boundedness and weak Harnack inequality), interior and  boundary regularity for solutions of \eqref{operator}, the weak maximum  principle,  well-posedness of the Dirichlet and obstacle problems, and finally to construct the  Green's function for our operator satisfying several quantitative estimates.  It is important to highlight that neither the bilinear form associated with the elliptic equation  is  coercive, nor the norms of the coefficients are small, which is one of the main technical difficulties.

\vv

We would like to point out  that  we will only state the theorems in the main body of the paper,  just before their proofs. Nevertheless, the reader can find a detailed description of our results  in the introduction.

\vv

Let us give a brief overview of our results. In section \ref{subsec:caccio} we prove  the standard interior and boundary Caccioppoli's inequality under either negativity condition (Theorems \ref{thm:subCaccioppoli}, \ref{thm:Caccioppoli2}, and \ref{thm:boundCaccioppoli}), while, in section \ref{sec:dirichlet},  having global assumptions on the coefficients, we show the well-posedness of the generalized Dirichlet problem \eqref{boundvar} satisfying the estimate \eqref{eq:Dir-estimate-nonzero}, as well as  the validity of the weak maximum principle (Theorem \ref{thm:maxprin-bd}). This maximum principle  allows us to solve the obstacle problem in bounded domains (Theorem \ref{thm:v.i.-unilateral-lower}). Then we assume that one of the following conditions hold: 
\begin{enumerate}
\item$|b|^2, |d| \in \khalf(\om)$ and either $|c|^2  \in  \mathcal{K}_\loc(\om)$ or $c \in L^{n,q}_\loc(\om)$, for  $q \in [n, \infty)$; \label{it:intro-1}
\item $\divv b +d \leq 0$ and either $|b+c|^2  \in  \mathcal{K}_\loc(\om)$ or $b+c \in L^{n,q}_\loc(\om)$, for  $q \in [n, \infty)$; \label{it:intro-2}
\item $-\divv c + d \leq 0$ and  $|b+c|^2 \in {\mathcal{K}_{\textup{Dini}}}(\Omega)$ (for the definition see \eqref{eq:Dini-a_i}). \label{it:intro-3}
\end{enumerate}
 In section \ref{subs:refinedCaccio}, we demonstrate that the refined Caccioppoli inequality holds in the interior and the boundary (Theorems \ref{thm:bdry-exp-subCaccioppoli1} and \ref{thm:int-exp-subCaccioppoli1}), which leads to the local boundedness of subsolutions (Theorem \ref{thm:Moser}) and the weak Harnack inequality for non-negative supersolutions (Theorem \ref{thm:reverse-Holder-bu}) both in the interior and at the boundary.  In section \ref{subs:regularity} we prove interior and boundary regularity for solutions and  finally, assuming the $cd$-negativity condition and either $b+c \in L^{n,q}(\om)$, for  $q \in [n, \infty)$, or $|b+c|^2  \in  \mathcal{K}'(\om)$, we use the aforementioned results to construct  the Green's function associated with the operator  $L$ satisfying several quantitative estimates. Under the additional hypothesis $|b+c|^2  \in  \mathcal{K}'(\om)$, we show global pointwise bounds  and construct the Green's function associated with the formal adjoint  operator of $L$. All our estimates are scale invariant and independent of the Lebesgue measure of the domain.

\vv

We now briefly review the history of work in this area for linear elliptic equations in divergence form with merely bounded leading coefficients and singular lower order terms.  The generalized Dirichlet problem in the Sobolev space $W^{1,2}$ is well-posed if there exists a unique $u \in W^{1,2}(\om)$ such that $Lu=f+\divv g$ and $u - \phi \in W^{1,2}_0(\om)$ for fixed $\phi \in W^{1,2}(\om)$ and $f, g_i \in L^2(\om)$. Moreover, there exists a constant $C_{\phi,f, g}$ so that  the global estimate $\| u\|_{W^{1,2}(\om)} \lesssim C_{\phi,f, g}$ holds. For operators without lower order terms this problem has a long history and we refer to \cite[p.214]{GiTr} and the references therein for details. In \textit{bounded} domains, in the presence of lower order terms, Ladyzhenskaya and Ural'tseva \cite{LU} and Stampacchia \cite{St2} proved well-posedness  of the generalized Dirichlet problem assuming conditions related to the coercivity of the operator or smallness  of the norms of the lower order coefficients. This was quite restrictive as, for example, the ``bad" terms coming from the lower order coefficients can be absorbed in view  of  smallness. Gilbarg and Trudinger  \cite{GiTr} gave an extension of the previous results replacing the smallness conditions by  the assumptions $b, c, d \in L^\infty(\om)$ assuming either \eqref{negativity} or \eqref{negativity2}. In fact,  they only need $b, c \in L^s(\om)$ and $d \in L^{s/2}(\om)$, for some $s>n$. Recently, Kim and Sakellaris  \cite{KSa}, generalized  it to operators whose coefficients are in the critical Lebesgue space. Unfortunately, in all those results, the implicit constant in the global estimate  depends on the Lebesgue measure of $\om$ and thus, they cannot be extended to unbounded domains by approximation. On the other hand, in unbounded domains with possibly infinite Lebesgue measure, already in 1976, Bottaro and Marina \cite{BM} proved that, if $b, c \in L^n(\om)$, $d \in L^{n/2}(\om)+L^\infty(\om)$, and $ \divv b + d \leq \mu <0$, then the generalized Dirichlet problem is well-posed. To our knowledge, this was the first paper establishing well-posedness  in such generality. Using the same method, Vitanza and Zamboni \cite{ViZa},  showed well-posedness of the same problem when $|b|^2, |c|^2, |d|  \in  \mathcal{K}'(\om)$.

\vv

The local pointwise estimates find their roots in De Giorgi's celebrated paper \cite{DeG} on the H\"older continuity of solutions of elliptic equations of the form $-\divv A \nabla u=0$, where Theorems \ref{thm:Moser} (i) and \ref{thm:Holder-cont} were proved in this special case (see also \cite{Na}). A few years later, Moser gave a new proof of De Giorgi's theorem in \cite{Mos1}. The same results were extended in equations of the form \eqref{operator} by Morrey \cite{Mor} when $b, c \in L^q$ and $d \in L^{q/2}$, for $q>n$ and Stampacchia \cite{St1} (in more special cases).  Moser also  established the weak Harnack inequality for solutions of $-\divv A \nabla u=0$ in \cite{Mos2}, while Stampacchia \cite{St2} proved all the a priori estimates for equations of the form \eqref{operator}  with $c \in L^n$ and $|b|^2, d \in L^s$, $s>n/2$, assuming that \eqref{negativity} holds and the radius of the balls are sufficiently small so that the respective norms of the lower order coefficients on those balls are small themselves. If the lower order coefficients are in the Stummel-Kato class $\mathcal{K}(\om)$ with sufficiently small norms, one can find such results in \cite{CFG} and \cite{Ku} (see the references therein as well). Under the  assumptions $b,c \in L^n$, and $d \in L^{\frac{n}{2}}$,  Kim and Sakellaris \cite{KSa} also established local boundedness for subsolutions of the equation \eqref{operator} satisfying either \eqref{negativity} or \eqref{negativity2} and $b+c \in L^s$, $s >n$ (with implicit constants dependent on the Lebesgue measure of $\om$). They also constructed a counterexample showing that if \eqref{negativity2} holds, it is necessary to have an additional hypothesis on $b+c$ (see \cite[Lemma 7.4]{KSa}). 

\vv

Proving the boundary regularity of solutions to the generalized Dirichlet problem with data $\phi \in W^{1,2}(\om) \cap C(\overline \om)$ has been an important problem in the area and stems back to the work of Wiener for the Laplace operator \cite{Wi}. Wiener characterized the points $\xi \in \d \om$ that a solution converges continuously to the boundary  in terms of the capacity of the complement of the domain in the balls centered at $\xi$. The proof was tied to the pointwise bounds of the Green's function and so were its generalizations to elliptic equations. In particular, Littman, Stampacchia and Weinberger \cite{LSW} constructed the Green's function in a bounded domain for equations $-\divv A \nabla u=0$, where $A$ is real and symmetric,  proving such a  criterion and later, Gr\"uter and Widman \cite{GW}  extended their results to operators with possibly non-symmetric  $A$. For equations with lower order coefficients in bounded domains, Stampacchia \cite{St2} showed a Wiener-type criterion in sufficiently small balls centered at the boundary of $\om$. On the other hand, Kim and Sakellaris \cite{KSa} succeeded to construct the Green's function with pointwise bounds (which was their main goal) following the method of Gr\"uter and Widman, assuming either \eqref{negativity2} and $b+c \in L^n$, or \eqref{negativity} and $b+c \in L^s$, $s >n$. This is the best known result in this setting in  domains with finite Lebesgue measure. In this case though, the construction of the Green's function was not used to conclude  boundary regularity. For elliptic systems in unbounded domains,  Hofmann and Kim \cite{HK} constructed the Green's function  assuming that their solutions satisfy the interior {\it a priori } estimates of De Giorgi/Nash/Moser. They also showed boundary H\"older continuity of the solution of the Dirichlet problem with $C^\alpha(\overline \om)$ data under the stronger assumption of Lebesgue measure density condition of the complement of $\om$ in the balls centered at $\d \om$ (see also \cite{KK}).  Recently,  Davey, Hill and Mayboroda \cite{DHMa} extended \cite{HK} to systems with lower order terms in $b \in L^q$, $c \in L^s $ and $d \in L^{t/2}$, with $\min \{q, s, t\}>n$, whose associated bilinear form is coercive. For lower order coefficients in the Stummel-Kato class in domains with $C^{1,1}$ boundary, the Green’s function was constructed in \cite{IR}, while in \cite{ZhZh}, elliptic systems were considered, assuming though smallness on the norms and coercivity.

\vv

Let us now discuss our methods. Inspired  by the treatment of the Dirichlet problem in \cite{BM} and specifically the use of Lemma \ref{lem:main-splitting-Ln}, we are able to extend their results to operators with either negativity assumption (as opposed to $-\divv b + d \leq \mu<0$) by requiring solvability in the Sobolev space $Y^{1,2}$ instead of $W^{1,2}$ with non-divergence interior data in $L^{\frac{2n}{n+2}}$ instead of $L^2$. This is the ``correct" Sobolev space in unbounded domains and had already  appeared in \cite{MZ} and in connection with the  Green's function in \cite{HK}. The main difficulty lies on the fact that when we are proving the global bounds for the solution of the Dirichlet problem, we arrive to an estimate where the term
$$
\|b+c\|_{L^{n,q}(\om)}\|\nabla u\|_{L^2(\om)}^2 
$$ 
should be absorbed. But unless one has smallness of $\|b+c\|_{L^{n,q}(\om)}$ this is impossible. To deal with this issue, we use Lemma  \ref{lem:main-splitting-Ln} and split the domain in a finite number of subsets $\om_i$ where the norm ${L^{n,q}(\om_i)}$ norm of $b+c$ becomes small. We also write $u$ as a finite sum of $u_i$ so that $\nabla u_i \subset \om_i$ and,  loosely speaking, the term above can be hidden.   An iteration argument is then required, which concludes  the desired result. An approximation argument on the data and the domain yields the desired well-posedness. The same considerations apply to prove the weak maximum principle for subsolutions with either negativity condition, which, in turn, allows us  to solve the unilateral variational poblem and thus, the obstacle problem in bounded domains. As a corollary we obtain that the minimum of two subsolutions of the inhomogeneouus equation $Lu = f - \divv g$ is also a subsolution. 

\vv

Moving further to the proof of Caccioppoli inequality, some serious difficulties arise. Up to now,  Caccioppoli's inequality was unknown with so  general conditions, since it could be solved only for balls $r \leq 1$ and then rescale. This resulted to the appearance of the Lebesgue measure of the domain in the constants and so, it could not serve our purpose for scale invariant estimates. To overcome this important obstacle, we had to make a technically challenging adaptation of the  method that solves the Dirichlet problem. The idea to use this iteration method to prove standard and refined Caccoppoli inequalities is novel and turns out to be the most important ingredient that overcomes the necessity for smallness of the norms of the coefficients in order  to develop a De Giorgi/Nash/Moser theory for so general operators.

\vv

To prove local boundedeness, weak Harnack inequality, interior and boundary regularity, we have to make a non-trivial adaptation of the arguments of Gilbarg and Trudinger \cite[pp. 194--209]{GiTr}. To do so, we are required to prove a refined version of Caccioppoli inequality (Theorems \ref{thm:bdry-exp-subCaccioppoli1}-\ref{thm:int-exp-subCaccioppoli1}), which in \cite{GiTr} was immediate. This turns out to be an even more demanding task than the proof of Caccioppoli inequality itself. Once we obtain them, we  show Lemma \ref{lem:caccio-w-moser}, which is the building block of a Moser-type iteration argument. For this lemma, we need an embedding inequality (see Corollary  \ref{cor:Kato-emb-glob}) with constants independent of the domain, which we prove, since we were not able to find it in the literature (with constants independent of the domain). The use of the Stummel-Kato class $\mathcal{K}(\om)$ as an appropriate class of functions for the interior data and the lower order coefficients is not new and has its roots to Schr\"odinger operators with singular potentials (see \cite{Ku} and the references therein). Although, in our case, due to the counterexample of Kim and Sakellaris \cite{KSa} (see Example \ref{ex:c1}), $|b+c|^2$ should be in appropriate subspace of it satisfying a Carleson-Dini-type condition.  In fact,  a $\tfrac{1}{2}$-Dini condition on the Stummel-Kato modulus was imposed in \cite{RZ} to prove local boundedness of subsolutions for certain quasi-linear equations, but their constants  depended on $\om$.  Our Moser-type iteration argument in the proof of Theorem \ref{thm:Moser} follows their ideas, but to get scale invariant estimates, it is necessary to come up with the condition \eqref{eq:Kato-dini} and deal with some technical details  that required attention already in the original proof. In  Example \ref{ex:d>0}, we also show that a negativity condition is necessary to obtain local boundedness.

\vv 

Regarding interior and boundary regularity, as is customary, we  go through an application of the weak Harnack inequality. But for this, we need the positivity condition to hold which would force us to assume $L1 =0$, or  equivalently $-\divv b +d =0$. But since this would lead to a significant restriction on the class of operators that our theorems would apply, we incorporate $-\divv(b u)$ and $d u$ to the interior data $-\divv g$ and $f$ respectively. The ``new" equation  has the form
$$
\wt L u=-\divv(A \nabla u) - c\nabla u= (f+d u) - \divv( g-b u),
$$
for which it is true that $\wt L 1=0$. The price we have to pay is to impose the additional assumptions $|b|^2$ and $ |d| \in \khalf(\om)$ (for interior regularity and boundary regularity under the CDC condition).  Of course,  we require $u$ to be locally bounded as well and thus, we need to assume one of the assumptions \ref{it:intro-1}-\ref{it:intro-3}.  It is interesting to see that the proof of Theorem \ref{thm:wiener-osc} (ii),  where we are proving a Wiener type criterion for boundary regularity, is quite laborious as it requires a modification of the original argument in \cite{GiTr} (which  is not obvious without the capacity density condition) and a new way of handling the second term $\Sigma_2$ in the iteration scheme.  Moreover we have to assume a slightly stronger condition, i.e., that $|f|, |d| \in \mathcal{K}_{\textup{Dini}, \delta}(\om)$ and $|b|^2, |g|^2 \in \mathcal{K}_{\textup{Dini}, \delta/2}(\om)$ for some $\delta\in(0,1)$. To our knowledge, this is the first Wiener-type criterion for boundary regularity of solutions for equations with lower order coefficients with so general assumptions. Moreover, the interior regularity is also new  in the case that the radii of the balls we consider are not small (and thus, we do not have smallness of the norms of the coefficients). Let us comment here that one could try to prove boundary regularity following \cite{GW} or even \cite{HKM}, but in both cases, there would only be treated   solutions  of equations with no right hand-side and $b_i=d=0$, $1 \leq i \leq n$. This is because of the need of lower pointwise bounds for the Green's function or  equivalently a Harnack inequality, which, in this situation, only holds  for equations of the form $Lu = -\divv A\nabla u - c \nabla u =0$. 

\vv

Finally,  having proved all the  results above, we are in a position to construct the Green's function  using the method of Hofmann and Kim \cite{HK} along with its variant of Kang and Kim \cite{KK}, where the main ingredients are the well-posedness of the Dirichlet problem, local boundedness, Caccioppoli's inequality, and maximum principle, while, for the approximating operators, we also use the interior continuity for solutions of equations with lower order coefficients that satisfy $|b|^2$, $|c|^2$, $|d| \in \khalf(\om)$. We would not need an approximation argument if it wasn't for the lack of continuity in the general case. This creates some trouble in the proof of $G(x,y)=G^t(y,x)$ (and nowhere else),  where $G^t$ stands for the Green's function associated with $L^t$, the formal adjoint of $L$. It is important to point out that the pointwise bounds for $G$ do not hold unless local boundedness of subsolutions of $L^t u =0$ is true; in view of Example \ref{ex:c1}, an additional condition on $b+c$ is necessary. In our case, this will be $|b+c|^2 \in \khalf(\om)$ as before. Remark that, since $\om$ may have infinite Lebesgue measure, we can assume $\om=\R^n$ and construct the fundamental solution.

\vv

\noindent{\bf Related results:} An interesting result, which is very related to our work, was obtained simultaneously and independently by Georgios Sakellaris.  The first version of our paper and \cite{Sak} were uploaded on ArXiv.org the same day (9th of April 2019).  His primary goal was to construct Green's functions for elliptic operators of the form \eqref{operator} in general domains under either negativity condition that satisfy scale invariant pointwise bounds. Then, he applies them to obtain global and local boundedness  for solutions to equations with interior data in the case \eqref{negativity2}. To do this, it was required $b+c$ to be in a scale invariant space, which for the author was the Lorentz space $L^{n,1}(\om)$ (as opposed to $|b+c|^2\in\khalf(\om)$ we identified). His method is totally different than ours and is based on delicate estimates for decreasing rearrangements. In fact, he first proves the existence of Green's functions via various approximations and then  uses their properties to obtain {\it a priori estimates}; our method follows the exact opposite direction.  Our paper and \cite{Sak} are complementary  since,  apart from the major differences in the approach, the conditions $|b+c|^2\in \khalf(\om)$ and $|b+c|\in L^{n,1}(\om)$ are not comparable.  Indeed,  if $g(x)=|x|^{-1} (-\log|x|)^{-3} {\bf 1}_{B}(x)$, where $B:=B(0,\tfrac{1}{e})$,  then  $g \in L^{n,1}(B)$ such that  $g^2 \notin  \mathcal{K}_{\textup{Dini}, \alpha}(B)$ for any $\alpha>0$ (see \cite{Sak}), while, in Example \ref{ex:dini-not-lorentz}, we show that there exists a non-negative function $f \in \mathcal{K}_{\textup{Dini}, \alpha}(\R^n_+) \setminus L^{p,q}(\R^n_+)$ for any $\alpha>0$, $p>0$ and $q \in (0, \infty]$, and so $h:=\sqrt f \notin L^{n,1}(\R^n_+)$ and $h^2\in  \mathcal{K}_{\textup{Dini}, \alpha}(\R^n_+)$.  We would like to note here that Sakellaris  observed that, due to a Lorentz-Sobolev embedding theorem and density,  \eqref{negativity} or \eqref{negativity2} can be applied assuming that $b, c \in L^{n, \infty} (\om)$, $d \in L^{n/2, \infty} (\om)$. Although our original assumptions were $b, c \in L^{n} (\om)$, $d \in L^{n/2} (\om)$, and the constants depended on $\|b+c \|_{L^{n} (\om)}$ (the same dependence as in \cite{Sak}), while working the details of the case $|b+c|^2 \in \mathcal{K}(\om)$, we realized that our method extends almost unchanged when $b+c \in  L^{n, q} (\om)$, for  $q \in [n, \infty)$, which is a slight improvement compared to our previous results and the ones in \cite{Sak}. We claim no credit though for the idea to use the Lorentz-Sobolev embedding theorem, which we learned from \cite{Sak}.

Around  a year after  the last version of the present manuscript was uploaded on ArXiv.org (26th of April 2019),  Sakellaris uploaded  \cite{Sak2} on the same preprint server (28th of May 2020), where,  under the assumptions of \cite{Sak},  he obtains interior and boundary Harnack inequalities and,  under smallness assumptions on the norms of the coefficients,  he further proves interior and boundary Moser's estimates as well as interior local continuity.

	\vv
	\noindent{\bf Data availability statement:} Data sharing is not applicable to this article as no new data were created or analyzed in this study.
	
	\vv

\vvv

\noindent{\bf Acknowldegements.} We would like to thank Georgios Sakellaris for making his paper available to us and for helpful discussions. We are also grateful to him  for spotting a gap in our previous proof  of \eqref{eq:Green-symmetry}.  We would also like to thank the anonymous referee for their careful reading and their suggestions, which have contributed to improve the readability of the paper.

\section{Preliminaries}

We will write $a\lesssim b$ if there is $C>0$ so that $a\leq Cb$ and $a\lesssim_{t} b$ if the constant $C$ depends on the parameter $t$. We write $a\approx b$ to mean $a\lesssim b\lesssim a$ and define $a\approx_{t}b$ similarly. If $B_r(x)$ is a ball of radius $r$ and center $x \in \overline{\om}$, we will denote $\om_r(x)= B_r(x) \cap \om$.

\subsection{Sobolev space}

\begin{definition}
If $1\leq p<n$ and $p^*=\frac{np}{n-p}$, we define the Sobolev spaces $Y^{1,p}(\om)$ and  $W^{1,p}(\om)$ to be the space of all weakly differentiable functions $u \in L^{p^*}(\om)$ and  $L^p(\om)$ respectively,  whose weak derivatives are functions in  $L^p(\om)$. We endow these spaces with the respective norms
\begin{align*}
\| u\|_{Y^{1,p}(\om) }&= \| u \|_{ L^{p^*}(\om)} + \| \nabla u \|_{ L^{p}(\om)}\\
\| u\|_{W^{1,p}(\om) }&= \| u \|_{ L^{p}(\om)} + \| \nabla u \|_{ L^{p}(\om)}.
\end{align*}
\end{definition}
We say that $u \in Y^{1,2}_{\loc}(\om)$ (resp.  $u \in W^{1,2}_{\loc}(\om)$) if  $u \in Y^{1,2}(K)$ (resp.  $u \in W^{1,2}(K)$) for any compact $K \subset \om$.  We also define $Y^{1,p}_0(\om)$ and $W^{1,p}_0(\om)$ as the closure of $C^\infty_c(\om)$ in $Y^{1,p}(\om)$ and $W^{1,p}(\om)$ respectively, and denote their dual spaces by  $Y^{-1,p'}(\om)$ and $W^{-1,p'}(\om)$, where $p'$ is the H\"older conjugate of $p$. 

By Sobolev embedding theorem,  it is clear that $W^{1,p}_0(\om) \subset Y^{1,p}_0(\om)$, while if $\om $ has finite Lebesgue measure they are in fact equal. See, for instance, Theorem 1.56 and Corollary 1.57 in \cite{MZ}. Moreover,  $Y^{1,p}_0(\R^n)=Y^{1,p}(\R^n)$ (see e.g. Lemma 1.76 in \cite{MZ}). We will denote by $2_*= \frac{2n}{n+2}$ the dual Sobolev exponent for $p=2$.

For $u\in Y_{\loc}^{1,2}(\Omega)$ and $ \vphi \in C^\infty_c(\Omega)$, the  bilinear form which corresponds to the elliptic operator \eqref{operator} is given by
\begin{equation}\label{bilinear}
\LL(u, \vphi )= \int_{\Omega}(A\nabla u + du)\cdot \nabla \vphi  - (c\cdot \nabla u+ du)\,\vphi.
\end{equation}
which, by the embedding given in \eqref{eq:cor-Kato-emb-glob} or the one in \cite[p.6 and Lemma 2.2]{Sak}, is well-defined if \eqref{eq:elliptic3} holds. For the same reasons we can use \eqref{negativity} and \eqref{negativity2} with $Y^{1,\frac{n}{n-1}}_0(\om)$ functions. 

When we write  $Lu=f-\divv g$,  where $f\in L^{1}_\loc(\om)$ and $g \in L^1_\loc(\om;\R^n)$,  we mean that it holds  ``in the weak sense", i.e.,
$$
\LL (u,v) = \int_{\Omega} f v + g \cdot \nabla v, \quad \textup{for all} \,\, v \in C^\infty_c(\om).
$$
If  $f\in L^{2_*}(\om)$ and $g \in L^2(\om)$, we can extend it by density to $ v \in  Y^{1,2}_0(\om)$.

\vv

In the sequel we will require a notion of supremum and infimum of a function in $Y^{1,2}(\om)$ at the boundary of an open set $\om  \subset \R^n$ since such a  function is not necessarily continuous all the way to  the boundary. Let $Y$ denote either $Y^{1,2}(\om)$ or $W^{1,2}(\om)$ and $Y_0$ be either  $Y_0^{1,2}(\om)$ or $W_0^{1,2}(\om)$. 
\begin{definition}\label{def:boundary-sup-inf}
Given a function $u \in Y$, we say that $u \leq 0$ on $\d \om$ if  $u^+ \in Y_0$. If $u$ is continuous in a neighborhood of $\d \om$ then $u \leq 0$ on $\d \om$ in the Sobolev sense if $u \leq 0$ in the pointwise sense. In the same way $u \geq 0$ if $-u \leq 0$ and $u \leq w$  if $u -w \leq 0$.  We define the boundary supremum and infimum of $u$ as
$$
\sup_{\partial \Omega} u = \inf \{ k \in \R: (u-k)^+ \in Y_0\} \quad \textup{and} \quad  \inf_{\d \om}  u = - \sup_{\d \om} (- u).
$$
\end{definition}

\begin{definition}\label{def:vanish-subset}
Let $E \subset \overline{\om}$ and $u \in Y$. We say that $u \leq 0$ on $E$ if $u^+$ is  the limit in $Y$-norm of a sequence of $C^\infty_c(\overline \om \setminus E)$. Then $u \geq 0$ and $u \leq v$ can be defined naturally. Moreover, if $\om$ has finite Lebesgue measure.
$$
\sup_{E} u = \inf \{ k \in \R: u \leq k \,\, \textup{on}\,\, E\} \quad \textup{and} \quad  \inf_{E}  u = - \sup_{\d \om} (- u).
$$
\end{definition}
If $E = \d \om$ the two definitions above coincide.  

\vv

We record  some results for Sobolev functions that we will need later. Their proofs can be found in \cite{MZ} and/or in \cite{HKM} for functions in $W^{1,2}(\om)$ or $W_0^{1,2}(\om)$. Although, one can make the obvious modifications to prove them for $Y^{1,2}(\om)$  or $Y_0^{1,2}(\om)$.

\begin{lemma} \label{lem:sobolev-u-constant}
If $\om \subset \R^n$ is open and connected,  $u \in Y$  and $ \nabla u = 0$ a.e. in $\om$, then $u $ is a constant in $\om$. If we also assume $u \in Y_0$, then $u=0$.
\end{lemma}

\begin{proof}
The fact that $u$ is a constant can be found in \cite[Corollary 1.42]{MZ}, while the second part can proved by a slight modification of the proof of \cite[Lemma 1.17]{HKM}.
\end{proof}

\vv

\begin{lemma}[\cite{MZ}, Corollary 1.43] \label{lem:Sobolev-maxmin}
If $u, v \in Y$ (resp. $Y_0$) then $ \max(u, v)$ and  $ \min(u, v)$ are in $Y$ (resp. $Y_0$) and 
\begin{equation*}
\nabla \max(u, v)(x)=
\begin{dcases}
\nabla u  &,\textup{if}\,\, u\geq v\\
\nabla v &,\textup{if}\,\, v \geq u\
\end{dcases},
\end{equation*}
\begin{equation*}
\nabla \min(u, v)(x)=
\begin{dcases}
\nabla v  &,\textup{if}\,\, u\geq v\\
\nabla u &,\textup{if}\,\, v \geq u\
\end{dcases}.
\end{equation*}
In particular, $\nabla u=\nabla v$  a.e. on the set $\{x\in \om:u(x)=v(x) \}$.
\end{lemma}

\vv

\begin{theorem}[\cite{MZ}, Theorem 1.74] \label{lemma:lip0integral}
Let $\om \subset \R^n$ be an open set and let f be a Lipschitz  function such that $f(0)=0$. 
\begin{itemize}
\item[(i)] If $u \in W^{1,1}_{\loc}(\om)$ then $f \circ u  \in W^{1,1}_{\loc}(\om)$. Moreover, for a.e. $x \in \om$, we have that either
$$
\nabla (f \circ u)(x)= f'(u(x)) \nabla u(x),
$$
or 
$$
\nabla (f \circ u)(x)= \nabla u(x) =0.
$$
\item[(ii)] If $u \in Y_0$, then $f \circ  u \in Y_0$ and
$$
\| f \circ  u\|_{Y} \leq \|f'\|_{L^\infty(\om)}\| u\|_{Y}.
$$
\end{itemize}
\end{theorem}
Remark that it is necessary to have $f(0)=0$ when $\om$ is unbounded. For example, if $f(t) = 1$, then $f \circ u\not \in Y^{1,2}(\om)$. 

\vv

\begin{lemma}\label{cor:liploc}
	Let $\om \subset \R^n$ be an open set and let $f: \R \to \R$ be a function in $\Lip(\R)$. If $u \in Y$, then $f \circ \, u \in Y_{\loc}$.
\end{lemma}

\vv

\begin{lemma}[\cite{HKM}, Theorem 1.25] \label{lem:sobolev-van-trace}
Let $\om \subset \R^n$ be an open set and  $u \in Y$.
\begin{itemize}
\item[(i)] If $u$ has compact support, then $u \in Y_0$.
\item[(ii)]  If $v \in Y_0$ and $0 \leq u \leq v$ a.e.in $\om$, then $u \in Y_0$.
\item[(iii)]  If $v \in Y_0$ and $|u| \leq |v|$ a.e. in $\om \setminus K$, where $K$ is a
compact subset of $\om$, then $u \in Y_0$.
\end{itemize}
\end{lemma}

\vv

\subsection{Stummel-Kato class}
\begin{definition}\label{def:Kato}
Let $f \in L^1_{\textup{loc}}(\R^n)$,  and set
\begin{equation}\label{eq:Kato-def}
\vartheta(f, r):= \sup_{x \in \R^n}  \left( \int_{B_r(x) } \frac{|f(y)|}{|x-y|^{n-2}} \,dy\right), \quad \textup{ for }\,\, r >0,
\end{equation}
We will denote by $\vartheta_{\om}(f, r):=\vartheta(f \chi_\om, r)$, for $r >0$. We define the Stummel-Kato class $\mathcal{K}$ and its variant $\mathcal{K}'$ as follows:
\begin{align}
&\widehat{\mathcal{K}}(\om)= \{ f \in L^1_\loc(\om):\vartheta_{\om}(f,r)< \infty, \,\,\textup{for each}\,\, r>0\},\\
&\mathcal{K}(\om)= \{ f \in L^1_\loc(\om):\lim_{r \to 0} \vartheta_{\om}(f,r)=0 \,\,\textup{and}\,\,\vartheta_{\om}(f,r)< \infty, \,\,\textup{for}\,\, r>0\},\label{eq:Kato-lim}\\
&\mathcal{K}'(\om)= \{f \in L^1(\om):\lim_{r \to 0} \vartheta_{\om}(f,r)=0 \,\,\textup{and}\,\,\vartheta_{\om}(f):=\sup_{r >0} \vartheta_{\om}(f,r)< \infty\}.
\end{align}
We will write that $f\in \widehat{\mathcal{K}}_{\loc}(\om)$ (resp. $\mathcal{K}_{1,\loc}(\om)$) if $f \in \widehat{\mathcal{K}}(D)$ (resp. $\mathcal{K}(D)$) for any bounded open set $D \subset \R^{n+1}$ so that $\overline{D} \subset \om$.  If $\om$ is bounded, 
$$\vartheta_{\om}(f)= \sup_{r \in (0,2 \diam(\om))}\vartheta_{\om}(f, r),$$
 and so  $\mathcal{K}(\om)=\mathcal{K}'(\om)$.
\end{definition}

\vv

It is easy to see that, by a simple covering argument, there exists a  dimensional constant $C_{\textup{db}}>0$ so that  
\begin{equation}\label{eq:kato-doubling}
\vartheta_{\om}(f, r) \leq C_{\textup{db}} \vartheta_{\om}(f, r/2) \quad \textup{for every}\,\, r>0.
\end{equation}
Therefore, since $\vartheta_{\om}(f, r)$ is non-decreasing in $r$, there exists  $c>0$ so that
$$
c:=\frac{\ln 2}{\,C_{\textup{db}}} \leq  \frac{1}{\vartheta_\om(f,r)}\int_{r/2}^r \vartheta_\om(f,t) \frac{dt}{t} \leq  \frac{1}{\vartheta_\om(f,r)}\int_{0}^r \vartheta_\om(f,t) \frac{dt}{t}.
$$

\vvv

Let us recall that  that a function $f \in L^1_\loc(\om)$ is in  the Morrey space $\mathcal{M}^{\lambda}(\om)$, if
\begin{equation}\label{eq:def-Morrey}
\sup_{r>0} \sup_{B_r \subset \R^n} \frac{1}{r^\lambda} \int_{B_r \cap \om} |f(x)| \,dx < \infty.
\end{equation}
that a function $f \in L^1_\loc(\om)$ is in the generalized Morrey space $\mathcal{M}^{\vphi}(\om)$ with modulus $\vphi$ if 
\begin{equation}\label{eq:def-Morrey}
\sup_{r>0} \sup_{B_r \subset \R^n}\frac{1}{\vphi(r)} \frac{1}{r^{n-2}} \int_{B_r \cap \om} |f(x)| \,dx < \infty \quad \textup{and}\quad \int_0^1 \vphi(t)\,\frac{dt}{t}<\infty.
\end{equation}
By \cite[Lemma 1.1]{RZ1},  $ \mathcal{M}^{n-2+\ve}(\om) \subset \mathcal{K}(\om)$, for any $\ve \in(0,2)$,  since for every $f \in \mathcal{M}^{n-2+\ve}(\om)$, it holds that 
$$
\vartheta_{\om}(f, r)  \lec r^{n-2+\ve} \| f\|_{\mathcal{M}^{n-2+\ve}(\om)},
$$  
while, if  $f \in \mathcal{K}(\om)$ and $\int_0^1\vartheta_\om(f, t) \frac{dt}{t}<\infty$,  then it is straightforward to see that $f \in \mathcal{M}^{\vartheta_\om(f,\cdot)}(\om)$ since 
$$
\int_{B(x,r)\cap \om } |f(y)|\,dy \leq r^{n-2} \vartheta_\om(f,r).
$$

\vvv

For fixed $r >0$, we  define the space
$$
L^1_{\loc, r}(\om)=\left\{ f \in L^1_\loc(\om): \| f\|_{L^1_{\loc,r}(\om)}:=\sup_{x \in \R^n} \|f\|_{L^1(B(x, r) \cap \om)} <\infty \right\},
$$
which clearly contains $\widehat{\mathcal{K}}(\om)$.
One case see that $\| \cdot \|_{L^1_{\loc,r}(\om)}$ is a norm on $L^1_{\loc, r}(\om)$ and $\vartheta_{\om}( \cdot, r)$ is a norm on $\widehat{\mathcal{K}}(\om)$ and  $\mathcal{K}(\om)$. Analogously, $\vartheta_{\om}( \cdot)$ is a norm on $\mathcal{K}'(\om)$. In the next lemma we provide an elementary proof of the fact that those spaces are complete.

\begin{lemma}
$L^1_{\loc, r}(\om), \widehat{\mathcal{K}}(\om), \mathcal{K}(\om),$ and $\mathcal{K}'(\om)$ are Banach spaces. 
\end{lemma}

\begin{proof}
To simplify  our notation, for fixed $r>0$, we will denote 
$$X_1=L^1_{\loc, r}(\om), \quad X_2=\widehat{\mathcal{K}}(\om),  \quad X_3=\mathcal{K}(\om),\quad \textup{and} \quad X_4=\mathcal{K}'(\om).
$$
We first prove that $X_1$ is complete.  Indeed, there exists $k \in \Z$ such that $2^k< r \leq 2^{k+1}$, and let $Q \in \DD_k(\R^n)$ be the dyadic grid in $\R^n$ that consists of cubes of sidelength $2^k$ and notice that, by easy geometric considerations,
$$
\|f\|_{L^1_{\DD_k}(\om)}:=\sup_{Q \in \DD_k} \|f\|_{L^1(Q \cap \om)} \approx_n \| f\|_{X_1}.
$$
In addition, $L^1_{\DD_k}(\om)$ is the direct sum $\bigoplus_{Q \in \DD_k} X_Q$ of the Banach spaces $X_Q=L^s(Q \cap \om)$ with norm $\sup_Q \| \cdot\|_{L^1(Q \cap \om)}$. In this case, the completeness is preserved and thus,  $L^1_{\DD_k}(\om)$ is a Banach space as well, which readily implies that $X_1$ is a Banach space.

Now, we will  show that $X_2$ is a Banach space. Let 
$$
B_{X_2}=\{ f \in X_2: \|f\|_{X_2} \leq 1\}
$$
be the closed unit ball in $X_2$, and let $f_k$ be a Cauchy sequence in $X_2$. It is easy to see that $\| f \|_{X_1} \leq r^\frac{n-2}{s}\vartheta_{\om}(f,r)= r^{n-2} \|f\|_{X_2} $, and  by the completeness of $X_1$, there exists $f \in X_1$ such that $f_k \to f$ in $X_1$. By Fatou's lemma,
$$
\vartheta_{\om}(f, r) \leq  \liminf_{k \to \infty} \vartheta_{\om}(f_k, r) \leq 1,
$$
and so $f \in B_{X_2}$. Therefore, since $X_1$ is a Banach  space and the embedding of $X_2$ in $X_1$ is continuous, by \cite[Proposition 14.2.3]{Gar}, we deduce that $X_2$ is Banach as well.  It is easy to see that $X_3$ is a closed subspace of $X_2$, and thus, Banach, while, if we replace $X_1$ by $X_2$ and $X_2$ by $X_4$ in the  argument  above, we infer that $X_4$ is Banach space as well. 
\end{proof}

\vv

\subsection{Carleson-Dini Stummel-Kato class}

For any $\epsilon>0$,  we define
\begin{equation}\label{eq:theta'}
\vartheta_{\epsilon, \om}(f,r)=\vartheta_\om(f,r)+\epsilon\, r,
\end{equation}
which is strictly  increasing, continuous, and satisfies the same properties as $\vartheta_\om(f,r)$. Therefore, it is invertible  with continuous and strictly increasing inverse ${\vartheta}^{-1}_{\epsilon,\om}(f, r)$. It is clear that $\vartheta_{\epsilon,\om}(f,\cdot)$ also satisfies the doubling condition \eqref{eq:kato-doubling} with constant $\max(C_{db},2)$. 

\begin{definition}\label{def:dini-kato}
If $\alpha >0$, we say that a function $f \in \widehat{\mathcal{K}}(\om)$ is in the {\it Careslon-Dini Stummel-Kato class} $\mathcal{K}_{\textup{Dini},\alpha}(\om)$  if it satisfies
\begin{equation}\label{eq:dini-kato}
\int_0^{r} \vartheta_{\om}(f, t)^\alpha\,\frac{dt}{t} \leq C\,\vartheta_{\om}(f,r)^\alpha,
\end{equation}
for every $r>0$.  and we denote
\begin{align}\label{eq:Kato-dini}
C_{f,\om,\alpha}&:=\sup_{r>0}\frac{1}{ \vartheta_{\om}(f,r)^\alpha}  \int_0^{r} \vartheta_{\om}(f, t)^\alpha\,\frac{dt}{t}.
\end{align}
 If $\alpha=1$ then we write that $f \in \khalf(\om)$ and $C_{f,\om}:=C_{f,\om,1}$.
\end{definition}

\vv

\begin{example} \label{ex:dini}
Let $e_j = (\delta_{1j},\dots, \delta_{nj} )$, for $j \in \{1,\dots,n\}$ be the orthonormal basis of $\R^n$ and, for any $k\in\{1,2, \dots, 2^{n}\}$,  let us denote  $\vec \lambda_k=(\lambda^1_k, \dots, \lambda^n_k)\neq 0$ to be  the distinct vectors such that $\lambda^i_k=0$ or $1$ for $ i \in \{1,\dots,n\}$.   For $k\in\{1, 2,\dots,2^{n}\}$ and $j \in \bN$,  define the distinct points in $\R^n$ by
$$
x_k := \sum_{i=1}^n \lambda^i_k \,e_i \quad \textup{and}\quad y^j_k:= 2^j x_k.
$$
Define now
$$
f(x)={\bf 1}_{B(0,\frac{1}{8})}(x)+\sum_{j=1}^\infty \sum_{k=1}^{2^{n}} {\bf 1}_{B(y^j_k,2^{j-3} )}(x).
$$
Note that the balls $B(y^j_k,2^{j-3} )$ are mutually disjoint and thus,  $|f(x)| \leq 1$ for any $x \in \R^n$.  So,  for fixed $r>0$ and  every $x \in \R^n$,
\begin{align*}
\int_{B(x,r)} \frac{|f(y)|}{|x-y|^{n-2}}dy &\leq \int_{B(x,r)} \frac{1}{|x-y|^{n-2}}dy =c_nr^2,
\end{align*}
which implies that $\vartheta_{\R^n}(f, r)\lec r^2$.  For the reverse inequality,  if $r\geq 1$ remark that there exists a positive integer $j_0$ such that $2^{j_0-1} \leq r <2^{j_0}$.  Then if we set $x_1=y_1^{j_0+3}$,
$$
\vartheta_{\R^n}(f, r) \geq \int_{B(x_1, r)} \frac{|f(y)|}{|x_1-y|^{n-2}}dy \geq \int_{B(x_1,r)} \frac{1}{|x_1-y|^{n-2}}dy=c_n\,r^2.
$$
For $r < 1$, 
$$
\vartheta_{\R^n}(f, r) \geq \int_{B(0,  r/8)} \frac{|f(y)|}{|y|^{n-2}}dy= \frac{c_n}{64} \,r^2.
$$
Therefore,  $\vartheta_{\R^n}(f, r) \approx r^2$ for any $r>0$, and so, for any $\alpha >0$, it holds
$$
\int_0^r \vartheta_{\R^n}(f,t)^\alpha \frac{dt}{t} \approx\int_0^r \ t^{2\alpha-1} \,dt= r^{2\alpha} \approx \vartheta_{\R^n}(f,r)^\alpha,
$$
which implies that $f \in \mathcal{K}_{\textup{Dini},\alpha}(\R^n)$.  If $\R^n_+:=\{(x_1, \dots, x_n) \in \R^n: x_1>0\}$,  by  similar arguments,  we can show that $f \in  \mathcal{K}_{\textup{Dini},\alpha}(\R^n_+)$ for any $\alpha>0$,
\end{example}

\vv
The next lemma is easy to prove by a simple change of variables and we leave the routine details to the interested reader.
\begin{lemma}\label{lem:changevariable}
Let $\om \subset \R^n$ be an open set and $f \in \khalf(\om)$.  For $\rho>0$,  set $f_\rho(x)=\rho\, f(\rho\,x)$ for any $x \in D_\rho:= \rho^{-1}\om$. Then the following hold:
\begin{itemize}
\item[(i)] If $\lambda>0$, then $\vartheta_{\om}(\lambda\,f,t)=\lambda \vartheta_{\om}(f, t)$,  for any $t>0$ and $C_{\lambda\, f,\om} = C_{f,\om}$.
\item[(ii)] $ \vartheta_{D_\rho}(f_\rho,t)=\vartheta_{\om}(f, \rho\,t)$, for any $t>0$. 
\item[(iii)]$C_{f_\rho,D_\rho} = C_{f,\om}$.
\end{itemize}
Moreover, if $g \in \khalf(\om)$,  and we set $g_\rho(x)=\rho\, g(\rho\,x)$, $V=|f|+|g|$,  and $V_\rho=|f_\rho|+|g_\rho|$,  then $V \in \khalf(\om)$ and
$$
C_{V_\rho,D_\rho} = C_{V,\om} \leq 2C_{f,\om}+2C_{g,\om}.
$$
\end{lemma}

\vvv

\subsection{Sobolev embedding and Interpolation inequalities}

The following considerations can be found in \cite[p.416]{Ku} and are based on an inequality proved by Simon in \cite[p.455]{Sim}. Assume that $f \in \mathcal{K}(\om)$ and let
\begin{equation}\label{eq:molifier}
\psi \in C^\infty_c(\R^n),\quad 0 \leq \psi \leq 1,\quad \psi = 0\,\, \textup{in}\,\, \R^n \setminus B(0,1), \quad\textup{and} \quad \int \psi =1.
\end{equation}
For $\delta>0$, set $\psi_\delta(x)= \delta^{-n} \psi(\delta^{-1} x)$ and define 
\begin{equation}\label{def:mol-f-delta}
f_\delta = f \ast \psi_\delta.
\end{equation}
 Then, if $G \subset \om$, $r>0$ and $0<\delta \leq r$, we have
\begin{align}\label{eq:Kato-molifier}
\vartheta_G((f {\bf 1}_G)_\delta, r) &\leq \vartheta((f {\bf 1}_G)_\delta, r)  \leq  \vartheta(f {\bf 1}_G, r) +  \vartheta(f {\bf 1}_G, \delta)  \\
&\leq 2 \vartheta(f {\bf 1}_G, r) \leq 2 \vartheta(f, r). \notag
\end{align}
Thus, for a ball $B_r$ so that $B_{2r} \subset \om$ and $0<\delta<r$, we also obtain 
\begin{equation}\label{eq:Kato-molifier-balls}
\vartheta_{B_r}(f_\delta, r) \leq \vartheta_{B_r}((f {\bf 1}_{B_{2r}})_\delta, r) \leq 2 \vartheta_{B_{2r}}(f, r).
\end{equation}
Moreover, if $|g|^2 \in \mathcal{K}(\om)$,
\begin{equation}\label{eq:Kato-molifier-b2}
\vartheta(|g_\delta|^2, r) \leq \vartheta(|g|^2, r) +  \vartheta(|g|^2, \delta) \leq 2 \vartheta( |g|^2, r).
\end{equation}
It is useful to remark that if 
\begin{equation}\label{def:Omega-delta}
 \om_\delta= \{x \in \om: \dist(x, \om^c)>\delta\} \cap B(0, \delta^{-1}),
\end{equation}
  then $\vartheta((f {\bf 1}_{\om_\delta})_\delta, r)=\vartheta_\om((f {\bf 1}_{\om_\delta})_\delta, r)$.
  
  \vv
  
  In the next lemma we use an argument from \cite{ViZa}.
 \begin{lemma}\label{lem:kato-approx-0} 
If $f \in \mathcal{K}(\om)$ and $\rho>0$, it holds  that $\vartheta_\om( (f {\bf 1}_{\om_\delta})_\delta- f), \rho) \to 0$, as $\delta \to 0$. If $f \in \mathcal{K}'(\om)$, then  $\vartheta_\om( (f {\bf 1}_{\om_\delta})_\delta- f) \to 0$, as $\delta \to 0$.
  \end{lemma}

  \begin{proof}
 Fix $\rho>0$ and note that by \eqref{eq:Kato-lim}, for  $\ve>0$, we can find  $r_0<\rho$, so that $\vartheta_\om(f ,r_0)< \frac{\ve}{6}$. Note that by \eqref{eq:Kato-molifier},  for $0<\delta<r_0$,  we have that $ \vartheta_\om( (f {\bf 1}_{\om_\delta})_\delta- f, r_0) \leq 3  \vartheta_\om( f, r_0)$.  Thus,
  \begin{align*}
\vartheta_\om( (f {\bf 1}_{\om_\delta})_\delta- f, \rho) &\leq \vartheta_\om( (f {\bf 1}_{\om_\delta})_\delta - f, r_0) \\ 
&+ \sup_{x \in \R^n} \int_{(B(x,r) \setminus B(x,r_0)) \cap \om} \frac{|(f {\bf 1}_{\om_\delta})_\delta (y) -f(y)|}{|x-y|^{n-2}} \,dy \\
 &\leq \ve/2+{r_0^{2-n}} \sup_{x \in \R^n} \int_{B(x,\rho) \cap \om} |(f {\bf 1}_{\om_\delta})_\delta (y) -f(y)|\,dy\\
 &=:\ve/2+{r_0^{2-n}} I_\rho.
  \end{align*}
  As $ \vartheta_\om( (f {\bf 1}_{\om_\delta})_\delta- f, \rho) \leq 3  \vartheta_\om( f, \rho)<\infty$, for $0<\delta<\rho$,  there exists $x_0\in \R^n$ such that 
  $$
  I_\rho \leq 2  \int_{B(x_0,\rho) \cap \om} |(f {\bf 1}_{\om_\delta})_\delta (y) -f(y)|\,dy.
  $$
Now, using that $ (f {\bf 1}_{\om_\delta})_\delta \to f$ in $L^1_\loc(\om)$, there exists  $\delta>0$ such that  $\delta<\min(r_0,\rho)$ and 
$$
\int_{B(x_0,\rho) \cap \om} |(f {\bf 1}_{\om_\delta})_\delta (y) -f(y)|\,dy< 4^{-1}{r_0^{n-2}} \ve.
$$
Collecting all the estimates we obtain that $\vartheta_\om( (f {\bf 1}_{\om_\delta})_\delta- f, \rho)  <\ve$. The proof  for $f \in \mathcal{K}'(\om)$ is the same.
 \end{proof}

\vv

\begin{lemma}\label{lem:Kato-emb-loc}
If $f \in \mathcal{K}(B_r)$, there exists a constant $c_1>0$ depending only on $n$ such that for any $r>0$ and  $u \in W^{1,2}(B_r)$, it holds 
\begin{equation}\label{eq:Kato-emb-loc}
\int_{B_r} |u|^2 f \leq c_1\,\vartheta_{B_r}(f, r) \left( \|\nabla u\|^2_{ L^{2}(B_r)}+ \frac{1}{r^2} \| u\|_{L^2(B_r)}^2 \right).
\end{equation}
\end{lemma}
\begin{proof}
This inequality can be found in the proof of Lemma 2.1 in \cite{Ku} (display (12), p. 416). It is stated with slightly different assumptions but an inspection of the proof reveals that \eqref{eq:Kato-emb-loc} is also true. For a similar inequality see  Lemma 7.3 in \cite{Sch}.
\end{proof}

\vv

Note that if we set $f=f_\delta$ in \eqref{eq:Kato-emb-loc} and use \eqref{eq:Kato-molifier-balls}, we can see that for $0<\delta<r$, 
\begin{equation}\label{eq:Kato-emb-loc-molifier}
\int_{B_r} |u|^2 f_\delta \leq 2 c_1\,\vartheta_{B_{2r}}(f, r) \left( \|\nabla u\|^2_{ L^{2}(B_r)}+ \frac{1}{r^2} \| u\|_{L^2(B_r)}^2 \right),
\end{equation}
where $c_1$ is independent of $\delta$.
\vv


\begin{lemma}\label{lem:Kato-emb-glob}
If $f \in \mathcal{K}(\R^n)$, then, there exists a constant $c_2>0$ depending only on $n$ such that for any $\ve>0$ and  $u \in W^{1,2}(\R^n)$, it holds 
\begin{equation}\label{eq:Kato-emb-glob}
\int_{\R^n} |u|^2 f \leq \ve \, \|\nabla u\|^2_{ L^{2}(\R^n)}+ \frac{\ve}{{\vartheta}^{-1}_{\epsilon,\R^n}(f, c_2^{-1} \ve)^2}\, \| u\|_{L^2(\R^n)}^2 .
\end{equation}
\end{lemma}

\begin{proof}
We cover $\R^n$ with balls $B(z_j, r)$, with center all the points $z_j$ so that $nz_j/r$ have integer coordinates. It is clear that each point $x \in \R^n$ is contained in at most $N$ balls $B(z_j, 2r)$, where $N$ is a positive constant depending only on the dimension $n$. Fix $\ve>0$ and choose $r>0$ small enough so that $\vartheta_{\epsilon,\R^n}(f,r)=(N c_1)^{-1}\ve$, where $c_1$ is the constant in \eqref{eq:Kato-emb-loc}. Thus, using $\vartheta_{B_r}(f, r)  \leq {\vartheta}_{\epsilon,\R^n}(f, r) $ and  \eqref{eq:Kato-emb-loc}, we have that 
\begin{align*}
\int_{\R^n} |u|^2 f \leq \sum_{j=1}^\infty \int_{B(z_j, r)} |u|^2 f  &\leq  \sum_{j=1}^\infty \frac{\ve}{N} \left( \int_{B(z_j, r)} |\nabla u|^2+\frac{1}{r^2}\int_{B(z_j, r)} | u|^2 \right)\\
& \leq \ve\int_{\R^n} |\nabla u|^2+ \frac{\ve}{r^2}\int_{\R^n} | u|^2,
\end{align*}
which, if we set $c_2=N c_1$, implies \eqref{eq:Kato-emb-glob}.  
\end{proof}

\vv

An immediate corollary of the latter theorem, which will be used in Section \ref{sec:local-est}, is the following:
\begin{corollary}\label{cor:Kato-emb-glob}
If $f \in \mathcal{K}(\om)$, then, there exists a constant $c_2>0$ depending only on $n$ such that for any $\ve>0$ and  $u \in W_0^{1,2}(\om)$, it holds 
\begin{equation}\label{eq:cor-Kato-emb-glob}
\int_{\om} |u|^2 f \leq \ve  \|\nabla u\|^2_{ L^{2}(\om)}+ \frac{\ve}{{\vartheta}^{-1}_{\epsilon,\om}(f, c_2^{-1} \ve)^2} \| u\|_{L^2(\om)}^2.
\end{equation}
\end{corollary}

  \vv

\begin{remark}\label{rem:kato-emb-molifier}
In view of \eqref{eq:Kato-emb-loc-molifier}, it is easy to see that \eqref{eq:Kato-emb-glob} and \eqref{eq:cor-Kato-emb-glob} still hold if we replace $f$ by $f_\delta$ on the left hand-side and keep the same term on the right hand-side. 
\end{remark}

\vv

The remark  above, combined with \eqref{eq:Kato-molifier-b2} and (the proofs of) Lemmas \ref{lem:Kato-emb-loc} and \ref{lem:Kato-emb-glob}, and Corollary \ref{cor:Kato-emb-glob}, leads to the following corollary which will be crucial in an approximation argument we will need later.
\begin{corollary}\label{cor:Kato-emb-glob-moli}
If $|g|^2 \in \mathcal{K}(\om)$, then there exists a constant $c'_2>0$ depending only on $n$ such that for any $\ve>0$ and  $u \in W_0^{1,2}(\om)$ it holds 
\begin{equation}\label{eq:cor-Kato-emb-glob-moli}
\int_{\om} |u|^2 |(g {\bf 1}_{\om_\delta})_\delta|^2 \leq \ve  \|\nabla u\|^2_{ L^{2}(\om)}+ \frac{\ve}{{\vartheta}^{-1}_{\epsilon,\om}(|g|^2, {c'_2}^{-1}\ve)^2} \| u\|_{L^2(\om)}^2.
\end{equation}
\end{corollary}

\vv

\begin{lemma}\label{lem:Kato-emb-grad-Br}
If $f $ is supported in a ball $B_r$ and $f \in \mathcal{K}(\R^n)$, there exists a constant $C_s'>0$ depending only on $n$ such that, if  $u \in Y^{1,2}(\R^n)$, it holds 
\begin{equation}\label{eq:Kato-emb-grad-Br}
\int_{\R^n} |u|^2 f \leq C_s'\,\vartheta_{\R^n}(f, r)  \|\nabla u\|^2_{ L^{2}(\R^n)}.
\end{equation}
\end{lemma}
\begin{proof}
This follows from the combination of  \cite[Theorem 1.79]{MZ} and the proof of \cite[Lemma 3]{Za}. 
\end{proof}

\vv

\begin{lemma}\label{lem:Kato-emb-grad-om}
If $f \in \mathcal{K}'(\om)$, there exists a constant $C_s'>0$ depending only on $n$ such that, if  $u \in Y_0^{1,2}(\om)$, it holds 
\begin{equation}\label{eq:Kato-emb-grad-om}
\int_{\om} |u|^2 f \leq C_s'\,\vartheta_{\om}(f)  \|\nabla u\|^2_{ L^{2}(\om)}.
\end{equation}
\end{lemma}
\begin{proof}
Let $B_k:= B(0,k)$ and $f_k=f {\bf 1}_{B_k}$. Then, since $|f_k| \leq |f|$ and $f_k \to f$ pointwisely, by Lemma \ref{lem:Kato-emb-grad-Br}, we have that
$$
\int_{\om} |u|^2 f_k \leq C_s'\,\vartheta_{\om}(f, k)  \|\nabla u\|^2_{ L^{2}(\om)} \leq C_s' \,\vartheta_{\om}(f)  \|\nabla u\|^2_{ L^{2}(\om)},
$$
which, by  the dominated convergence theorem, concludes the proof of \eqref{eq:Kato-emb-grad-om}.
\end{proof}

  \vv  

\subsection{Lorentz spaces}

\begin{definition}\label{def:Lor-space}
If $f$ is a measurable function we define the distribution function
$$
d_{f, \om}(t)=|\{ x \in \om: |f(x)|>t\}|, \quad t>0,
$$
and its decreasing rearrangement by
$$
f^*(t) = \inf\{s >0: d_{f, \om}(t) \leq s\}.
$$
If $p \in (0, \infty)$ and $q \in (0, \infty]$, we can define the Lorentz semi-norm
\begin{equation}\label{eq:Lorentz-seminorm}
\| f\|_{L^{p,q}(\om)}=
\begin{dcases}
p^\frac{1}{q}\left( \int_0^\infty \left( t \,d_{f, \om}(t)^\frac{1}{p} \right)^q\,\frac{dt}{t} \right)^{\frac{1}{q}} &,\textup{if}\,\, q <\infty\\
\sup_{t>0} t \,d_{f, \om}(t)^\frac{1}{p}  &,\textup{if}\,\, q= \infty.
\end{dcases},
\end{equation}
If $\|f\|_{L^{p,q}(\om)}<\infty$,  we will say that $f$ is in the Lorentz space $(p,q)$ and write $f \in L^{p,q}(\om)$.  This is quasi-norm and $({L^{p,q}(\om)},\|\cdot\|_{L^{p,q}(\om)})$ is a quasi-Banach space.
\end{definition}

We can also define 
$$
\| f\|_{L^{(p,q)}(\om)}=
\begin{dcases}
\left( \int_0^\infty \left( t^\frac{1}{p} \,f^{**}(t) \right)^q \right)^{\frac{1}{q}}  \,\frac{dt}{t}&,\textup{if}\,\, q <\infty\\
\sup_{t>0} t^\frac{1}{p} \,f^{**}(t)  &,\textup{if}\,\, q= \infty.
\end{dcases},
$$
which, for $p \in (1, \infty)$ and $q \in [1, \infty]$, is a norm and  it holds that
\begin{equation}\label{Lorentz-equiv-norm}
\| f\|_{L^{p,q}(\om)}\leq \| f\|_{L^{(p,q)}(\om)} \leq \frac{p}{p-1} \| f\|_{L^{p,q}(\om)}.
\end{equation}
If we equip $L^{p,q}(\om)$ with this norm, it becomes a Banach space (see \cite[Lemma 4.5 and Theorem 4.6]{BeSh}).  We will write $f \in L_\loc^{p,q}(\om)$  if $f \in L^{p,q}(\om')$ for any bounded open set $\om' \subset \om$.

We record that 
\begin{enumerate}
\item If $0 < p, r \leq \infty$ and  $0<q \leq \infty$,
$$
\| |f|^r \|_{L^{p,q}(\om)} =\| f \|^r_{L^{pr,qr}(\om)};
$$
\item If $0 < p \leq \infty$ and  $0<q_2<q_1 \leq \infty$,
\begin{equation}\label{eq:Lor-inclusion}
\| f \|_{L^{p,q_1}(\om) }\lesssim_{p,q_1,q_2} \| f \|_{L^{p,q_2}(\om)};
\end{equation}
\item If $0<p,q,r \leq \infty$, $0<s_1,s_2 \leq \infty$, $1/p + 1/q=1/r$, and $1/s_1 + 1/s_2=1/s$,
\begin{equation}\label{eq:Lor-Holder}
\| f g  \|_{L^{r,s}(\om) }\lesssim_{p,q, s_1,s_2} \| f \|_{L^{p,s_1}(\om)}\| g \|_{L^{q,s_2}(\om)}.
\end{equation}
\end{enumerate}
We refer to \cite[Chapter 4]{BeSh} and \cite[Chapter 1]{Gra} for the proofs.
It is worth noting that
$$
L^{\frac{n}{2},1}(\om) \subset \mathcal{K}'(\om),
$$
while, for $n \geq 3$, $\mathcal{K}(\om)$ and $L^{\frac{n}{2},q}(\om)$, $q \geq n$, are not comparable.

\vv

\begin{example}\label{ex:dini-not-lorentz}
If $f$ is the function of Example \ref{ex:dini}, then it is easy to see that 
\begin{equation*}
d_{f, \R^n}(t)=
\begin{cases}
0 &,\textup{if}\,\, t>1\\
+\infty &,\textup{if}\,\, t\in (0,1].
\end{cases}
\end{equation*}
and,  by definition,  for every $p>0$ and $q \in (1, \infty)$,
$$
\|f\|^q_{L^{p,q}(\R^n)} =p \int_0^1 d_{f, \R^n}(t)^\frac{q}{p} t^{q-1}\,dt \geq 2^{q-1}p \int_{1/2}^1 d_{f, \R^n}(t)^\frac{q}{p} \,dt =+ \infty.,
$$
while for every $p>0$ and $q\in  (0,1]$,
$$
\|f\|^q_{L^{p,q}(\R^n)} =p \int_0^1 d_{f, \R^n}(t)^\frac{q}{p} t^{q-1}\,dt \geq p \int_{0}^1 d_{f, \R^n}(t)^\frac{q}{p} \,dt =+ \infty.
$$
It is clear that $\|f\|_{L^{p,\infty}(\R^n)} =+\infty$. Therefore,  $f \in  \mathcal{K}_{\textup{Dini},\alpha}(\R^n)\setminus L^{p,q}(\R^n)$ for any $\alpha>0$, $p>0$,  and $q \in(0,\infty]$.  Similarly,  one can show that $f \in  \mathcal{K}_{\textup{Dini},\alpha}(\R^n_+)\setminus L^{p,q}(\R^n_+)$ for any $\alpha>0$, $p>0$,  and $q \in(0,\infty]$. 
\end{example}

\vv 

\begin{definition}\label{def:abs.cont.}
Let $\{E_k\}_{k=1}^\infty$ be a sequence of measurable subsets of $\om$. We will write $E_k \to \emptyset$ a.e. if ${\bf 1}_{E_k} \to 0$  a.e. in $\om$, which is equivalent to $|\limsup_{k \to \infty} E_k|=0$.

We will say that a function $f$ in a Banach  function space $X$ (see \cite[Definition 6.5]{Saw}) has  {\it absolutely continuous norm}  in $X$ if $\|f {\bf 1}_{E_k} \|_X \to 0$ for every sequence $\{E_k\}_{k \geq 1}$ such that $E_k \to \emptyset$ a.e. The set of all functions in $X$ of absolutely continuous norm is denoted by $X_a$. If $X_a=X$, then the space itself is said to have {\it absolutely continuous norm}.  In this case, simple functions supported on a set of finite Lebesgue measure are dense in $X$.
\end{definition}

 Record that $L^{p,q}(\om)$, for $p\in (1, \infty)$ and $q \in [1, \infty)$, is a Banach function space (see \cite[p.219, Theorem 4.6]{BeSh}).

  \vv
  
\begin{lemma} \label{lem:abs.cont. norms}
Let $f \in X$ where $X=\mathcal{K}'(\om)\,\,\textup{or}\,\, L^{p,q}(\om)$, $1 <p< \infty$ and $1 \leq  q < \infty$. If $\| \cdot \|_X$ stands for either  $\vartheta_\om(\cdot)$ or $\| \cdot\|_{L^{p,q}(\om)}$, then $X$ has absolutely continuous norm. In fact, for every $\ve>0$, there exists $\delta>0$ such that
\begin{equation}\label{eq:abs.cont. norms}
\textup{if}\,\, E\subset \om\,\, \textup{with}\,\, |E|< \delta,\,\, \textup{then}\quad  \| f {\bf 1}_E\|_X<\ve.
\end{equation}
\end{lemma}
\begin{proof}
For $\mathcal{K}'(\om)$ this was proved in \cite[Lemma 2.2]{ViZa}, while for $L^{p,q}(\om)$ it follows from \cite[p. 23, Corollary  4.3]{BeSh} and \cite[p. 221, Corollary  4.8]{BeSh}. 
\end{proof}

  \vv

\begin{lemma}[\cite{Cos1}, Theorem V4]\label{Lorentz-molifier}
 Let $f \in L^{p,q}(\om)$, with $p \in (1, \infty)$ and $q \in [1, \infty)$, and for $\delta>0$, let $\om_\delta$  be as in \eqref{def:Omega-delta}. 	Then, it holds that 
$$
\| (f {\bf 1}_{\om_\delta})_\delta \|_{ L^{p,q}(\om)} \leq C_{p,q} \, \| f \|_{ L^{p,q}(\om)}\quad \textup{and}\quad \| (f {\bf 1}_{\om_\delta})_\delta -f \|_{ L^{p,q}(\om)} \to 0.
$$
\end{lemma}

\vv

In the following definitions and lemmas we follow \cite{Sak}.

\begin{definition} We define $ Y^{1,(p,q)}_0(\om)$, for $1< p <n$ and $1 \leq q \leq \infty$, to be the closure of $C^\infty_c(\om)$ under the semi-norm
 $$
 \| u\|_{Y^{1,(p,q)}_0(\om)} =  \| u\|_{L^{\frac{np}{n-p},q}(\om)} + \| \nabla u\|_{L^{p,q}(\om)}.
 $$
\end{definition}

\begin{lemma}\label{lem:Lor-Sobolev-emb-p}
If $u \in Y^{1,(p,q)}_0(\om)$, there exists a constant $C_s>0$ depending on $n$ such that 
\begin{equation}\label{eq:Lor-Sobolev-emb-p}
\| u \|_{L^{\frac{np}{n-p},q}(\om)} \leq C_s \| \nabla u\|_{L^{p,q}(\om)}.
\end{equation}
If $u \in Y^{1,2}_0(\om)$, the same is true for $p=q=2$.
\end{lemma}
\begin{proof} The proof of the first part can be found in \cite[Theorem 4.2(i)]{Cos2} and of the second one in  \cite[Lemma 2.2]{Sak}.
\end{proof}

  \vv

\begin{lemma}\label{lem:Lor-Sob-uw}
 If $u, w \in Y_0^{1,2}(\om)$, then  $u w  \in Y^{1,(\frac{n}{n-1},1)}_0(\om)$ and, in particular, it holds that
\begin{equation}\label{eq:Lor-Sob-uw}
\| u w  \|_{L^{\frac{n}{n-2},1}(\om)} \leq 2 C_s^2 \| \nabla u \|_{L^2(\om)}\| \nabla w \|_{L^2(\om)}.
\end{equation} 
\end{lemma}
\begin{proof} 
Here we follow the scheme of the proof of \cite[Lemma 2.2]{Sak}.
Since both $u $ and $w $ belong to $Y^{1,2}_0(\om)$, we can use \eqref{eq:Lor-Sobolev-emb-p} and \eqref{eq:Lor-Holder} to deduce that
\begin{equation}\label{eq:Lor-sob-w-nabla w}
\| w \nabla u \|_{L^{\frac{n}{n-1},1}(\om)} \leq \| \nabla u \|_{L^2(\om)} \| w \|_{L^{{\frac{2n}{n-2}},2}(\om)}\leq C_s \| \nabla u \|_{L^2(\om)}\| \nabla w \|_{L^2(\om)}.
\end{equation}
The analogous estimate holds if we switch the roles of $w$ and $u$.
Since $u, w \in Y^{1,2}_0(\om)$, there exist sequences  $\{ \phi_k\}_{k\geq 1}, \{ \psi_k\}_{k\geq 1}  \subset C^\infty_c(\om)$ such that $\phi_k \to u $ and $\psi_k \to w $ in $Y^{1,2}_0(\om)$. By Lemma  \ref{lem:Lor-Sobolev-emb-p}, we can find a subsequence of $ \phi_k \psi_k$ that is weakly-* convergent in  $Y^{1,(\frac{n}{n-1},1)}_0(\om)$ to some $v \in Y^{1,(\frac{n}{n-1},1)}_0(\om)$. But since $v \in L^{\frac{n}{n-2},1}(\om) \subset L^{\frac{n}{n-2}}(\om)$, it holds that $v= u w $ in $L^{\frac{n}{n-2},1}(\om)$. Thus,
\begin{align*}
\| u w  \|_{L^{\frac{n}{n-2},1}(\om)} &\leq \liminf_{k \to \infty} \| \phi_k \psi_k\|_{L^{\frac{n}{n-2},1}(\om)}\\
&  \leq C_s  \liminf_{k \to \infty} \|\nabla ( \phi_k \psi_k)\|_{L^{\frac{n}{n-1},1}(\om)} \leq 2C_s^2 \| \nabla u \|_{L^2(\om)}\| \nabla w \|_{L^2(\om)},
\end{align*}
where in the last step we used the same argument as in \eqref{eq:Lor-sob-w-nabla w} and the strong convergence of $ \phi_k$ and $ \psi_k $ in $Y^{1,2}_0(\om)$.
\end{proof}

  \vv

\begin{lemma}[\bf Embedding inequality]\label{lem:Lor-Sobolev-emb} 
Let $ h \in L^{n, q}(\om)$, for  $q \in [n , \infty] $,  $u \in Y^{1,2}(\om)$ and  $w \in Y_0^{1,2}(\om)$. Then if $D \subset \om$ is a Borel set, there exists a  constant $C_{s,q}>0$ (depending only on $n$ and $q$) such that
\begin{equation}\label{eq:Lorentz-Sobolev-emb}
\left|\int_D h \nabla u w \right| \leq C_{s,q}  \|h\|_{L^{n, q}(D)} \|\nabla u\|_{L^2(D)} \|\nabla w\|_{L^2(\om)}.
\end{equation}
\end{lemma}
\begin{proof}
This follows from \eqref{eq:Lor-Holder}, \eqref{eq:Lor-Sobolev-emb-p}, and \eqref{eq:Lor-inclusion}.
\end{proof}

  \vv

\begin{remark}\label{rem:bd-cd}
In \cite[eq. (2.9)]{Sak}, it was observed  that if $b, c \in L^{n, \infty}(\om)$ and $d \in L^{\frac{n}{2}, \infty}(\om)$, \eqref{negativity} and \eqref{negativity2} hold if $\vphi \in  Y^{1,(\frac{n}{n-1},1)}_0(\om)$.
\end{remark}

\vv

\subsection{Two auxiliary lemmas}

The next lemma was stated  in  \cite{RZ}.    The  proof as written in  \cite{RZ} is not totally correct since $\hm$ is not absolutely continuous. We overcome this obstacle by an approximation argument. 
\begin{lemma}\label{lem:modulus-dini}
Let $\hm:(0, \infty) \to (0,\infty)$ be a strictly increasing and continuous function such that $\lim_{r \to 0^+} \hm(r) =0$ and $\lim_{r \to \infty} \hm(r) = +\infty$.  Let $\tau \in (0.1)$, $c>0$, and $q \geq 1$, and set 
\begin{equation}\label{eq:Dini-a_i}
b_k= c\,\tau^{kq} \quad \textup{and}\quad a_k=  b_k^{1/q} \log \hm^{-1}( b_k).
\end{equation}
Then it holds
\begin{equation}\label{eq:Dini-sum}
-\sum_{k=1}^\infty a_k  \leq \frac{1}{1-\tau}\int_0^{\hm^{-1}(c)} \hm(t)^{1/q} \frac{dt}{t}.
\end{equation}
\end{lemma}

\begin{proof}
Note that $\hm$ is one-to-one and its inverse $\hm^{-1}$  is also strictly increasing and continuous.  If we define $\hm_\delta$ as in \eqref{def:mol-f-delta} in $\R$, then $\hm_\delta$ is  strictly increasing and smooth satisfying 
$$
\lim_{t \to 0} \hm_\delta(t)=\int \psi_\delta(-s)\, \hm(s)\,ds=:\alpha_\delta \in [0, \hm(\delta)].
$$
Therefore,  $\hm_\delta^{-1}$ is also strictly increasing  and smooth on $\ran(\hm_\delta)$, the range of $\hm_\delta$. 
As $\lim_{\delta \to 0} \hm_{\delta}(t) =\hm(t)$ locally uniformly in $(0,\infty)$\footnote{Just pointwise convergence is enough here.},  it is not hard to show that $\lim_{\delta \to 0} \hm^{-1}_{\delta}(r) =\hm^{-1}(r)$ for all $r \in \ran(\hm)=(0,\infty)$.  Indeed, let $\ve>0$ and $r >0$. Then, by the continuity of $\hm$ in $(0,\infty)$, there exists $ \delta'=\delta'(\ve,r)>0$ such that
$$
|\hm^{-1}(r+\delta') - \hm^{-1}(r)|<\ve \quad \textup{and}\quad |\hm^{-1}(r-\delta') - \hm^{-1}(r)|<\ve.
$$
For any sequence  $\{\delta_n\}_{n=1}^\infty$ such that $\delta_n \to 0$ as $n \to \infty$,  it holds that  $\lim_{n \to \infty} \hm_{\delta_n} = \hm$,  and so there exists $n_0>0$ such that for every $n>n_0$,
$$
| \hm_{\delta_n} (\hm^{-1}(r+\delta') ) - (r+\delta') | <\delta' \quad \textup{and}\quad | \hm_{\delta_n} (\hm^{-1}(r-\delta') ) - (r-\delta') | <\delta'.
$$
Therefore,
$$
\hm_{\delta_n} (\hm^{-1}(r+\delta') ) >r  \quad \textup{and}\quad \hm_{\delta_n} (\hm^{-1}(r-\delta') ) < r,
$$
which, using that $\hm_{\delta_n}^{-1}$ is strictly increasing in $(0,\infty)$,  implies that
$$
\hm_{\delta_n}^{-1}(r) \in [\hm^{-1}(r-\delta'), \hm^{-1}(r+\delta')]
$$
and thus,  $|\hm_{\delta_n}^{-1}(r) -\hm^{-1}(r)| < \ve$. This concludes the proof of $\lim_{\delta \to 0} \hm_\delta^{-1}=\hm$ pointwisely.

For any fixed positive $N \in \bN$,  it holds that
\begin{align*}
\sum_{k=0}^N (\tau a_k - a_{k+1}) &= \sum_{k=0}^N  b_{k+1}^{1/q} (\log \hm^{-1}(b_k) - \log \hm^{-1}(b_{k+1}))\\
&=\lim_{\delta \to 0} \sum_{k=0}^N  b_{k+1}^{1/q} (\log \hm_\delta^{-1}(b_k) - \log \hm_\delta^{-1}(b_{k+1}))\\
&= \lim_{\delta \to 0} \sum_{k=0}^N  b_{k+1}^{1/q}  \int_{b_{k+1}}^{b_k} \frac{1}{\hm_\delta^{-1}(t)} \frac{1}{\hm_\delta'(\hm_\delta^{-1}(t))} \,dt\\
&\leq  \lim_{\delta \to 0}\sum_{k=0}^N \int_{b_{k+1}}^{b_k} \frac{ t^{1/q} }{\hm_\delta^{-1}(t)} \frac{1}{\hm_\delta'(\hm_\delta^{-1}(t))} \,dt\\
&= \lim_{\delta \to 0}\int_{b_{N+1}}^{c} \frac{ t^{1/q} }{\hm_\delta^{-1}(t)} \frac{1}{\hm_\delta'(\hm_\delta^{-1}(t))} \,dt =  \lim_{\delta \to 0} \int_{\hm_\delta^{-1}(b_{N+1})}^{\hm_\delta^{-1}(c)} \hm_\delta(t)^{1/q} \frac{dt}{t}.
\end{align*}

Remark that $\hm^{-1}(b_{N+1})>0$.   For $\eta>0$,  there exists $\delta_0=\delta(\eta, c, b_{N+1})>0$ such that  for every $\delta <\delta_0$, 
$$
|\hm_\delta^{-1}(c)-\hm^{-1}(c)|<\eta \quad  \textup{and} \quad |\hm_\delta^{-1}(b_{N+1})-\hm^{-1}(b_{N+1})|<\eta.
$$
Therefore, for  $\delta<\delta_0$, 
\begin{align*}
\int_{\hm_\delta^{-1}(b_{N+1})}^{\hm_\delta^{-1}(c)} \hm_\delta(t)^{1/q} \frac{dt}{t} \leq \int_{\hm^{-1}(b_{N+1})-\eta}^{\hm^{-1}(c)+\eta} \hm_\delta(t)^{1/q} \frac{dt}{t}.
\end{align*}
Now, by the local uniform convergence of $\hm_\delta$,  we can find $0<\delta_1 \leq \delta_0$ such that for every $\delta<\delta_1$,  it holds that $|\hm_\delta(t) -\hm(t)|<\eta$ for every $t  \in [\hm^{-1}(b_{N+1})-\eta, \hm^{-1}(c)+\eta]$. Therefore,  for  $\delta<\delta_1$, we infer that
\begin{align*}
\sum_{k=0}^N  b_{k+1}^{1/q} (\log \hm_\delta^{-1}(b_k) &- \log \hm_\delta^{-1}(b_{k+1})) \leq \int_{\hm^{-1}(b_{N+1})-\eta}^{\hm^{-1}(c)+\eta} \hm_\delta(t)^{1/q} \frac{dt}{t}\\
& \leq \eta^{1/q}\ \,\log\frac{\hm^{-1}(c) +\eta}{\hm^{-1}(b_{N+1}) -\eta} + \int_{\hm^{-1}(b_{N+1})-\eta}^{\hm^{-1}(c)+\eta} \hm(t)^{1/q} \frac{dt}{t},
\end{align*}
which, by taking $\delta \to 0$, implies that
$$
\sum_{k=0}^N (\tau a_k - a_{k+1})  \leq   \eta^{1/q}\ \,\log\frac{\hm^{-1}(c) +\eta}{\hm^{-1}(b_{N+1}) -\eta}  + \int_{\hm^{-1}(b_{N+1})-\eta}^{\hm^{-1}(c)+\eta} \hm(t)^{1/q} \frac{dt}{t}.
$$
Since $\eta$ is arbitrary, we may take $\eta \to 0$ and deduce that 
$$
\sum_{k=0}^N (\tau a_k - a_{k+1})  \leq \int_{\hm^{-1}(b_{N+1})}^{\hm^{-1}(c)} \hm(t)^{1/q} \frac{dt}{t}\leq \int_0^{\hm^{-1}(c)} \hm(t)^{1/q} \frac{dt}{t}.
$$
If we take limits as $N \to \infty$,  we get
$$
\sum_{k=0}^\infty (\tau a_k - a_{k+1})  \leq \int_0^{\hm^{-1}(c)} \hm(t)^{1/q} \frac{dt}{t},
$$ 
which, combined with the equality
\begin{equation*}
\sum_{k=0}^\infty (\tau a_k - a_{k+1})= (\tau-1) \sum_{k=0}^\infty a_k,
\end{equation*}
shows \eqref{eq:Dini-sum}. 
\end{proof}

\vv

\begin{lemma}\label{lem:kato-uniformly<1}
Let $\hm:(0, \infty) \to (0,\infty)$ be a strictly increasing and continuous function such that $\lim_{r \to 0^+} \hm(r) =0$.  Assume that  
$$
C_\hm:=\sup_{r>0} \frac{1}{\hm(r)} \int_0^r \hm(t)\, \frac{dt}{t}< \infty \quad \textup{and}\quad \hm(2r) \leq c_0 \hm(r), \,\,\textup{for any}\,\, r>0.
$$
Then 
\begin{equation}
\sup_{t \in (0,\infty)} \frac{\hm(t)}{\hm(2t)} <1.
\end{equation}
\end{lemma}

\begin{proof}
Since $\hm$ is strictly  increasing and doubling, we have that 
$$
c_0^{-1} \leq \frac{\hm(t)}{\hm(2t)}  < 1, \quad \textup{for every}\,\,t>0.
$$
This inequality and the continuity of $\hm$   in $(0, \infty)$ imply that 
$$
\sup_{t \in (0,\infty)}  \frac{\hm(t)}{\hm(2t)} = 1  \Leftrightarrow \lim_{t \to 0}  \frac{\hm(t)}{\hm(2t)}= 1.
$$
Assume that $\lim_{t \to 0}  \frac{\hm(t)}{\hm(2t)}=1$. Then, by continuity, if we fix $\ve < (4 \,c_0\,C_\hm)^{-1}$, there exists $\rho>0$ such that for  $t< \rho$ it holds that  $ \hm(t) >( 1-\ve)\,\hm(2t)$. If we apply this for $t_m= 2^{-m}\rho$,  $m=0,1,\dots, N-1$, the Dini condition yields
$$
\frac{1- ( 1-\ve)^N}{\ve} \hm(\rho) = \sum_{m=0}^{N-1} ( 1-\ve)^m \hm(\rho) < \sum_{m=0}^{N-1} \hm(2^{-m}\rho) \leq 2\,c_0\,C_\hm \,\hm(\rho).
$$
Letting $N\to \infty$, we get $\ve^{-1} \leq 2\,c_0\,C_\hm$ which is a contradiction.
\end{proof}

\vvv

\subsection{The splitting  lemmas}

The following lemma will be used repeatedly in this manuscript and for the case $p=q=n$ was  proved in \cite{BM}. We extend it to the case of Lorentz spaces $L^{p,q}(\om)$ with $1 < p \leq q < \infty$. 

\begin{lemma}\label{lem:main-splitting-Ln}
 Let $\om \subset \R^{n}$ be an open set, $u \in Y^{1,2}(\om)$, $h \in L^{p,q}(\om)$, for $1 < p \leq q < \infty$ and $a > 0$. Then there exist mutually disjoint measurable sets $\Omega_i \subset \om$ and functions $u_i \in Y^{1,2}(\Omega)$ for $1 \leq i \leq \kappa$ with the following properties:
\begin{enumerate}
\item $\|h\|_{L^{p,q}(\om_i)} =a$, for $1 \leq i \leq \kappa-1$, and  $\|h\|_{L^{p,q}(\om_\kappa)} \leq a $,\label{1exist}
\item $\{ x \in \om: \nabla u_i \neq 0\} \subset \om_i$,\label{1bis-exist}
\item $\nabla u = \nabla u_i$ in $\om_i$, \label{2exist}
\item $|u_i| \leq |u|$,\label{3bis-exist}
\item $uu_i \geq 0$,\label{3exist}
\item $u = \sum_{i=1}^{m}u_i$, \label{4exist}
\item $u_i  \nabla u   = \left(\sum_{j=1}^{i}\nabla u_j\right)u_i $, \label{5exist}
\item $u \nabla u_i = \left(\sum_{j=i}^{\kappa}u_j\right)\nabla u_i $, \label{6exist}
\end{enumerate}
and $\kappa$ has the upper bound
$$\kappa \leq a^{-q}\|h\|_{L^{p,q}(\om)}^q+1.$$
If $u \in Y_0^{1,2}(\om)$,  then $u_i\in Y_0^{1,2}(\om)$ for $1 \leq i \leq \kappa$. 
\end{lemma}

\begin{proof}
If $0 \leq k < t \leq \infty$, we define
\begin{equation}\label{eq:L(omkinfty)finite}
\Omega(k,t):= \{x \in \Omega: k < |u| \leq  t, \nabla u \not = 0 \},
\end{equation}
and by Chebyshev's inequality,  for $k>0$, it holds
$$|\om(k, t)| \leq |\om(k, \infty)|\leq k^{-2^*}\| u\|_{L^{2^*}}^{2^*}< \infty.$$
Let us  define the function $f: [0, \infty]^2 \to [0, \infty)$  by
$$f(k, t) = |\{ k < |u| \leq t, \nabla u \not = 0\}|.$$
We will show that $f(\cdot, t)$ is  continuous  in $[0, \infty)$ for any fixed $t \in (0, \infty]$.

To this end, fix $t \in (0, \infty]$ and $k < t$, and let $\{k_\ell\}_{\ell \in \bN}$ be a positive decreasing sequence so that $k_\ell \to k$. Thus, 
$$f(k,t) = |\Omega(k,t)|= \big|\bigcup_{\ell=1}^{\infty}\Omega(k_\ell,t) \big| = \lim_{\ell \to \infty}f(k_\ell,t),$$
which gives right continuity. Consider now an increasing sequence of positive numbers $\{k_l\}_{l \in \bN}$ so that $k_l \to k$. Then  
$$
\bigcap_{l=1}^{\infty}\Omega(k_l,t) = \Omega(k,t) \cup \{x \in \om: |u|= k, \nabla u \not = 0 \}.
$$
By Lemma \ref{lem:Sobolev-maxmin}, we get  $\big|\{x \in \om: |u|= t, \nabla u \not = 0 \}\big|=0$, and thus, since $|\om(k_1, \infty)| < \infty$, we infer that
$$f(k,t) = |\Omega(k,t) | = \big|\bigcap_{l=1}^{\infty}\Omega(k_l,t) \big| = \lim_{l \to \infty} \big|\Omega(k_l,t) \big|,$$
which implies left continuity of $f(\cdot, t)$ and consequently continuity. 

If we set 
\begin{equation*}
\sigma(x)=
\begin{dcases}
1 &,\textup{if}\,\, x>0\\
-1&,\textup{if}\,\,x<0
\end{dcases},
\end{equation*}
we define 
\begin{equation*}
F_{k,t}(u)=
\begin{dcases}
(t-k)\sigma(u), & |u| > t \\
u- k \sigma(u),&  k< |u| \leq t \\
0, & |u| \leq k,
\end{dcases}
\quad\textup{and}\quad
F_{k,\infty}(u)=
\begin{dcases}
u-k \sigma(u), & |u| > k \\
0, & |u| \leq k
\end{dcases}.
\end{equation*}

For fixed $k, t \in [0, \infty]$, $F_{k,t} \in \Lip (\R)$ and $F_{k,t}(0)=0$, and thus, since $u \in Y^{1,2}(\om)$ (resp. $Y_0^{1,2}(\om)$), by Lemma \ref{lemma:lip0integral}, $F_{k,t}(u) \in Y^{1,2}(\om)$ (resp. $Y_0^{1,2}(\om)$). 

 Recall that the $L^{p,q}$-norm is absolutely continuous  by  Lemma \ref{lem:abs.cont. norms} and thus, since, for any fixed $t \in [0, \infty]$, ${\bf 1}_{\om(k,t)} \to 0$ a.e. as $k \to t$, we will have that $\| h {\bf 1}_{\om(k,t)}\|_{L^{p,q}(\om)} \to 0$. For $1< p \leq q < \infty$, let us define
  \begin{equation}
 H(k,t) :=  \int_0^\infty s^q \,d_{h {\bf 1}_{\om(k,t)}}(s)^\frac{q}{p} \,\frac{ds}{s}.
 \end{equation}
 
 If  $H(0, \infty)\leq  a^q$,  then  we set $\om_1= \{x\in \om: \nabla u \not = 0\}$ and $u_1 = u$. Suppose now that $H(0, \infty) > a^q$, and thus, by the absolute  continuity of $L^{p,q}$, there exists $k_1 > 0$ such that
$$H(k_1, \infty) = a^q.$$
If $H(0, k_1) \leq a^q$, we set $\om_1= \om(k_1, \infty)$ and $\om_2= \om(0, k_1)$, and $u_1=F_{k_1, \infty}(u)$ and $u_2=F_{0, k_1}(u)$. If, on the other hand,  $H(0, k_1) \geq a^q$,  there exists $k_2 \geq 0$ so that
$$H(k_2, k_1) = a^q.$$
If we iterate, there exists $j_0 \in \bN$ so that $H(k_i, k_{i-1}) = a^q$, if $1 \leq i <j_0$, and $H(0, k_{j_0}) \leq a^q$, where $k_0=+\infty$. Indeed, if there were infinitely many $i$ so that $H(k_i, k_{i-1}) = a^q$, then, since $\{\om(k_i, k_{i-1})\}_{i\geq 1}$ are disjoint, we would have 
$$
\infty=  \sum_{i=1}^\infty a^q = \sum_{i=1}^\infty H(k_i,k_{i-1}) \leq \int_0^\infty s^q \,d_{h {\bf 1}_{\om(0, \infty)}}(s)^\frac{q}{p} \,\frac{ds}{s}  \leq \| h\|_{L^{p,q}(\om)}^q<\infty,
$$
which is a contradiction. Here we used that $p \leq q $ and that for disjoint sets $A$ and $B$ it holds that 
$$
|\{ x \in A: |f|>t\}|+ |\{ x \in B: |f|>t\}| \leq |\{ x \in A \cup B: |f|>t\}|.
$$
 The same argument  gives us $j_0 \,a^q \leq  \|h\|_{L^{p,q}(\om)}$, that is, $ j_0 \leq a^{-q}\|h\|_{L^{p,q}(\om)}^q$. 

If we set $\kappa=j_0+1$ and $k_{\kappa}=0$, for $i \in \{1, \dots, \kappa\}$, we define 
$$
\om_i=\om(k_i, k_{i-1})\quad \textup{ and} \quad u_i=F_{k_i, k_{i-1}}(u).
$$ 
We have already shown \eqref{1exist}, so it remains to prove that \eqref{1bis-exist}-\eqref{6exist} hold as well. 

Firstly, \eqref{1bis-exist}, \eqref{2exist}, and \eqref{3bis-exist} are clear by definition, while \eqref{3exist} follows by simple computations; indeed,  note first that $u u_i=0$ whenever $|u| < k_{i}$. In the set where  $|u| > k_{i-1} > k_{i}$, we have that
$$uu_i = u \, \sigma(u) \,(k_{i-1}-k_i) =|u| \,(k_{i-1}-k_i)  \geq 0,$$
while, when $ k_{i}<|u| \leq k_{i-1}$,
$$uu_i = u^2- \sigma(u)\,u\,k_{i}=|u| (|u| - k_{i}) \geq 0.$$ 
This concludes the proof of \eqref{3exist}.

 For \eqref{4exist} and \eqref{5exist}, we may rewrite $u_j= u_{k_j, \infty}-u_{k_{j-1, \infty}}$, in view of which, we have
\begin{equation}\label{eq:mainlemma-sum-tele}
\sum_{j=1}^{i}u_j=  F_{k_1, \infty}(u)+\sum_{j=2}^{i} (F_{k_j, \infty}(u)-F_{k_{j-1}, \infty}(u)) =F_{k_i, \infty}(u).
\end{equation}
 In the case $i= \kappa $, we have
$$\sum_{j=1}^{\kappa } u_j = F_{k_\kappa, \infty}(u) = u,$$
yielding \eqref{4exist}.
By definition, $\nabla u_{k_i, \infty} = \nabla u$, when $|u| > k_i$ (i.e.,  in the support of $u_i$), while $u_i = 0$, whenever $|u| \leq k_i$.
and so, \eqref{5exist} follows from  \eqref{eq:mainlemma-sum-tele}.
Since $\{\nabla u_i \neq 0 \} \subset \om_i$ we can use \eqref{4exist} to get
$$ 
 u \nabla u_i = u_i \sum_{j=1}^{\kappa } \nabla  u_j= \nabla u_i \sum_{j=i}^{\kappa}  u_j.
$$
This concludes the proof of the lemma.
\end{proof}

  \vv
The direct analogue of this lemma for the space $\mathcal{K}'(\om)$ was proved in \cite{ViZa} but it is not stated as such. For the reader's convenience we will give a sketch of the proof. 
\begin{lemma}\label{lem:main-splitting-Kato} 
Let $\om \subset \R^{n}$ be an open set, $u \in Y^{1,2}(\om)$ (resp. $ Y_0^{1,2}(\om)$), $h \in \mathcal{K}'(\Omega)$ and $a > 0$. Then, there exist mutually disjoint measurable sets $\Omega_i \subset \om$ and functions $u_i \in Y^{1,2}(\Omega)$ (resp. $ Y_0^{1,2}(\om)$), for $1 \leq i \leq \kappa$, satisfying  \eqref{1bis-exist}-\eqref{6exist}, so that 
$$
 \vartheta_\om (h {\bf 1}_{\om_i})=a^2, \,\,\textup{for}\,\,1 \leq i \leq \kappa-1,\quad \textup{and}\quad \vartheta_\om (h {\bf 1}_{\om_\kappa} )\leq a^2.
$$
If $\rho_0>0$ is such that  $ \vartheta_\om (h, \rho_0)= a^2/4$, then  $\kappa$ has the upper bound
$$
 \kappa \leq   1 + 2 \,  a^{-2} \,{\rho_0^{2-n}} \|h\|_{L^1(\om)}.
$$
If $\om$ is a bounded open set contained in a ball $B_r$,  we can assume $h \in \mathcal{K}(\Omega)$ replacing $ \vartheta_\om (\cdot)$ by $ \vartheta_\om (\cdot,r)$.
\end{lemma}

\begin{proof}
Using the same notation as before, we define
$$
H(k,t) = \vartheta_\om (h {\bf 1}_{\om(k,t)}).
$$
Making the same stopping time argument with respect to the condition $h(k,t) =a^2$ and noticing that  we only used the absolute continuity of the norm, we can reason as in the proof of Lemma \ref{lem:main-splitting-Ln}. The only difference lies on the estimate of $\kappa$ since we cannot linearize it as we did in the previous case.

Let us first show that the stopping process results to  a finite number of  sets. Indeed, arguing as in the proof of Lemma \ref{lem:kato-approx-0}, we can find $\rho_0$ so that $ \vartheta_\om (h, \rho_0)= a^2/4$ so that 
$$
a^2= \vartheta_\om ( h {\bf 1}_{\om_i}) \leq 2  \vartheta_\om (h {\bf 1}_{\om_i}, \rho_0) + {\rho_0^{2-n}} \int_{\om_i} |h| {\bf 1}_{\om_i}\,dy \leq \frac{a^2}{2} + {\rho_0^{2-n}} \int_{\om_i} |h| {\bf 1}_{\om_i}\,dy.
$$
So, if assume that there infinite many $\om_i$, we can sum in $i$ as before and get
$$
\infty \leq {\rho_0^{2-n}} \sum_i \int_{\om_i} |h| {\bf 1}_{\om_i}\,dy \leq  {\rho_0^{2-n}} \| h\|_{L^1(\om)},
$$
which is a contradiction. If $j_0$ is the number of $i$'s for which $\vartheta_\om (h {\bf 1}_{\om_i})=a^2$, the same argument will give the bound
$$
j_0 \leq 2 a^{-2} {\rho_0^{2-n}} \| h\|_{L^1(\om)}.
$$
\end{proof}

  \vv
  
\begin{remark}
It is  interesting to see that the bound on $\kappa$, although at a first glace does not seem to be scale invariant,  in fact it is (with the correct scaling).  Indeed, let $h_r=r^2 h(rx)$ in the open set  $\om_r = r^{-1} \om$. Then, by making  the change of variables $y = rx$  we have that
\begin{align*}
{\rho_0^{2-n}} \| h_r\|_{L^1(\om_r)} = (\rho_0 \,r)^{2-n}  \| h\|_{L^1(\om)} .
\end{align*}
Now, recall that $\rho_0$ was chosen  so that $\vartheta_{\om_r} ( h_r, \rho_0) = a^2/4$, which, by the same change of variables, implies that $\vartheta_{\om} ( h, \rho_0\, r) = a^2/4$. Note that  if $\vartheta_{\om} ( h, \cdot)$ is invertible, we  have that $\rho_0 \,r = \vartheta_{\om}^{-1}( h, a^2/4)$.
\end{remark}

\vv

\subsection{Variational capacity }

\begin{definition}
Let $\Omega \subset \R^n$ be  open and  $E \subset \Omega$. If we  set 
$$
\bK_E(\om):=\{w \in Y^{1,2}_0(\Omega): E \subset \{w \geq 1 \}^\textup{o} \}
$$
then we define the {\it (variational) capacity of the condenser} $(E, \om)$ as
$$
\Cap(E, \om)= \inf_{w \in \bK_E} \int_\om |\nabla w|^2.
$$

The following properties of capacity   verify that it is a Choquet capacity and satisfies the axioms considered by Brelot. A proof can be found for instance in Theorem 2.3 in \cite{MZ}. 

\begin{itemize}
\item[(i)] If $E \subset \om$ is compact,
$$
\Cap(E, \om)= \inf\left\{ \int_\om |\nabla w|^2 : w \in C^\infty_c(\om), \,\, u \geq 1 \,\,\textup{in}\,\, E\right\}.
$$
\item[(ii)] If $E \subset \om$ is open, 
$$
\Cap(E, \om)=\sup_{\textup{compact} \,\,K \subset E} \Cap(K, \om).
$$
\item[(iii)] If $ E_1 \supset E_2 \supset \dots$ is a sequence of compact subsets of $\om$,
$$
\Cap\Big( \bigcap_{j \geq 1} E_j, \om \Big) = \lim_{j \to \infty} \Cap(E_j, \om).
$$
\item[(iv)] If $ E_1 \subset E_2 \subset \dots$ is a sequence of arbitrary subsets of $\om$,
$$
\Cap\Big( \bigcup_{j \geq 1} E_j, \om \Big) = \lim_{j \to \infty} \Cap(E_j, \om).
$$
\item[(v)] If $ E_1, E_2 \subset \dots$ are arbitrary subsets of $\om$, then
$$
\Cap\Big( \bigcup_{j \geq 1} E_j, \om \Big)  \leq  \sum_{j \geq 1} \Cap(E_j, \om).
$$
\end{itemize}
\end{definition} 

\vv

\section{Interior and boundary Caccioppoli inequality}\label{sec:caccio}

In sections \ref{sec:caccio}-\ref{sec:dirichlet} we will be dealing with subsolutions and supersolutions of the equation
\begin{equation}\label{eq:Lu=f-divg}
Lu=-\divv (A\nabla u + b u) -c \nabla u - du = f - \divv g,
\end{equation}
where $ f \in L_{\textup{loc}}^{1}(\om)$ and  $g \in L_{\textup{loc}}^2(\om; \R^n)$.


\subsection{Standard Caccioppoli inequality}\label{subsec:caccio}

\begin{theorem}[\bf Caccioppoli inequality I] \label{thm:subCaccioppoli}
Let $u \in Y_{\loc}^{1,2}(\om)$ be either a solution  or a non-negative subsolution of \eqref{eq:Lu=f-divg} and $ f \in L_\loc^{2_*}(\om)$. Assume also that   \eqref{negativity} is satisfied and either {(i)} $b+c \in L^{n,q}_\loc(\om)$, for $q \in [n, \infty)$, or {(ii)} $|b+c|^2 \in \mathcal{K}_\loc(\om)$.  For a  non-negative function $\eta \in C_c^{\infty}(\Omega)$, we let $\om'$ be a bounded open set  such that $\supp \eta \subset \om' \Subset \om$. Then it holds
\begin{equation}\label{Caccio-sub-1}
\|\eta \nabla u \|^2_{L^2(\om')}  \lesssim  \|u \nabla \eta \|^2_{L^2(\om')}+ \|f\eta\|_{L^{2_*}(\om')}^2 +\|g \eta\|_{L^2(\om')}^2,
\end{equation}
where the implicit constant depends only on $\lambda$,  $\Lambda$, and also either on  $C_{s,q}$ and $\|b+c\|_{L^{n,q}(\om')}$, for $q \geq n$ under  assumption (i), or  $C_s'$ and $\vartheta_{\om'}(|b+c|^2, 2 \diam \om')$ under  assumption (ii)\footnote{ Recall that $C_s'$ and  $C_{s,q}$  are the constants in Lemmas   \ref{lem:Kato-emb-grad-Br}  and \ref{lem:Lor-Sobolev-emb} respectively.}.
\end{theorem}

\begin{proof}
We will only treat the case that $u$ is a non-negative subsolution of \eqref{eq:Lu=f-divg} as the proof when $u$ is a solution is almost identical and is omitted. Notice that since $K:= \supp\eta$  is a compact subset of $\om$, we can always find a bounded open set $\om'$ such that $K \subset \om' \Subset \om$, and as $u \in Y^{1,2}_\loc(\om)$, it holds that $u \in Y^{1,2}(\om')$. Working in $\om'$ instead of $\om$, we may assume, without  loss of generality, that $u \in Y^{1,2}(\om)$. Moreover, $u$ is clearly a subsolution in any open subset of $\om$.  For simplicity, let us preserve the notation $\om$ instead of $\om'$.

We first assume that $b+c \in L^{n,q}(\om')$. Apply Lemma $\ref{lem:main-splitting-Ln}$ to the function $u$, for $p= n$, $q \geq n$, $h= b+c$, and $a = \frac{\lambda}{8C_{s,q}}$, where $C_{s,q}$ is the constant in \eqref{eq:Lorentz-Sobolev-emb}, to find $\om_i \subset \om$ and  $u_i \in Y^{1,2}(\om)$,  $1 \leq i \leq \kappa$, satisfying  \eqref{1exist}--\eqref{6exist}. Note that  $\eqref{3exist}$ tells us that $u_i$ and $u$ have the same sign, and so, the functions $\eta^2 u_i \in Y^{1,2}_0(\om)$ are non-negative. Thus, using that $u$ is a subsolution for \eqref{eq:Lu=f-divg} we have
\begin{align*}
\int_\om f (\eta^2u_i) + \int_\om  g \nabla(\eta^2 u_i)&\geq \int_{\om}A \nabla u \nabla(\eta^2 u_i) + bu\nabla (\eta^2u_i) -c\nabla u(\eta^2 u_i)-du (\eta^2u_i)  \\
&= \int_{\om}A \nabla u \nabla(\eta^2 u_i) + b \nabla (\eta^2 u u_i) -(b+c)  \nabla u \eta^2u_i-d \eta^2 u u_i\\
&\geq \int_{\om}A \nabla u \nabla(\eta^2 u_i) -(b+c)  \nabla u \eta^2u_i,
\end{align*}
where in the last inequality we used \eqref{3exist}, Lemma \eqref{lem:Lor-Sob-uw}, Remark \ref{rem:bd-cd}, and \eqref{negativity}. In view of $\eqref{2exist}$ and $\eqref{4exist}$, the latter inequality can be written as
\begin{multline}\label{eq:twosums}
\int_{\om_i} A \nabla u_i \nabla u_i \eta^2 \leq -2\int_{\om}A\nabla u \nabla \eta u_i \eta   + \sum_{j=1}^{i}\int_{\om_j}(b+c)\nabla u_j \eta^2 u_i\\
 +\int_\om f (\eta^2u_i) + \int_\om  g \nabla(\eta^2 u_i) =:\textup{I}_1(i)+\textup{I}_2(i)+\textup{I}_3(i)+\textup{I}_4(i).
\end{multline}
By \eqref{eqelliptic1} we get
\begin{equation}\label{eq:lambdabound}
 \lambda \|\eta \nabla u_i\|_{L^2}^2  \leq  \int_{\om_i}A \nabla u_i \nabla u_i \eta^2,
\end{equation}
while, by  H\"older's inequality,
\begin{align}
|\textup{I}_1(i)|  &\leq 2\Lambda \| \eta \nabla u \|_{L^2} \| u_i \nabla \eta  \|_{L^2}. \label{eq:I1(i)}
\end{align}

If we apply \eqref{eq:Lorentz-Sobolev-emb} and Young's inequality, along with  the fact that $ \|b+c\|_{L^{n,q}(\om_j)} \leq \frac{\lambda}{8C_{s,q}}$ for any $1 \leq j \leq \kappa$,  we get that
 \begin{align}
 \textup{I}_2(i) &=  \int_{\om_i} (b+c) \nabla u_i \eta^2 u_i  +   \sum_{j=1}^{i-1 }\int_{\om_j}(b+c) \nabla u_j \eta^2 u_i\notag \\
& \leq  C_{s,q} \frac{\lambda}{8C_{s,q}}  \| \eta \nabla u_i \|_{L^2}  \| \nabla(u_i \eta)\|_{L^{2}} +  C_{s,q} \frac{\lambda}{8C_{s,q}}  \sum_{j=1}^{i-1} \|\eta \nabla u_j \|_{L^2}  \| \nabla(u_i \eta)\|_{L^{2}}\notag\\
 &\leq \frac{3\lambda}{16}  \|\eta \nabla u_i \|^2_{L^2}+  \frac{\lambda}{16} \| u_i \nabla \eta\|_{L^2}^2+  \frac{\lambda}{16}( \|u_i \nabla\eta \|_{L^2} + \|\eta \nabla u_i  \|_{L^2})^2\notag \\ 
 &+  \frac{\lambda}{16}  \left( \sum_{j=1}^{i-1}\|   \eta \nabla u_j\|_{L^2} \right)^2\notag\\
 &\leq\frac{5\lambda}{16} \|\eta \nabla u_i \|^2_{L^2}+\frac{3\lambda}{16} \|u_i \nabla \eta\|^2_{L^2} +  \frac{\lambda}{16}  \left( \sum_{j=1}^{i-1}\|   \eta \nabla u_j\|_{L^2} \right)^2.\label{eq:I2(i)}
 \end{align}
 
By H\"older's, Sobolev's and Young's inequalities we obtain
\begin{align} \label{eq:I34(i)}
 \textup{I}_3(i) +\textup{I}_4(i) 
\leq & \frac{C_{s,q}^2}{4 \delta} \|f \eta \|_{L^{2_*}}^2 + \frac{1}{2 \delta} \|g \eta \|_{L^{2}}^2+2 \delta \| u_i \nabla  \eta\|^2_{L^{2}} + 2 \delta\| \eta \nabla u_i\|^2_{L^2}.
 \end{align}
Choosing $\delta= \frac{\lambda}{32}$ in \eqref{eq:I34(i)}, we can combine  \eqref{eq:twosums},  \eqref{eq:lambdabound},  \eqref{eq:I1(i)}, and \eqref{eq:I2(i)} and infer that
 \begin{align*}
 \frac{ 3 \lambda}{8} \|\eta \nabla u_i \|^2_{L^2} 
 &\leq   \left( \frac{4 \Lambda^2}{\lambda} + \frac{\lambda}{4} \right) \|u_i \nabla \eta\|^2_{L^2} +  \frac{\lambda}{16}  \left( \sum_{j=1}^{i-1}\|   \eta \nabla u_j\|_{L^2} \right)^2  + \frac{16}{\lambda}\|g \eta \|_{L^{2}}^2\\
 &+\frac{16\,C_{s,q}^2}{\lambda} \|f \eta \|_{L^{2_*}}^2,
  \end{align*}
which implies that there exist positive constants  $C_1$, $C_2$ and $ C_3$ depending on $\lambda$, $\Lambda$ and $C_{s,q}$ so that 
\begin{align}
\|\eta \nabla u_i \|_{L^2} & \leq C_1 \|u_i \nabla \eta\|_{L^2} + C_2 \left( \|f \eta \|_{L^{2_*}}+ \|g \eta \|_{L^{2}} \right) +  \sum_{j=1}^{i-1}\|   \eta \nabla u_j\|_{L^2}.\\
&+ C_3 \| \eta \nabla u \|_{L^2}^{1/2} \| u_i \nabla \eta  \|_{L^2}^{1/2}.
\end{align}
Note that  the constant  the sum is multiplied with is indeed $1$, which is convenient in the iteration argument below.
If we denote $C_0:=\max(C_1, C_2, C_3)$,
$$
 x_j:=\|\eta \nabla u_j \|_{L^2}, \,\,\textup{and}\,\, y_0:=\|u \nabla \eta\|_{L^2}+ \| \eta \nabla u \|_{L^2}^{1/2} \| u \nabla \eta  \|_{L^2}^{1/2} + \|f \eta \|_{L^{2_*}} +\|g \eta \|_{L^{2}},
$$
and use that \eqref{3bis-exist}, the latter inequality can be written as
\begin{align}\label{eq:Caccio-recursive}
 x_1 &\leq C_0\, y_0, \notag\\
 x_i &\leq C_0\, y_0 + \sum_{j=1}^{i-1} x_j, \quad \textup{for} \,\,i=2, \cdots, \kappa.
  \end{align}
By induction,  we get
\begin{align}\label{eq:Caccio-induction}
 x_i  \leq 2^{i-1} \,C_0\, y_0.
  \end{align}
 Indeed, for $i=1$, it holds $x_1 \leq C_0\, y_0$. Assume now that $x_j \leq 2^{j-1}  \,C_0\, y_0$ for all $1 \leq j \leq i-1$. Then, by \eqref{eq:Caccio-recursive} and the induction hypothesis, 
$$
 x_i \leq C_0\, y_0 + C_0\, y_0 \sum_{j=1}^{i-1}  2^{j-1} = 2^{i-1}\,C_0\, y_0.
 $$

Summing \eqref{eq:Caccio-induction} in $i \in \{ 1, \dots, \kappa\}$  we obtain
 \begin{align}\label{eq:Caccio-sum-xi}
\sum_{i=1}^\kappa  x_i  \leq 2^{\kappa} \,C_0\, y_0,
 \end{align}
which, in light of \eqref{4exist}, \eqref{eq:Caccio-sum-xi} and Young's inequality (with a small constant), implies that
$$
\| \eta \nabla u \|_{L^2} \leq \sum_{i=1}^\kappa   \| \eta \nabla u_i \|_{L^2} \leq  4^{\kappa} \,C_0^2\ \left( \| u \nabla \eta\|_{L^2}+ \|f\eta\|_{L^{2_*}} +\|g \eta\|_{L^2} \right).
$$
This concludes our proof when $b+c \in {L^{n,q}(\om;\, \R^n)}$, since  $\kappa$ depends only $\lambda$,  $\Lambda$,  $C_{s,q}$, and also on $\|b+c\|_{L^{n,q}(\om;\, \R^n)}$. 

Let us now prove the same result in the case $|b+c|^2 \in \mathcal{K}(\om')$. We apply Lemma $\ref{lem:main-splitting-Kato}$ to the function $u$, for $h= b+c$, and $a = \frac{\lambda}{8C_s'}$, where $C_s'$ is the constant in \eqref{eq:Kato-emb-grad-Br}, to find $\om_i \subset \om$ and  $u_i \in Y^{1,2}(\om)$,  $1 \leq i \leq \kappa$, satisfying  \eqref{1exist}--\eqref{6exist}. The main argument will be exactly the same as in the previous case will not be repeated. Although, there is a difference coming from the embedding theorem we apply, which is Lemma \ref{lem:Kato-emb-grad-Br} as opposed to Lemma \ref{lem:Lor-Sobolev-emb} we used before. Taking this under consideration, it is enough to handle the term $I_2(i)$.

To this end,  apply Cauchy-Scwharz's inequality, \eqref{eq:Kato-emb-grad-om}, Sobolev's and Young's inequalities, along with  the fact that for any $1 \leq j \leq m$ it holds $ \vartheta_{\om'}(|b+c|^2{\bf 1}_{\om_j}, 2 \diam \om') \leq \frac{\lambda}{8C_s'}$, and get that
\begin{multline}
 \textup{I}_2(i) =  \int_{\om_i} (b+c) \nabla u_i \eta^2 u_i  +   \sum_{j=1}^{i-1 }\int_{\om_j}(b+c) \nabla u_j \eta^2 u_i\notag \\
\leq   C_s'   \| \nabla(u_i \eta)\|_{L^{2}}  \left( \vartheta^{1/2}_{\om'}(|b+c|^2{\bf 1}_{\om_i})   \| \eta \nabla u_i\|_{L^2} +   \sum_{j=1}^{i-1 }\vartheta^{1/2}_{\om'}(|b+c|^2{\bf 1}_{\om_j})  \|\eta \nabla u_j \|_{L^2} \right) \notag\\
 \leq  \frac{\lambda}{8}  \| \eta \nabla u_i \|_{L^2}  \| \nabla(u_i \eta)\|_{L^{2}} +   \frac{\lambda}{8}  \sum_{j=1}^{i-1} \|\eta \nabla u_j \|_{L^2}  \| \nabla(u_i \eta)\|_{L^{2}}\notag\\
 \leq \frac{\lambda}{16}  \|\eta \nabla u_i \|^2_{L^2}+  \frac{\lambda}{16} \| u_i \nabla \eta\|_{L^2}^2+  \frac{\lambda}{16}( \|u_i \nabla\eta \|_{L^2} + \|\eta \nabla u_i  \|_{L^2})^2 + \frac{\lambda}{16}  \left( \sum_{j=1}^{i-1}\|   \eta \nabla u_j\|_{L^2} \right)^2\notag\\
 \leq\frac{3\lambda}{16} \|\eta \nabla u_i \|^2_{L^2}+\frac{3\lambda}{16} \|u_i \nabla \eta\|^2_{L^2} +  \frac{\lambda}{16}  \left( \sum_{j=1}^{i-1}\|   \eta \nabla u_j\|_{L^2} \right)^2.\label{eq:I2(i)}
\end{multline}
This concludes the proof the Theorem.
\end{proof}

  \vv

\begin{theorem}[\bf Caccioppoli inequality II]\label{thm:Caccioppoli2} 
Let $u \in Y_{\loc}^{1,2}(\om)$ be either a solution  or a non-negative subsolution of \eqref{eq:Lu=f-divg} and $ f \in L_\loc^{2_*}(\om)$. Assume also that   \eqref{negativity2} is satisfied and either (i) $b+c \in L^{n,q}_\loc(\om)$, for $q \in [n, \infty)$, or (ii) $|b+c|^2 \in \mathcal{K}_\loc(\om)$.  For a  non-negative function $\eta \in C_c^{\infty}(\Omega)$, we let $\om'$ be a bounded open set  such that $\supp \eta \subset \om' \Subset \om$. Then it holds
\begin{equation}\label{Caccio-sub-1}
\|\eta \nabla u \|^2_{L^2(\om')}  \lesssim  \|u \nabla \eta \|^2_{L^2(\om')}+ \|f\eta\|_{L^{2_*}(\om')}^2 +\|g \eta\|_{L^2(\om')}^2,
\end{equation}
where the implicit constant depends only on $\lambda$,  $\Lambda$,  and also either on  $C_{s,q}$ and $\|b+c\|_{L^{n,q}(\om')}$, for $q \geq n$ under  assumption (i), or  $C_s'$ and $\vartheta_{\om'}(|b+c|^2, 2 \diam \om')$ under  assumption (ii).
\end{theorem}

\begin{proof}
We only deal with the case that $u$ is a non-negative subsolution \eqref{eq:Lu=f-divg}. As seen in Theorem \ref{thm:subCaccioppoli}, we may assume that $u \in Y^{1,2}(\om)$ and apply  Lemma $\ref{lem:main-splitting-Ln}$ to the function $u$, for $p= n$, $q \geq n$, $h= b+c$, and $a = \frac{\lambda}{8C_{s,q}}$, where $C_{s,q}$ is the constant in \eqref{eq:Lorentz-Sobolev-emb}. Using that $\eta^2 u_i \in Y^{1,2}_0(\om)$ and non-negative, along with the fact that $u$ is a subsolution, we have
\begin{align*}
\int_\om f (\eta^2u_i) + \int_\om  g \nabla(\eta^2 u_i)&\geq \int_{\om}A \nabla u \nabla(\eta^2 u_i) + bu\nabla (\eta^2u_i) -c\nabla u(\eta^2 u_i)-du (\eta^2u_i)  \\
&\geq \int_{\om}A \nabla u \nabla(\eta^2 u_i) -(b+c)   u \nabla(\eta^2u_i),
\end{align*}
where in the last inequality we used \eqref{3exist}, Lemma \eqref{lem:Lor-Sob-uw},  Remark \ref{rem:bd-cd}, and \eqref{negativity2}. In view of $\eqref{2exist}$ and $\eqref{4exist}$, the latter inequality can be written as
\begin{align}
\int_{\om}A \nabla u_i \nabla u_i &\eta^2 \leq  -2\int_{\om}A\nabla u \nabla \eta u_i \eta   + \int_{\om_i}(b+c)\nabla u_i \eta^2 u
+2\int_{\om}(b+c)\nabla \eta u_i u \eta \notag\\ &+\int_\om f (\eta^2u_i) + \int_\om  g \nabla(\eta^2 u_i)=:-2\textup{I}_1(i)+\textup{I}_2(i)  + 2 \textup{I}_3(i)+\textup{I}_4(i)+\textup{I}_5(i).\label{eq:twosums2}
\end{align}

By H\"older's inequality,
\begin{align}
\textup{I}_1(i)  \leq \Lambda \|\eta \nabla u \|_{L^2} \|u_i \nabla \eta \|_{L^2}.
 \label{eq:I1(i)-2}
\end{align}

Using \eqref{6exist} and the fact that $ \|b+c\|_{L^{n,q}(\om_j)} \leq \frac{\lambda}{8C_{s,q}}$ for all $1 \leq j \leq \kappa$, along with \eqref{eq:Lorentz-Sobolev-emb} and Young's inequality, we have 
 \begin{align}
 \textup{I}_2(i) &=  \int_{\om_i} (b+c) \nabla u_i \eta^2 u_i  +  
  \sum_{j=i+1}^{\kappa}\int_{\om_j}(b+c) \nabla u_j \eta^2 u_i \notag\\
& \leq  C_{s,q} \frac{\lambda}{8C_{s,q}}  \| \eta \nabla u_i \|_{L^2}  \| \nabla(u_i \eta)\|_{L^{2}} +  
C_{s,q} \frac{\lambda}{8C_{s,q}}  \sum_{j=i+1}^{\kappa} \|\eta \nabla u_j \|_{L^2}  \| \nabla(u_i \eta)\|_{L^{2}}\notag\\
 &\leq \frac{\lambda}{8}  \|\eta \nabla u_i \|^2_{L^2}+  
 \frac{\lambda}{8} \|\eta \nabla u_i \|_{L^2}  \| u_i \nabla \eta\|_{L^2}\notag\\
 &+  \frac{\lambda}{16}( \|u_i \nabla\eta \|_{L^2}^2 + \|\eta \nabla u_i  \|_{L^2}^2)+ 
 \frac{\lambda}{16} \left( \sum_{j=i+1}^{\kappa}\|\eta \nabla u_j   \|_{L^2} \right)^2\notag\\
 &\leq \frac{\lambda}{4}  \|\eta \nabla u_i \|^2_{L^2}+\frac{ \lambda}{8} \|u_i \nabla \eta\|^2_{L^2} + 
\frac{\lambda}{16} \left( \sum_{j=i+1}^{\kappa}\|\eta \nabla u_j   \|_{L^2} \right)^2.\label{eq:I2(i)-2}
 \end{align}

 If $\delta>0$ is small enough to be chosen, then by similar  (but easier) considerations we get
  \begin{align}
   \textup{I}_3(i) 
&\leq  C_{s,q} \|b+c\|_{L^{n,q}}  \|   u \nabla \eta\|_{L^2}  \| \nabla(u_i \eta)\|_{L^{2}} \\
 &\leq \frac{ C_{s,q}^2}{4 \delta} \|b+c\|^2_{L^{n,q}}   \| u \nabla \eta\|^2_{L^2}+  
  \delta \, \|\eta \nabla u_i \|^2_{L^2} + \delta \,\| u_i \nabla \eta\|^2_{L^2}.
 \label{eq:I3(i)-2}
 \end{align}

If we apply H\"older's, Sobolev's and Young's inequalities we get
  \begin{align}
 \textup{I}_4(i) +\textup{I}_5(i) 
 \leq  \frac{C_{s,q}^2}{4 \rho} \|f \eta \|_{L^{2_*}}^2 &+ \left(1+\frac{1}{ 4 \rho}\right) \|g \eta \|_{L^{2}}^2 \label{eq:I45(i)2} \\
& + (1+2 \rho) \| u_i \nabla  \eta\|^2_{L^{2}} + 2 \rho\| \eta \nabla u_i\|^2_{L^2}. \notag
 \end{align}
 
Choose now  $\delta= \frac{\lambda}{16}$ and $\rho= \frac{\lambda}{8}$. Combining  \eqref{eq:twosums2},  \eqref{eqelliptic1},  \eqref{eq:I1(i)-2}, \eqref{eq:I2(i)-2},  \eqref{eq:I3(i)-2}, and \eqref{eq:I45(i)2}, and using \eqref{3bis-exist}, we can find positive constants $C_1=C_1(\lambda, C_{s,q}, \|b+c\|_{L^{n,q}})$, $C_2=C_2(\lambda, C_{s,q})$ and $C_3=C_3(\lambda)$ so that
 \begin{align*}
\|\eta \nabla u_i \|^2_{L^2} &\leq 2\Lambda \|\eta \nabla u \|_{L^2} \|u \nabla \eta \|_{L^2} + C_1 \|u \nabla \eta\|^2_{L^2} + C_2 \|f \eta \|_{L^{2_*}}^2 \\
 &+ C_3 \|g \eta \|_{L^{2}}^2 + \frac{\lambda}{16} \left( \sum_{j=i+1}^{\kappa}\|\eta \nabla u_j   \|_{L^2} \right)^2.
  \end{align*}
For $j \in \{1,\dots, \kappa\}$, we set
$$
x_j:=\|\eta \nabla u_j \|_{L^2} 
$$
and
\begin{align*}
y_0:= \sqrt{2\Lambda} \|\eta \nabla u \|^{\frac{1}{2}}_{L^2(\om)} \|u \nabla \eta \|^{\frac{1}{2}}_{L^2(\om)} &+ \sqrt{C_1} \|u \nabla \eta\|_{L^2} + \sqrt{C_2}  \|f \eta \|_{L^{2_*}} + \sqrt{C_3} \|g \eta \|_{L^{2}},
\end{align*}
and so, the latter inequality can be written as
\begin{align}\label{eq:Caccio-recursive2}
 x_\kappa &\leq y_0 \,\, \textup{and} \,\, x_i \leq y_0  + \sum_{j=i+1}^{\kappa} x_j, \,\, \textup{for}\,\, i =1,2, \cdots, \kappa-1.
  \end{align}
  By induction, \eqref{eq:Caccio-recursive2} yields  $x_i \leq 2^{\kappa-i} y_0$ for $i =1,2, \cdots, \kappa-1$, and thus, summing over all such $i$, we infer 
\begin{align*}
\| \eta \nabla u \|_{L^2} \leq \sum_{i=1}^\kappa  \| \eta \nabla u_i \|_{L^2} &\leq 2^\kappa  \sqrt{\Lambda} \|\eta \nabla u \|^{\frac{1}{2}}_{L^2} \|u \nabla \eta \|^{\frac{1}{2}}_{L^2}\\
 &+ 2^\kappa \left( \sqrt{C_1} \|u \nabla \eta\|_{L^2} + \sqrt{C_2}  \|f \eta \|_{L^{2_*}} + \sqrt{C_3} \|g \eta \|_{L^{2}} \right),
  \end{align*}
where in the first inequality we used \eqref{4exist}. The theorem readily follows from another application of Young's inequality. This finishes the proof in the case $b+c \in L^{n,q}(\om')$, while the modifications to obtain the result the case $|b+c|^2 \in \mathcal{K}(\om')$ are identical to the ones presented in the proof of Theorem \ref{thm:subCaccioppoli} and are omitted.
\end{proof}

  \vv
	
The proofs of Theorems \ref{thm:subCaccioppoli} and \ref{thm:Caccioppoli2} can easily be adapted to prove the following Caccioppoli inequality at the boundary. 

\begin{theorem}[\bf Caccioppoli inequality at the boundary]\label{thm:boundCaccioppoli}
If $B_r$ is a ball such that $B_r \cap \om \neq \emptyset$, set $\om_r= B_r \cap \om$ and assume that  $u \in Y^{1,2}(\om_r)$ vanishing on $\d\om \cap B_r$ in the sense of definition \ref{def:vanish-subset}. Assume that $ f \in L^{2_*}(\om_r)$, $g \in L^2(\om_r)$ and  either \eqref{negativity} or \eqref{negativity2} holds.  If  either  $b+c \in L^{n,q}(\om_r)$, $q  \in [n, \infty)$, or  $|b+c|^2 \in \mathcal{K}(\om_r)$, and  $u$ is either a solution  or a non-negative subsolution of \eqref{eq:Lu=f-divg} in $\om_r$, then for any  non-negative function $\eta \in C_c^{\infty}(B_r)$ it holds
\begin{equation}\label{eq:bCacciop}
\|\eta \nabla u \|^2_{L^2(\om_r)}  \lesssim  \|u \nabla \eta \|^2_{L^2(\om_r)}+ \|f\eta\|_{L^{2_*}(\om_r)}^2 +\|g \eta\|_{L^2(\om_r)}^2,
\end{equation}
where the implicit constant depends only on $\lambda$,  $\Lambda$, and also either on  $C_{s,q}$ and $\|b+c\|_{L^{n,q}(\om_r)}$, for $q \geq n$, or  $C_s'$ and $\vartheta_{\om_r}(|b+c|^2,  r )$.
\end{theorem}

\begin{proof}
We follow the same strategy as before and apply  either Lemma $\ref{lem:main-splitting-Ln}$ to the function $u$ in $\om_r(x)$, for $p= n$, $q \geq n$, $h= b+c$, and $a = \frac{\lambda}{8C_{s,q}}$, where $C_{s,q}$ is the constant in \eqref{eq:Lorentz-Sobolev-emb}, or apply Lemma $\ref{lem:main-splitting-Kato}$ to the function $u$ in $\om_r(x)$, for $h= b+c$, and $a = \frac{\lambda}{8C_s'},$ where $C_s'$ is the constant in \eqref{eq:Kato-emb-grad-Br}.  Thus, we find $\om_i \subset \om_r(x)$ and  $u_i \in Y^{1,2}(\om_r)$ that vanishes on $B_r \cap \d \om$,  for $1 \leq i \leq \kappa$, satisfying  \eqref{1exist}--\eqref{6exist}. Using that the non-negative function $\eta^2 u_i$ is in $Y_0^{1,2}(\om_r(x))$, along with the fact that $u$ is either a solution  or a non-negative subsolution of \eqref{eq:Lu=f-divg} in $\om_r$, we may proceed as in the proofs of  Theorems \ref{thm:subCaccioppoli} and \ref{thm:Caccioppoli2} to obtain \eqref{eq:bCacciop}. We skip the details. \end{proof}

  \vv
  
\begin{remark}
We would like to note that if $b+c \in \mathcal{K}'(\om)$, we can dominate $\vartheta_{\om_r}(|b+c|^2, 2 r )$ by  $\vartheta_{\om}(|b+c|^2)$.
\end{remark}

\vvv

\subsection{Refined Caccioppoli inequality} \label{subs:refinedCaccio}

Let $ m=\inf_{\partial \om \cap B_{r} } u $  and $M=\sup_{\partial \om \cap B_{r} } u $ in the sense of Definition \ref{def:boundary-sup-inf}. Define 
\begin{align*}
\umm(x):= 
\begin{cases}
\inf(u(x),  m) & ,x \in \Omega\\
m&  ,x \in \R^n \setminus \Omega
\end{cases}
\end{align*}
and 
\begin{align*}
\ump(x):= 
\begin{cases}
\sup(u(x), M) & ,x \in \Omega\\
M&  ,x \in \R^n \setminus \Omega
\end{cases}
\end{align*}

  \vv

\begin{theorem}\label{thm:bdry-exp-subCaccioppoli1}
Let $B_r$ be a ball such that $\om_r=B_r \cap \om \neq \emptyset$ and assume  that either  $b+c \in L^{n,q}(\om_r)$, $q \in [n, \infty)$, or  $|b+c|^2 \in \mathcal{K}(\om_r)$.  We also assume  that one of the following holds:
\begin{enumerate}
\item $ \divv b + d \geq 0$, $\beta \in (-\infty, 0)$ and $u \in Y^{1,2}(\om_r)$ is a non-negative $L$-supersolution of \eqref{eq:Lu=f-divg}  in $\om_r$;
\item  $ \divv b + d \leq 0$, $\beta \in (0, \infty)$ and $u \in Y^{1,2}(\om_r)$ is a non-negative $L$-subsolution of \eqref{eq:Lu=f-divg}  in $\om_r$. 
\end{enumerate}
If we set 
\begin{equation*}
 \widehat{\om}_r=
\begin{dcases}
\om^m_r:= \{x\in \om_r: u<m\} &,\textup{in Case} \,(1),\\
\om^M_r:= \{x\in \om_r: u>M\}&,\textup{in Case} \, (2),
\end{dcases}
\end{equation*}
 and for $k>0$ we define
\begin{equation*}
\bu=
\begin{dcases}
\umm +k  &,\textup{in Case} \,(1),\\
\ump +k &,\textup{in Case} \, (2),
\end{dcases}
\quad\textup{and}\quad
\wt \om_r=
\begin{dcases}
\{ x \in \om_{r}: \nabla \umm (x) \neq 0 \}&,\textup{in Case} \,(1),\\
\{ x \in \om_{r}: \nabla \ump (x) \neq 0 \}&,\textup{in Case} \, (2),
\end{dcases}
\end{equation*}
then, there exist constants $C_0$, $C_1$, $C_2$ depending on $\beta$, such that for any  non-negative function $\eta \in C_c^{\infty}(B_r)$ we have
\begin{equation}\label{eq:bdry-exp-Caccio}
\|\eta\, \bu^{\frac{\beta-1}{2}}\, \nabla u \|^2_{L^2(\wt \om_r)}  \lesssim C_0 \|\bu^{\frac{\beta+1}{2}} \nabla \eta \|^2_{L^2( \widehat{\om}_r)}+ \int_{\widehat{\om}_r} \left(C_1|\bar f|+C_2 |\bar g|^2\right)\bu^{\beta+1}  \eta^2, 
\end{equation}
where  $\bar f =|f|/\bu$, $\bar g =|g|/\bu$, and the implicit constant depends on $\lambda$,  $\Lambda$,  and also either on  $C_{s,q}$ and $\|b+c\|_{L^{n,q}(\om_r)}$, for $q \geq n$, or  $C_s'$ and $\vartheta_{\om_r}(|b+c|^2,  r )$.  When $|\beta|>1$,  $C_0=|\beta+1|^{-2}$,  $C_1=|\beta+1|^{-1}$,   and $C_2=1+|\beta+1|^{-2}$, while when $ |\beta |<1$,  $C_0=4^{\kappa}|\beta|^{-2}$ and $C_1=C_2=2^{\kappa} |\beta|^{-1}$, where either $\kappa \leq 1+ \frac{1}{C|\beta|^n}\|b+c\|_{L^{n,q}(\om_r)}^n$ or $\kappa \leq 1 + 2 \,  a^{-2} \,{\rho_0^{2-n}} \|h\|_{L^1(\om_r)}$. In the case $\beta=-1$, $C_0=C_1=C_2=1$.
\end{theorem}

\begin{proof}

We first assume that $u$ is a non-negative supersolution of \eqref{eq:Lu=f-divg} and  $\beta < -1$. 

For $k>0$ we define the auxiliary function
\begin{align*}
w&= \bu^{\frac{\beta+1}{2}}-(m+k)^{\frac{\beta+1}{2}}.
\end{align*}
It is clear  that $w \in Y^{1,2}(\om_r)$ vanishing on $ \d \om \cap B_r$ and so, we can apply Lemma $\ref{lem:main-splitting-Ln}$ to $w$ and $\om_r$ with $p= n$, $q \geq n$, $h= b+c$, and $a = \frac{\lambda}{8C_{s,q}},$ where $C_{s,q}$ is the constant in Sobolev's inequality, to find $w_i \in Y^{1,2}(\om_r)$ that vanishes on $ \d \om \cap B_r$ and $\om_i \subset \wt \om_r$,  $1 \leq i \leq \kappa$, so that  \eqref{1exist}--\eqref{6exist} hold. 

Since  $w_i$ vanishes on $ \d \om \cap B_r$. there is a sequence $\phi_k \in C^\infty_c(\bar \om \setminus (\d \om \cap B_r))$ such that $\phi_k \to w_i$ in $Y^{1,2}(\om)$. Thus, the sequence $\eta^2\phi_k \in C^\infty_c(\om_r)$ converges to $\eta^2 w_i$ in $Y^{1,2}(\om_r)$, which implies that  $\eta^2 w_i \in Y^{1,2}_0(\om_r)$. Note also that, by  \eqref{3exist}, $\eta^2 w_i$ is non-negative.  Thus,  for $i=1,2,\dots \kappa$,
\begin{multline}
\lambda \int_{\om_i}|\nabla w|^2 \,\eta^2 = \lambda \int_{\om_r} |\nabla w_i|^2  \,\eta^2 \\
\leq \int_{\om_r} A \nabla w_i \nabla w_i  \,\eta^2  
= \frac{\beta+1}{2}  \int_{\om_r} A \nabla u \nabla w_i \, \bu^{\frac{\beta-1}{2}} \eta^2 \\
= \frac{\beta+1}{2}  \left(\int_{\om_r} A \nabla u \nabla (w_i \, \bu^{\frac{\beta-1}{2}} \eta^2) - 2 \int_{\om_r} A \nabla u \nabla \eta \, \eta w_i  \bu^{\frac{\beta-1}{2}}  \right) \\
- \frac{\beta+1}{2}  \left(\int_{\om_r} A \nabla u \nabla\bu^{\frac{\beta-1}{2}}  w_i \, \eta^2 \right) 
=: \frac{\beta+1}{2} \left(J_1-J_2-J_3\right). \label{eq:bdryCaccio1-split}
\end{multline}

Let us point out that
\begin{align}\label{eq:bdryCaccio1-wi<w<bu}
0 \leq w_i &\leq w \leq  \bu^{\frac{\beta+1}{2}} 
\end{align}
and
\begin{align}\label{eq:bdryCaccio1-support}
\nabla \bu \,{\bf 1}_{\om_r}= \nabla u\, {\bf 1}_{\om^m_r} \quad& \textup{and} \quad  \{x\in \om_r: w_i \neq 0\}\subset \{x\in \om_r: w \neq 0\}= \om^m_r.
\end{align}

Recalling that  $\beta<-1$ and using \eqref{eq:bdryCaccio1-support}, \eqref{eqelliptic1}, and that $\bu> 0$, we get that
\begin{align}\label{eq:bdryCaccio1-J1}
J_3 &={\frac{\beta-1}{2}} \int_{\om^m_r} A \nabla u \nabla u \, \bu^{\frac{\beta-3}{2}} \eta^2 \leq  \lambda {\frac{\beta-1}{2}}\int_{\om^m_r} |\nabla u|^2 \, \bu^{\frac{\beta-3}{2}} \eta^2 \leq 0,
\end{align}
and thus,  $- \frac{\beta+1}{2} J_3 \leq 0$.  
Moreover, by \eqref{eqelliptic2}, H\"older's inequality, \eqref{eq:bdryCaccio1-wi<w<bu}, and \eqref{eq:bdryCaccio1-support}, 
\begin{align}\label{eq:bdryCaccio1-J2}
\left|J_2 \right| &\leq 2 \Lambda \|\eta \nabla \bu^{\frac{\beta+1}{2}} \|_{L^2(\om^m_r)} \| w_i \nabla \eta\|_{L^2(\om^m_r)}\\
&\leq 2 \Lambda \|\eta \nabla \bu^{\frac{\beta+1}{2}} \|_{L^2(\om^m_r)} \| \bu^{\frac{\beta+1}{2}} \nabla \eta\|_{L^2(\om^m_r)}.\notag
\end{align}
Since $u$ is a supersolution of \eqref{eq:Lu=f-divg}, $\beta+1<0$, and $\divv b - d \geq 0$, we obtain
\begin{align}\label{eq:bdryCaccio1-J1}
J_1 &\geq  \int_{\om_r} (b+c) \, \nabla u \, w_i \, \bu^{\frac{\beta-1}{2}} \eta^2 + \int_{\om_r} f w_i \, \bu^{\frac{\beta-1}{2}} \eta^2  +\int_{\om_r} g \nabla\left(w_i \, \bu^{\frac{\beta-1}{2}} \eta^2 \right)\\
 &=: I_1 +I_2 +I_3,\notag
\end{align}
and so $\frac{\beta+1}{2} J_1 \leq \frac{\beta+1}{2} (I_1 +I_2 +I_3 )$.
As
$$
\nabla\left(w_i \, \bu^{\frac{\beta-1}{2}} \,   \eta^2 \right)= \nabla w_i \, \bu^{\frac{\beta-1}{2}} \, \eta^2+ 2 \nabla \eta \,w_i \,  \eta \,  \bu^{\frac{\beta-1}{2}}+  \nabla  \bu^{\frac{\beta-1}{2}} \,  w_i\eta^2,
$$
we may write $I_3$ as the sum of  three  integrals $I_{31}, I_{32}, I_{33}$ that correspond to the terms on the right hand-side of the latter equality.
So, by Young's inequality (for $\ve$ small enough to be chosen) along with \eqref{eq:bdryCaccio1-wi<w<bu} and \eqref{eq:bdryCaccio1-support}, we get
\begin{align}
\frac{|\beta+1|}{2}|I_{31}| &\leq \ve  \| \nabla w_i \, \eta \|_{L^2(\om_r)}^2 + \frac{|\beta+1|^2}{16 \ve} \int_{\om^m_r} |g|^2 \bu^{\beta-1} \eta^2, \label{eq:bdryCaccio1-I31} \\ 
\frac{|\beta+1|}{2}|I_{32}| &\leq   \| \bu^{\frac{\beta+1}{2}} \nabla \eta \|_{L^2(\om^m_r)}^2 +  \frac{|\beta+1|^2}{ 4} \int_{\om^m_r} |g|^2 \bu^{\beta-1} \eta^2. \label{eq:bdryCaccio1-I32} \\
\frac{|\beta+1|}{2} |I_{33}|&\leq \ve \frac{|\beta+1|^2}{4}  \| \bu^{\frac{\beta-1}{2}}\, \nabla \bu \,\eta \|_{L^2(\om^m_r)}^2 + \frac{|\beta-1|^2}{16\, \ve} \int_{\om^m_r} |g|^2 \bu^{\beta-1} \eta^2\label{eq:bdryCaccio1-I33} \\
|I_2| &\leq \int_{\om^m_r} |f| \bu^\beta \eta^2. \label{eq:bdryCaccio1-I2} 
\end{align}

Moreover, by \eqref{5exist},
\begin{align}\label{eq:bdryCaccio1-I2} 
\frac{\beta+1}{2} I_1 &=\int_{\om_r}(b+c) \nabla w \,w_i \,  \eta^2 \\
&= \int_{\om_i} (b+c) \nabla w_i \, w_i \,  \eta^2 + \sum_{j=1}^{i-1} \int_{\om_j} (b+c) \nabla w_j \,w_i \,  \eta^2=:I_{1}^i +\sum_{j=1}^{i-1} I_1^j.\notag
\end{align}
If we apply \eqref{eq:Lorentz-Sobolev-emb} and Young's inequality,
\begin{align}\label{eq:bdryCaccio1-I2i} 
|I_1^i| &\leq C_{s,q} \|b+c\|_{L^{n,q}(\om_i)}\|\eta \nabla w_i\|_{L^2(\om_r)} \| \nabla( \eta w_i)\|_{L^2(\om_r)}\\
& \leq \frac{3 a C_{s,q}}{2} \|\eta \nabla w_i\|_{L^2(\om_r)}^2 +  \frac{a C_{s,q}}{2} \| w_i \nabla \eta \|_{L^2(\om^m_r)}^2.\notag
\end{align}
Similarly,
\begin{align}\label{eq:bdryCaccio1-I2j} 
&\sum_{j=1}^{i-1} |I_1^j| \leq C_{s,q} \|b+c\|_{L^{n,q}(\om_i)} \| \nabla( \eta w_i)\|_{L^2(\om_r^m)} \sum_{j=1}^{i-1}\|\eta \nabla w_j\|_{L^2(\om_r)}\\
& \leq  a C_{s,q} \|\eta \nabla w_i\|_{L^2(\om_r)}^2 +  a C_{s,q} \| w_i \nabla \eta \|_{L^2(\om_r^m)}^2 +  \frac{a C_{s,q}}{2} \left(\sum_{j=1}^{i-1}\|\eta \nabla w_j\|_{L^2(\om_r)}\right)^2.\notag
\end{align}

Let us set 
$$
x_0=\|  \eta \bu^{\frac{\beta-1}{2}}  \nabla \bu \|_{L^{2}( \om_r^m)}, \,\, x_j=\|  \eta  \nabla w_j \|_{L^{2}(\om_r)}, \,\,y_0=\|  \bu^{\frac{\beta+1}{2}}  \nabla \eta\|_{L^{2}( \om_r^m)},
$$
 and also,  if $\gamma_0:=|\beta+1|/2$,   set
\begin{multline*}
z_0=\| |f|^{\frac{1}{2}}  \bu^{\frac{\beta}{2}} \eta \|_{L^2({\om_r^m})}, \qquad z_1=  \| |g| \bu^\frac{\beta-1}{2} \eta \|_{L^2({\om_r^m})},\,\,\textup{and}\\
C(\ve,\gamma_0):=\left[ \left( (4\ve)^{-1} + 1 \right) \gamma_0^2 + (4\ve)^{-1} (1+\gamma_0)^2 \right]^\frac{1}{2}.
\end{multline*} 
 Then, using this notation, $|\beta -1| /2\leq 1+ \gamma_0$,  and choosing $\alpha$  small enough, depending on $\lambda$, $\Lambda$, $\|b+c\|_{L^n(\om)}$, and $C_{s,q}$, we can collect the inequalities \eqref{eq:bdryCaccio1-split}-\eqref{eq:bdryCaccio1-I2j} and find a constant $C_0$ (depending on $\lambda$, $\Lambda$ and $C_{s,q}$) so that
 $$
 x_i \leq  C_0 ( \sqrt{\gamma_0}  z_0 + C(\ve,\gamma_0) z_1 +  \sqrt{\ve}\,\gamma_0 \,x_0 + y_0) + \sum_{j=1}^{i-1} x_j.
 $$
By the  induction argument that appeared in the proof of Theorem \ref{thm:subCaccioppoli} and \eqref{eq:bdryCaccio1-support}, we can show that
$$
\gamma_0\, x_0=\|\eta\, \nabla w \|^2_{L^2(\om_r)}  \leq C_1 ( \sqrt{\gamma_0}  z_0 + C(\ve,\gamma_0) z_1 +  \sqrt{\ve}\,\gamma_0 \ \,x_0 + y_0),
$$
where $C_1$ depends on $\lambda$, $\Lambda$, $\|b+c\|_{L^{n,q}(\om)}$ and $C_{s,q}$. We may choose $\ve$ small enough compared to $C_1^{-2}$ and use Young's inequality with $\ve$ to deduce
$$
\gamma_0 \,x_0 \leq C_2\, ( y_0 + \sqrt{\gamma_0}  z_0 + (1+\gamma_0^2)^{1/2} z_1)
$$
in order to show \eqref{eq:bdry-exp-Caccio}. The details are omitted.

\vv

We turn our attention to the case that $u$ is a non-negative supersolution of \eqref{eq:Lu=f-divg} and $\beta \in [-1,0)$. For $k>0$ we define the auxiliary function
\begin{align*}
w&= \bu^{\beta}-(m+k)^{\beta}.
\end{align*}
Since  $w \in Y^{1,2}(\om)$ and vanishes on $\d \om \cap B_r$, we apply Lemma \ref{lem:main-splitting-Ln} as in the previous case to $w$ and $\om_r$, for $p= n$, $h= b+c$, and $a$ small enough  depending on $\lambda, \beta,  C_{s,q}$ (to be picked later), to find $w_i \in Y^{1,2}(\om)$ that also vanishes on $\d \om \cap B_r$ and $\om_i \subset \wt \om_r$,  $1 \leq i \leq m$, satisfying  \eqref{1exist}--\eqref{6exist}. By $\eqref{3exist}$ we see that $\eta^2 w_i \in Y^{1,2}_0(\om)$ is non-negative and we may use it as a test function. Therefore,
\begin{align}
\int_{\om_r} f (\eta^2 w_i) &+ \int_{\om_r} g \nabla(\eta^2 w_i) \notag\\
&\leq \int_{\om_r}A \nabla u \nabla(\eta^2 w_i) + bu\nabla (\eta^2 w_i) -c\nabla u(\eta^2 w_i)-du (\eta^2w_i)  \notag\\
&= \int_{\om_r}A \nabla u \nabla(\eta^2 w_i) + b \nabla (\eta^2 u w_i) -(b+c)  \nabla u \eta^2w_i-d \eta^2 u w_i \notag\\
&\leq \int_{\om_r}A \nabla u \nabla(\eta^2 w_i) -(b+c)  \nabla u \eta^2w_i,\label{eq:bdryCaccio1-split-bis}
\end{align}
where in the last inequality we used \eqref{negativity}.
 
At this point let us  recall  \eqref{eq:bdryCaccio1-support} and also record that
\begin{align}\label{eq:bdryCaccio1-wi<w<bu-bis}
0 \leq w_i &\leq w \leq  \bu^{\beta} 
\end{align}
and
\begin{equation}\label{eq:bdryCaccio1-nablawi-bis}
\nabla w_i = \beta \bu^{\beta-1} \nabla u \, {\bf 1}_{\om_i}.
\end{equation}
Therefore, by  \eqref{eq:bdryCaccio1-nablawi-bis} and $\beta<0$, \eqref{eq:bdryCaccio1-split-bis} can be written
\begin{align}\label{eq:bdryCaccio2-split}
\lambda  |\beta| \|\eta\bu^\frac{\beta-1}{2}& \nabla u\|^2_{L^2(\om_i)} \leq  |\beta|\, \int_{\om_i} A \nabla u  \cdot \nabla u \eta^2 \bu^{\beta-1} \leq 2  \int_{\om_r} A \nabla u \cdot \nabla \eta \,w_i \eta  \\
&- \int_{\om_r} (b+c)  \nabla u \eta^2w_i - \int_{\om_r} f (\eta^2 w_i) - \int_{\om_r}  g \nabla(\eta^2 w_i) = \sum_{i=1}^4 I_i.\notag
\end{align}

We apply H\"older's inequality along with \eqref{eq:bdryCaccio1-support} and  \eqref{eq:bdryCaccio1-wi<w<bu-bis} to get
\begin{align}\label{eq:bdryCaccio2-I1}
|I_1| \leq 2  \Lambda \| \eta \bu^\frac{\beta-1}{2} \nabla u\|_{L^2( \om_r^m)} \| \bu^\frac{\beta+1}{2}  \nabla \eta\|_{L^2( \om_r^m)}.
\end{align}
By Young's inequality,  \eqref{eq:bdryCaccio1-wi<w<bu-bis},  and \eqref{eq:bdryCaccio1-nablawi-bis},  it is easy to see that
\begin{align}\label{eq:bdryCaccio2-I34}
|I_3| + |I_4| \leq \int_{\om_r^m} |f| \bu^{\beta} \eta^2  &+ \left(1+ \frac{|\beta| }{4\ve} \right)  \int_{ \om_r^m} |g|^2 \bu^{\beta-1} \eta^2 \\
&+ |\beta| \, \ve\, \int_{\om_i} \bu^{\beta-1} |\nabla u|^2 \eta^2  + \int_{\om_r^m} \bu^{\beta+1} | \nabla \eta|^2.\notag
\end{align}
It only remains to handle $I_2$. At this point we cannot use \eqref{4exist} or \eqref{5exist} as in previous arguments. The reason why is that we do not have $u$ and $u_i$ but rather two different functions $u$ and $w_i$. Although, we can recall that  $\{x\in \om_r: w_i \neq 0\} = \cup_{j=1}^i \om_j$ and thus,  using \eqref{eq:Lorentz-Sobolev-emb}, \eqref{eq:bdryCaccio1-support},  \eqref{eq:bdryCaccio1-wi<w<bu-bis}, $\|b+c\|_{L^{n,q}(\om_j)} \leq a$ for any $j \in \{1,2, \cdots m\}$, and $w_i \bu^{\frac{1-\beta}{2}} \eta \in Y^{1,2}_0(\om_r)$, we get 
\begin{align}\label{eq:bdryCaccio2-I2i}
| I_{2} | &\leq  C_{s,q}   \|b+c\|_{L^{n,q}(\om_i)}  \| \eta \bu^{\frac{\beta-1}{2}} \nabla u\|_{L^2(\cup_{j=1}^i \om_j)}  \| \nabla(w_i \bu^{\frac{1-\beta}{2}} \eta)\|_{L^{2}( \om_r)}\\  
&\leq a C_{s,q}   \| \eta \bu^{\frac{\beta-1}{2}}  \nabla u\|_{L^2(\cup_{j=1}^i \om_j)}  \| \nabla (w_i \bu^{\frac{1-\beta}{2}}\eta)\|_{L^{2}(\om_r)}.\notag
\end{align}
Note that
$$
 \nabla (w_i \bu^{\frac{1-\beta}{2}} \eta)\, {\bf 1}_{ \om_r} =\beta  \eta \bu^{\frac{\beta-1}{2}}  \nabla u \, {\bf 1}_{ \om_i} +w_i \bu^{\frac{1-\beta}{2}}  \nabla \eta+ {\frac{1-\beta}{2}} w_i \bu^{-\frac{\beta+1}{2}} \eta \nabla u.
$$
Also, for $\beta \in [-1, 0)$, it holds $\frac{\beta-1}{2\beta}>0$ and $\frac{\beta+1}{2}>0$. Thus, by \eqref{eq:bdryCaccio1-wi<w<bu-bis},
$$
 w_i \bu^{-\frac{\beta+1}{2}} \leq w_i^{\frac{\beta-1}{2\beta}} \leq \bu^{\frac{\beta-1}{2}}{\bf 1}_{\cup_{j=1}^i \om_j} \quad \textup{and}\quad w_i \bu^{\frac{1-\beta}{2}} \leq \bu^\beta \bu^{\frac{1-\beta}{2}}{\bf 1}_{\cup_{j=1}^i \om_j} \leq  \bu^{\frac{\beta+1}{2}} {\bf 1}_{\om_r^m},
$$
which, in turn, implies that 
\begin{align}\label{eq:bdryCaccio2-grad-wibueta-split}
\| \nabla (w_i \bu^{\frac{1-\beta}{2}} \eta)\|_{L^{2}(\om_r)} \leq |\beta| \|  \eta \bu^{\frac{\beta-1}{2}}  \nabla \bu \|_{L^{2}(\om_i)} &+\|  \bu^{\frac{\beta+1}{2}}  \nabla \eta\|_{L^{2}(\om^m_r)} \\ &+\frac{1-\beta}{2}\|  \eta \bu^{\frac{\beta-1}{2}}  \nabla \bu \|_{L^{2}(\cup_{j=1}^i \om_j)}.\notag
\end{align}

Set now
\begin{multline*}
x_0=\|  \eta \bu^{\frac{\beta-1}{2}}  \nabla \bu \|_{L^{2}(\om_r^m)}, \,\, x_j=\|  \eta \bu^{\frac{\beta-1}{2}}  \nabla \bu \|_{L^{2}(\om_j)}, \,\,y_0=\|  \bu^{\frac{\beta+1}{2}}  \nabla \eta\|_{L^{2}(\om^m_r)},\\
z_0=  \| |f|^{1/2}  \eta \bu^{\frac{\beta}{2}}  \|_{L^{2}(\om_r^m)} \quad \textup{and} \quad z_1 =    \| |g|  \eta \bu^{\frac{\beta-1}{2}}  \|_{L^{2}(\om_r^m)}.
\end{multline*}
With this notation, we can write
\begin{align*}
 &\|  \eta \bu^{\frac{\beta-1}{2}} \nabla \bu \|^2 _{L^{2}(\cup_{j=1}^i \om_j)}= x_i^2+ \sum_{j=1}^{i-1} x_j^2\\
 &\| \nabla (w_i \bu^{\frac{1-\beta}{2}} \eta)\|_{L^{2}(\om_r)} \leq |\beta| x_i + y_0 + \frac{1+|\beta|}{2} \Big(x_i^2+ \sum_{j=1}^{i-1} x_j^2\Big)^{1/2},
\end{align*}
which, in combination with  inequalities  \eqref{eqelliptic1} and  \eqref{eq:bdryCaccio2-split}-\eqref{eq:bdryCaccio2-grad-wibueta-split},  and $|\beta| \leq 1$, implies
\begin{align*}
|\beta| \lambda x_i^2&\leq 2\Lambda x_0 y_0 + a C_{s,q} \Big(x_i^2+ \sum_{j=1}^{i-1} x_j^2\Big)^{1/2} \Big(  |\beta| x_i + y_0 + \Big(x_i^2+ \sum_{j=1}^{i-1} x_j^2\Big)^{1/2}\Big)\\
&+ \Big(  |\beta| \ve \, x_0^2+ y_0^2 + z_0^2+ \Big(1+ \frac{|\beta| }{4\ve} \Big) z_1^2\Big).
\end{align*}
Therefore, if we choose $\alpha$ small enough (depending linearly on $|\beta|$), by Young's inequality, we can find a positive constant $C_0$ depending only on $\lambda, \Lambda$, and $C_{s,q}$ so that
$$
x_i \leq \frac{C_0}{\sqrt{ |\beta}|} \left( ( x_0 y_0)^{1/2} + \sqrt{|\beta| \ve} \, x_0 +(1+\sqrt{|\beta|} ) y_0 +  z_0 +(1+\sqrt{|\beta|} ) z_1 \right)+  \sum_{j=1}^{i-1} x_j,
$$
The proof of \eqref{eq:bdry-exp-Caccio} is concluded by the same iteration argument as in the proof of Theorem \ref{thm:subCaccioppoli} along with the facts that $\cup_{i=1}^\kappa \om_i= \wt \om_r$ and $|\beta|<1$ obtaining
$$
x_0\leq \frac{C_0 2^{\kappa}}{\sqrt{ |\beta}|}\left( ( x_0 y_0)^{1/2} + \sqrt{|\beta| \ve} \, x_0 +  y_0 +  z_0 + 2 z_1 \right),
$$
where $\kappa \leq 1+ \frac{1}{C|\beta|^{n}}\|b+c\|_{L^{n,q}(\om_r)}^n$.  By Young's inequality and if we choose $\ve$ small enough (depending on $\lambda$, $\Lambda$,  $C_0$, and $\kappa$), we obtain \eqref{eq:bdry-exp-Caccio}. The case  $\beta >0$ and $u$  positive subsolution of \eqref{eq:Lu=f-divg} is almost identical and we will not repeat it.

The same reasoning  shows \eqref{eq:bdry-exp-Caccio} when $|b+c|^2 \in \mathcal{K}(\om_r)$ if we use Lemma \ref{lem:main-splitting-Kato}.  The only difference lies on the manipulation of the terms that include $b+c$ and a similar argument can be found  at the end of the proof of Theorem \ref{thm:subCaccioppoli}. The details are omitted.
\end{proof}

  \vv

In fact, if we incorporate $-\divv(b u)$ and $d u$ into the interior data, the same proof  gives the following theorem:
\begin{theorem}\label{thm:bdry-exp-subCaccioppoli1-bis}
If we use the same notation as in Theorem \ref{thm:bdry-exp-subCaccioppoli1} and  either $c \in L^{n,q}(\om_r)$, for $q \in [n, \infty)$ or  $|c|^2 \in \mathcal{K}(\om_r)$, then for any  non-negative function $\eta \in C_c^{\infty}(B_r)$, we have
\begin{equation}\label{eq:bdry-exp-Caccio1-bis}
\|\eta\, \bu^{\frac{\beta-1}{2}}\, \nabla u \|^2_{L^2(\wt \om_r)}  \lesssim C_0 \|\bu^{\frac{\beta+1}{2}} \nabla \eta \|^2_{L^2( \widehat{\om}_r)}+ \int_{\widehat{\om}_r} (C_1 \bar f+C_1 |d|+ C_2\bar g^2 + C_2 |b|^2) \bu^{\beta+1} \eta^2, 
\end{equation}
where $\bar f =|f|/\bu$, $\bar g =|g|/\bu$, and $C_0$, $C_1$, and $C_2$ are the constants given in Theorem \ref{thm:bdry-exp-subCaccioppoli1}. The implicit constant depends on $\lambda$, $\Lambda$, and either on $C_{s,q}$ and  $\|c\|_{L^{n,q}(\om_r)}$, or $C_s'$ and  $\vartheta_{\om_r}(|c|^2, r)$.
\end{theorem}




  \vv

The analogue of Theorem  \ref{thm:bdry-exp-subCaccioppoli1} for the case $-\divv c +d \geq 0$ (or $-\divv c +d \leq 0$) will be a lot easier to prove, as one does not need to handle either the $L^{n,q}$-norm of $b+c$ or the $\mathcal{K}$-norm of $|b+c|^2$ in a delicate way as before. Instead, we will incorporate  $|b+c|^2$ into the interior data side (as in Theorem \ref{thm:bdry-exp-subCaccioppoli1-bis}). It may look surprising bearing in mind the special case $\beta=1$ we proved in Theorem \ref{thm:Caccioppoli2}, but \eqref{eq:bdry-exp-Caccio} cannot hold in this case. The reason is that it is the main ingredient of the proof of local boundedness and weak Harnack inequality and, by Example \eqref{ex:c1}, we know that if $b+c$ does not have any additional hypothesis, solutions may  not be locally bounded.

\begin{theorem}\label{thm:bdry-exp-subCaccioppoli2}
If we replace $ \divv b + d \geq 0$  (or $\divv b + d \leq 0$) with $ -\divv c + d \geq 0$ (or $-\divv c + d  \leq 0$) in the assumptions of Theorem \ref{thm:bdry-exp-subCaccioppoli1} and use the same notation, we can find constants $C_0$, $C_1$, $C_2$ depending on $\beta$, such that for any  non-negative function $\eta \in C_c^{\infty}(B_r)$ we have
\begin{equation}\label{eq:bdry-exp-Caccio2}
\|\eta\, \bu^{\frac{\beta-1}{2}}\, \nabla u \|^2_{L^2(\wt \om_r)}  \lesssim C_0 \|\bu^{\frac{\beta+1}{2}} \nabla \eta \|^2_{L^2(\widehat{\om}_r)}+ \int_{\widehat{\om}_r} (C_1 \bar f+ C_2\bar g^2 + C_2 |b+c|^2) \bu^{\beta+1} \eta^2, 
\end{equation}
where $\bar f =|f|/\bu$, $\bar g =|g|/\bu$, and the implicit constant depends $\lambda$ and $\Lambda$. When $|\beta|>1$,  $C_0=|\beta+1|^{-2}$,  $C_1=|\beta+1|^{-1}$,   and $C_2=1+|\beta+1|^{-2}$,  while when $ |\beta |<1$,  $C_0=|\beta|^{-2}$ and $C_1=C_2=|\beta|^{-1}$. When $\beta=-1$, $C_0=C_1=C_2=1$.
\end{theorem}

\begin{proof}
We will only give a sketch of the proof. Let us assume that $\beta \in [-1, 0)$. For $k>0$ we define the auxiliary function
\begin{align*}
w&= \bu^{\beta}-(m+k)^{\beta}.
\end{align*}
Since $\eta^2 w \in Y^{1,2}_0(\om_r)$, arguing as in Case $\beta>-1$ in  the proof of the previous theorem and using $-\divv c +d \geq 0$, we get
\begin{align*}
\int_{\om_r} f (\eta^2 w) + \int_{\om_r}  g \nabla(\eta^2 w) \leq \int_{\om_r}A \nabla u \nabla(\eta^2 w) -(b+c)  u  \nabla(\eta^2w).
\end{align*}
Because $\beta<0$ and $ \{x \in \om_r: w\neq 0 \} = \om_r^m$, the latter inequality can be written as
\begin{align}\label{eq:bdryCaccio2-split-cd}
|\beta|\, \int_{\om_r} A \nabla u  \cdot \nabla u \eta^2 \bu^{\beta-1} &\leq 2  \int_{\om_r} A \nabla u \cdot \nabla \eta \,w \eta - \int_{\om_r} (b+c)  u \nabla (\eta^2w) \\
&- \int_{\om_r} f (\eta^2 w) - \int_{\om_r}  g \nabla(\eta^2 w) = \sum_{i=1}^4 I_i.\notag
\end{align}
Note that if we use $0 \leq u \leq \bu$, then $I_1$, $I_3$ and $I_4$ can be bounded as in \eqref{eq:bdryCaccio2-I1} and \eqref{eq:bdryCaccio2-I34}. So, it only remains to handle $I_2$. But as we do not need to use Lemma \ref{lem:main-splitting-Ln} it will be fairly easy to do so.  Indeed, 
$$
I_2= -2 \int_{\om_r^m} (b+c) \nabla \eta w u \eta +|\beta| \int_{\om_r^m} (b+c) \nabla u \eta^2 \bu^{\beta -1} u,
$$
which, in light of Young's inequality with $\ve$ small (to be picked), $w \leq \bu^\beta {\bf 1}_{\om_r^m}$ and  $\beta \in [-1,0)$, implies
$$
|I_2| \leq (1+ |\beta| (4\ve)^{-1} ) \int_{\om_r^m} |b+c|^2 \bu^{\beta+1} \eta^2 +  \int_{\om_r^m} |\nabla \eta|^2 \bu^{\beta+1} + \ve |\beta| \int_{\om_r^m} |\nabla u|^2 \bu^{\beta -1} \eta^2.
$$
If we choose $\ve$ small enough we conclude our result. We may handle the case $\beta <-1$ and $\beta \geq 0$ for subsolutions in a similar fashion adapting the argument in the proof of Theorem \ref{thm:bdry-exp-subCaccioppoli1}. We omit the routine details. 
\end{proof}

  \vv
  
Moreover, the proofs of Theorems \ref{thm:bdry-exp-subCaccioppoli1}, \ref{thm:bdry-exp-subCaccioppoli1-bis}, and \ref{thm:bdry-exp-subCaccioppoli2} can be easily  adapted to get a refined version of Theorems  \ref{thm:subCaccioppoli} and  \ref{thm:Caccioppoli2}. We only state the first one.

\begin{theorem}\label{thm:int-exp-subCaccioppoli1}
Let $B_r$ be a ball of radius $r>0$ so that $\overline B_{r} \subset \om$ and assume  that either  $b+c \in L^{n,q}(B_r)$, $q  \in [n, \infty)$, or  $|b+c|^2 \in \mathcal{K}(B_r)$. If  $u \in Y^{1,2}(B_r)$  and one of the following holds:
\begin{itemize}
\item[(1)] $ \divv b + d \leq 0$ and u is  $L$-subsolution  in $B_r$ and $\beta \in (0, +\infty)$;
\item[(2)] $ \divv b + d \leq 0$ and u is  $L$-supersolution in $B_r$ and $\beta \in (0, +\infty)$;
\item[(3)] $ \divv b + d \geq 0$ and u is a non-negative $L$-supersolution in $B_r$ and $\beta \in (-\infty, 0)$.
\end{itemize}
For $k>0$, we set
\begin{align}\label{eq:int-ref-caccio-w}
\bu=
\begin{dcases}
u^+ + k  \quad &,\textup{in Case} \,(1),\\
u^- + k  \quad &,\textup{in Case} \, (2),\\
u + k  \quad  &,\textup{in Case} \, (3).\\
\end{dcases}
\end{align}
Then, there exist constants $C_0$, $C_1$, $C_2$ depending on $\beta$, such that for any  non-negative function $\eta \in C_c^{\infty}(B_r)$ we have
\begin{equation}\label{eq:inter-exp-Caccio}
\|\eta\, \bu^{\frac{\beta-1}{2}}\, \nabla u \|^2_{L^2(B_r)}  \lesssim C_0 \|\bu^{\frac{\beta+1}{2}} \nabla \eta \|^2_{L^2( B_r)}+ \int_{B_r} \left(C_1|\bar f|+C_2 |\bar g|^2\right)\bu^{\beta+1}  \eta^2, 
\end{equation}
where  $\bar f =|f|/\bu$, $\bar g =|g|/\bu$, and the implicit constant depends on $\lambda$,  $\Lambda$,  and also either on  $C_{s,q}$ and $\|b+c\|_{L^{n,q}(B_r)}$, or  $C_s'$ and $\vartheta_{B_r}(|b+c|^2,  r )$.  When $|\beta|>1$,  $C_0=|\beta+1|^{-2}$,  $C_1=|\beta+1|^{-1}$,   and $C_2=1+|\beta+1|^{-2}$, while when $ |\beta |<1$,  $C_0=4^{\kappa}|\beta|^{-2}$ and $C_1=C_2=2^{\kappa} |\beta|^{-1}$, where either $\kappa \leq 1+ \frac{1}{C|\beta|^n}\|b+c\|_{L^{n,q}(B_r)}^n$ or $\kappa \leq 1 + 2 \,  a^{-2} \,{\rho_0^{2-n}} \|h\|_{L^1(B_r)}$. In the case $\beta=-1$, $C_0=C_1=C_2=1$.
\end{theorem}

\vvv

\section{Local estimates and  regularity of solutions up to the boundary}\label{sec:local-est}

In this part we will present  the iterating method of Moser to obtain the following results: 
\begin{itemize}
\item Local boundedness for subsolutions;
\item Weak Harnack inequality for supersolutions;
\item H\"older continuity in the interior for solutions;
\item A Wiener criterion for continuity of solutions at the boundary. 
\end{itemize} 

\vv

\subsection{Local boundedness and weak Harnack inequality} \label{subs:a priori}

Denote $\om_{r_0}=B_{r_0}  \cap \om \neq \emptyset$,  where $r_0 \in (0, \infty]$, and let $f \in \mathcal{K}(\om_{r_0})$ and $|g|^2 \in \mathcal{K}(\om_{r_0})$. Set $$\gamma:= \beta+1$$ and
\begin{align}\label{eq:Moser-k(r)}
k(r):= \vartheta_{\om_{r_0}}(|f|,r) + \vartheta_{\om_{r_0}}(|g|^2, r)^{1/2}, \quad \textup{for any}\,\,r \in (0, {r_0}].
\end{align}
Define 
\begin{equation}\label{eq:Moser-w}
w=
\begin{dcases}
\bu^{\frac{\beta+1}{2}},  \quad \textup{if} \,\, \beta\neq -1\\
\log \bu,  \quad \textup{if} \,\, \beta =-1,
\end{dcases}
\end{equation}
where 
$\bu$ is either the one given in Theorem \ref{thm:bdry-exp-subCaccioppoli1} or in Theorem \ref{thm:int-exp-subCaccioppoli1}, with $$
k=k(r).
$$
Here $B_r$ is a ball of radius $r \in(0, r_0]$ which is either centered at the boundary (as in Theorem \ref{thm:bdry-exp-subCaccioppoli1}) or  such that $B_r \subset \om$ (as in Theorem \ref{thm:int-exp-subCaccioppoli1}). 
We will handle both cases simultaneously and  it should be understood from the context what kind of balls we are referring to.  Set
\begin{equation}
\tilde f= \frac{|f|}{k(r)}, \,\,\,  \tilde g = \frac{ | g|}{k(r)},\,\,\, \textup{and} \,\,\, V= \tilde f + \tilde g^2.
\end{equation}
Notice that for  $k=k(r)$,  we have that $|\bar f| \leq |\tilde f|$  and $|\bar g|\leq |\tilde g|$ and so \eqref{eq:bdry-exp-Caccio},  \eqref{eq:bdry-exp-Caccio1-bis}, \eqref{eq:bdry-exp-Caccio2}, and \eqref{eq:inter-exp-Caccio} hold for $\tilde f$ and $\tilde g$ instead of $\bar f$ and $\bar g$. Moreover,  
\begin{align}\label{eq:v<f+g}
\vartheta_{\om_{r_0}}(V,r)= \frac{1}{k(r)} \sup_{x \in \R^n}  \int_{B(x,r) \cap \om_{r_0}}& \frac{|f(y)|}{|x-y|^{n-2}} \,dy \\
&+\frac{1}{k(r)^2} \sup_{x \in \R^n}  \int_{B(x,r) \cap \om_{r_0}} \frac{|g(y)|^2}{|x-y|^{n-2}} \,dy\leq 2.\notag
\end{align}

\vv

\begin{lemma}\label{lem:caccio-w-moser}
Assume that $B_r$ be a ball such that $\om_r=B_r \cap \om \neq \emptyset$, $r \leq r_0$, and that either  $b+c \in L^{n,q}(\om_{r_0})$, $q \in [ n, \infty)$, or $|b+c|^2 \in \mathcal{K}(\om_{r_0})$.  If $w$ is defined in \eqref{eq:Moser-w},  and $\eta \in C^\infty_c(B_r)$ is non-negative, then the following hold:
If $|\beta|>1$, there exist constants $c'_3>1$ and $c'_4 \in(0,1)$ so that for any $0<\epsilon \leq 1$,
\begin{equation}\label{eq:iterat-moser-1}
\|\eta  w \|_{L^{2^*}(B_r)}  \leq \frac{c'_3 (1+|\gamma|^{-2})}{{\vartheta}^{-1}_{\epsilon,\om_{r_0}}(V, \epsilon \, c'_4\, (1+|\gamma|^{-2})^{-1}  )}\| (\eta + |\nabla \eta|) w \|_{L^2(B_r)}.
\end{equation}
and if, in addition, $|\gamma|>\tfrac{1}{2}$,  there exist $c_3>1$ and $c_4\in(0,1)$ such that
\begin{equation}\label{eq:iterat-moser}
\|\eta  w \|_{L^{2^*}(B_r)}  \leq \frac{c_3}{{\vartheta}^{-1}_{\epsilon,\om_{r_0}}(V, \epsilon \, c_4\, |\gamma|^{-1})}\| (\eta+ |\nabla \eta|) w \|_{L^2(B_r)}.
\end{equation}
If there exists $\beta_0 \in (0,1)$ such that $\beta_0\leq |\beta|<1$, then there exist constants $c_5>1$ and $c_6=c_6(\beta_0) \in(0,1)$ so that
\begin{equation}\label{eq:iterat-moser-b<1}
\|\eta  w \|_{L^{2^*}(B_r)}  \leq \frac{c_5}{{\vartheta}^{-1}_{\epsilon,\om_{r_0}}(V, \epsilon \,c_6\, |\gamma|)}\| (\eta+ |\nabla \eta|) w \|_{L^2(B_r)}.
\end{equation}
The implicit constants are independent of $\epsilon$ and $gamma$.
\end{lemma}

\begin{proof}
If $|\beta|>1$,  for $\ve$ to be chosen, by \eqref{eq:cor-Kato-emb-glob} we have that 
\begin{align} \label{eq:f-Mor-embd}
\int_{\om_r} (| \tilde f| +| \tilde g|^2) w^2 \eta^2 &\leq  c_1 \ve \left( \int_{\om_r} |\nabla( w \eta)|^2 + \frac{1}{{\vartheta}^{-1}_{\epsilon,\om_{r_0}}(V, c_2^{-1} \ve)^2} \int_{\om_r} | w \eta|^2 \right).
\end{align}
By \eqref{eq:f-Mor-embd}, we may rewrite  \eqref{eq:bdry-exp-Caccio} or \eqref{eq:inter-exp-Caccio},
\begin{align*}
\int_{\om_r}  |\eta \nabla w|^2 &\leq C\, |\gamma|^{-2}\, \int_{\om_r}  |\nabla \eta|^2 w^2  + 2\,\ve \,C \,c_1\, (1 + |\gamma|^{-2})  \int_{\om_r} |\nabla( w \eta)|^2\\
& + \ve \,C \,c_1 \,(1 + |\gamma|^{-2}) \frac{1}{{\vartheta}^{-1}_{\epsilon,\om_{r_0}}(V, c_2^{-1} \ve)^2}  \int_{\om_r} | w \eta|^2.
\end{align*}
Therefore, if we choose $\ve =\frac{\epsilon}{ 10 C c_1 (1+|\gamma|^{-2})}<0.1$, we deduce
\begin{align*}
\int_{\om_r}  |\eta \nabla w  |^2 &\leq C\, |\gamma|^{-2}\,   \int_{\om_r}  |\nabla \eta|^2 w^2 +  \frac{1}{5}  \int_{\om_r} |\nabla( w \eta)|^2 +  \frac{1}{10{\vartheta}^{-1}_{\epsilon,\om_{r_0}}(V, c_2^{-1} \ve)^2}\int_{\om_r} | w \eta|^2,
\end{align*}
which, in turn,  since  $C>1$, implies
\begin{align}\label{eq:eq4.6}
\int_{B_r}  | \nabla (w \eta)  |^2 &\leq \frac{10C (1+|\gamma|^{-2})}{3}   \int_{\om_r}  |\nabla \eta|^2 w^2 +  \frac{1}{3{\vartheta}^{-1}_{\epsilon,\om_{r_0}}(V, c_2^{-1} \ve)^2}\int_{\om_r} | w \eta|^2.
\end{align}
Notice that  $ \epsilon< \vartheta_{\epsilon, \om_{r_0}}(V,  1)$ and so $ \vartheta^{-1}_{\epsilon, \om_{r_0}}(V,  \epsilon) \leq 1$. Thus
$$
{\vartheta}^{-1}_{\epsilon,\om_{r_0}}(V, c_2^{-1} \ve)={\vartheta}^{-1}_{\epsilon,\om_{r_0}}\left(V, \epsilon\left(10 C c_1 c_2 \left(1+|\gamma|^{-2}\right)\right)^{-1}\right) \leq  \vartheta_{\epsilon,\om_{r_0}}^{-1}(V,  \epsilon) \leq 1,
$$
which,  if we set $c'_4:=  (10 C c_1 c_2)^{-1}<\tfrac{1}{10}$,  in light of  \eqref{eq:eq4.6},  gives
$$
\| \nabla (w \eta)  \|_{L^2(\om_r)} \leq \frac{(11C/3)(1+|\gamma|^{-2}) }{ {\vartheta}^{-1}_{\epsilon,\om_{r_0}}\left(V, \epsilon \,c'_4  \left(1+|\gamma|^{-2}\right)^{-1}\right)}\| (\eta+ |\nabla \eta|) w \|_{L^2(\om_r)}.
$$
Moreover, if $|\gamma| >\tfrac{1}{2}$,  it holds that $\frac{|\gamma|^2}{1+|\gamma|^2} \geq \frac{1}{10 |\gamma|}$,  and,  if we set $c_4:=\tfrac{c'_4}{10}$,  we can deduce that 
$$
\| \nabla (w \eta)  \|_{L^2(\om_r)} \leq 20C ({\vartheta}^{-1}_{\epsilon,\om_{r_0}}(V, \epsilon \, c_4 |\gamma|^{-1}))^{-1}\| (\eta+ |\nabla \eta|) w \|_{L^2(\om_r)}.
$$
Since $ \eta w \in Y^{1,2}_0(B_r)$,  \eqref{eq:iterat-moser-1} and \eqref{eq:iterat-moser} follow by Sobolev's inequality.

In a similar fashion, for $0< |\beta|<1$,  if we choose $\ve =\frac{\epsilon |\beta|^2}{ 10 C c_1}<\tfrac{1}{10}$,  since $4^\kappa\geq 1$, we obtain
\begin{align*}
\int_{\om_r}  |\eta \nabla w |^2 &\leq \frac{C}{ |\beta|^2} \int_{\om_r}  |\nabla \eta|^2 w^2  + \frac{1}{5}  \int_{\om_r} |\nabla( w \eta)|^2 +  \frac{1}{10{\vartheta}^{-1}_{\epsilon,\om_{r_0}}(V, c_2^{-1} \ve)^2} \int_{\om_r} | w \eta|^2.
\end{align*}
which  entails
\begin{align*}
 \int_{B_r}   | \nabla (w \eta)  |^2 &\leq \frac{10C}{3} \left(1+\frac{ 1}{ |\beta|^2} \right)   \int_{\om_r}  |\nabla \eta|^2 w^2 +  \frac{1}{3{\vartheta}^{-1}_{\epsilon,\om_{r_0}}(V, \epsilon \, c_4'|\beta|^2)^2}\int_{\om_r} | w \eta|^2.
\end{align*}
Thus,  as $0<\beta_0 \leq |\beta|<1$,   we have that $c_2^{-1} \,\ve\geq \epsilon\,\beta_0^2 c'_4\geq \epsilon\, |\gamma|\beta_0^2 c'_4/2$  and so, if we set $c_6:= \beta_0^2 c'_4/2 $,  since $\epsilon\,c_6 < \vartheta_{\epsilon,\om_{r_0}}(V,  c_6) $ and so $\vartheta^{-1}_{\epsilon,\om_{r_0}}(V, \epsilon\,c_6) <  c_6 $, there exists $c_5>1$ (independent of $\beta_0$) such that
\begin{equation}\label{eq:caccio-w-moser}
\|\nabla (\eta w)\|_{L^2(\om_r)} \leq  \frac{c_5}{{\vartheta}^{-1}_{\epsilon,\om_{r_0}}(V,  \epsilon \,c_6 \,|\gamma|)}\| (\eta+ |\nabla \eta|) w \|_{L^2(B_r)}.
\end{equation}
We conclude the proof of \eqref{eq:iterat-moser-b<1} by Sobolev's inequality.
\end{proof}

\vv

\begin{remark}\label{rem:c-d-moser}
Lemma \ref{lem:caccio-w-moser} can be proved  in the cases
\begin{enumerate}
\item  $-\divv c +d \leq 0$ (or $\geq 0$) and  $|b+c|^2 \in \mathcal{K}(\om_{r_0})$,\label{item:moser1}
\item  $|b|^2 \in  \mathcal{K}(\om_{r_0})$ and $d \in \mathcal{K}(\om_{r_0})$, and either $c \in L^{n,q}(\om_{r_0})$, $ q \in [n, \infty)$, or   $|c|^2 \in  \mathcal{K}(\om_{r_0})$.\label{item:moser2}
\end{enumerate}
We set 
\begin{equation}\label{eq:def-Moser-k(r)-cd}
k(r)= 
\begin{dcases}
  \vartheta_{\om_{r_0}}(|f|,r) + \vartheta_{\om_{r_0}}(|g|^2, r)^{1/2} +\vartheta_{\om_{r_0}}(|b+c|^2, r)^{1/2} &, \textup {in Case}\, \ref{item:moser1},\\
 \vartheta_{\om_{r_0}}(|f|,r) + \vartheta_{\om_{r_0}}(|g|^2, r)^{1/2} +\vartheta_{\om_{r_0}}(|b|^2, r)^{1/2} + \vartheta_{\om_{r_0}}(|d|,r)  &,\textup {in Case}\, \ref{item:moser2},
\end{dcases}
\end{equation} 

For $k$ as in \eqref{eq:def-Moser-k(r)-cd},  we   use Theorem \ref{thm:bdry-exp-subCaccioppoli2} and Theorem \ref{thm:bdry-exp-subCaccioppoli1-bis} respectively, and set 
\begin{equation}\label{eq:Moser-k(r)-cd}
V= 
\begin{dcases}
|\tilde f|+|\tilde g|^2+|b+c|^2 &, \textup {in Case}\, \ref{item:moser1},\\
|\tilde f|+|\tilde g|^2+|b|^2 + |d|&,\textup {in Case}\, \ref{item:moser2},
\end{dcases}
\end{equation} 
in order to obtain the same results as in Lemma \ref{lem:caccio-w-moser}.
\end{remark}

\vv

We are now ready to prove the local boundedness of subsolutions.  

\begin{definition}\label{def:N-P-D}
We will say that {\it the condition} $\textup{(N)}_{r_0}$ is satisfied if one the following conditions hold:
\begin{enumerate}
\item $\divv b + d \leq 0$ and $b+c \in L^{n,q}(\om_{r_0})$, $q \in [n, \infty)$ or $|b+c|^2 \in \mathcal{K}(\om_{r_0})$;
\item $-\divv c + d \leq 0$ and $|b+c|^2 \in \khalf(\om_{r_0})$.
\end{enumerate}
Analogously, we will say that {\it the condition} $\textup{(P)}_{r_0}$ is satisfied if we reverse the inequalities in condition  \textup{(N)}. Here, (N) and (P) stand for the negativity and positivity condition respectively. We will also say that  {\it the condition} $\textup{(D)}_r$ is satisfied if $|b|^2 \in  \khalf(\om_r)$,  $d \in  \khalf(\om_r)$, and either $c \in L^{n,q}(\om_{r_0})$, $q \in [n, \infty)$, or $|c|^2 \in \mathcal{K}(\om_{r_0})$.  If the above conditions hold globally, i.e., for $r_0=\infty$ and $\om$ instead of $\om_{r_0}$, we will  drop the subscript ${r_0}$ and simply write $\textup{(N)}$, $\textup{(P)}$, and $\textup{(D)}$.
\end{definition}

\vv

In the next theorem we borrow ideas from \cite{RZ}, although, some details are  different in our case. For example, we had to introduce the auxiliary modulus $\vartheta'_{\om_r}$ to be able to use Lemma \ref{lem:modulus-dini} and define  the appropriate Dini condition that gives constants independent of $\om$. 

\begin{theorem}[\bf Local boundedness]\label{thm:Moser}
Let $B_r$ be a ball such that $B_r \cap \om \neq \emptyset$, for  $r \leq r_0$, and assume that $f,  |g|^2 \in \khalf (\om_{r_0} )$. If $\sigma\in (0,1)$, then the following hold:\\
\textup{(1)} If $u$ is a subsolution of \eqref{eq:Lu=f-divg} in $B_r \cap \om$ and the condition $\textup{(N)}_{r_0}$ or $\textup{(D)}_{r_0}$  holds, then 
\begin{itemize}
\item[(i)]  if $B_r \subset \om$
\begin{equation}\label{eq:int-localbddness-u+}
 \sup_{B_{\sigma r}} u^+ \lesssim (1-\sigma)^{-n/p} \left(r^{-n/p}\,\| u^+ \|_{L^p(B_r)} + k(r) \right);
\end{equation}
\item[(ii)] if $ B_r$ is centered at a point $\xi \in \d \om$, 
\begin{equation}\label{eq:bdry-localbddness-u+}
  \sup_{B_{\sigma r}} \ump \lesssim (1-\sigma)^{-n/p} \left(r^{-n/p}\,\| \ump \|_{L^p(B_r)} + k(r)  \right).
  \end{equation}
  \end{itemize}
  \textup{(2)} If $u$ is a supersolution of \eqref{eq:Lu=f-divg} in $B_r \subset \om$ and the condition $\textup{(P)}_{r_0}$ or $\textup{(D)}_{r_0}$ holds, then 
\begin{itemize}
\item[(i)]  if $B_r \subset \om$
\begin{equation}\label{eq:int-localbddness-u-}
 \sup_{B_{\sigma r}} u^- \lesssim (1-\sigma)^{-n/p} \left(r^{-n/p}\,\| u^- \|_{L^p(B_r)} + k(r) \right).
\end{equation}
\item[(ii)] if $ B_r$ is centered at a point $\xi \in \d \om$, 
\begin{equation}\label{eq:bdry-localbddness-u-}
 \sup_{B_{\sigma r}} (- \umm) \lesssim (1-\sigma)^{-n/p} \left(r^{-n/p}\,\| \umm \|_{L^p(B_r)} + k(r) \right).
\end{equation}
  \end{itemize}
The implicit constants depend only on $p$, $\sigma$, $n$, $\lambda$, $\Lambda$, $C_{|f|, \om_{r_0}}, C_{|g|^2, \om_{r_0}}$ and according to our assumptions, on the following:  a) $C_{s,q}$ and $\|b+c\|_{L^{n,q}(\om_{r_0})}$, or  $C_s'$ and $\vartheta_{\om_{r_0}}(|b+c|^2, r)$, b) $C_{|b+c|^2, \om_{r_0}}$, and c) $C_{|b|^2, \om_{r_0}}$, $C_{|d|, \om_{r_0}}$, and either $C_{s,q}$ and $\|c\|_{L^{n,q}(\om_{r_0})}$, or  $C_s'$ and $\vartheta_{\om_{r_0}}(|c|^2, r)$.
\end{theorem}

\begin{proof}
Let us now pick $\eta$ so that, for $0 \leq \sigma_1 < \sigma_2 \leq \frac{1}{2}$,
$$
0 \leq \eta \leq 1,\quad  \eta=1 \,\,\textup{in}\,\, B_{\sigma_1 r},\quad \eta=0\, \,\textup{in}\,\, B_{\sigma_2 r},\quad \|\nabla \eta\|_{\infty} \leq 2/(\sigma_2-\sigma_1)r.
$$
If we set $\chi= \frac{n}{n-2}$ and $k=k(r)$,   then \eqref{eq:iterat-moser} for $r\leq 1$ can be written as
\begin{equation*}
\| w \|_{L^{2 \chi }(B_{\sigma_1 r})} \leq \frac{2 c_3}{(\sigma_2-\sigma_1)r} \, \frac{1}{{\vartheta}^{-1}_{\epsilon, \om_{r_0}}(V, \epsilon \,c_4 \, |\gamma|^{-1}  )}\,\| w \|_{L^2(B_{\sigma_2 r})},
\end{equation*}
which, in turn, implies that 
\begin{align}\label{eq:iterat-moser-pos}
\| \bu \|_{L^{ \gamma \chi }(B_{\sigma_1 r})} \leq \left( \frac{2 c_3}{(\sigma_2-\sigma_1)r}\right)^{2/\gamma} \frac{1}{{\vartheta}^{-1}_{\epsilon, \om_{r_0}}(V, \epsilon \, c_4 \,|\gamma|^{-1}  )^{2/\gamma}}\,\| \bu \|_{L^\gamma(B_{\sigma_2 r})},
\end{align}
if $|\gamma|>\tfrac{1}{2}$ and $u$ is a subsolution. 

For $p>1$ and  any non-negative integer $i$, we set
$$
\gamma_i:= \chi^i \,p=(1+\tfrac{2}{n-2})^i p\geq p >1  \quad \textup{and}\quad \sigma_i:=\frac{1}{2}+\frac{1}{2^{i+1}},
$$
and apply \eqref{eq:iterat-moser-pos} with $\gamma = \gamma_i$, $\sigma_1=\sigma_{i+1}$ and $\sigma_2=\sigma_{i}$ to obtain
\begin{align*}
\| \bu \|_{L^{ \gamma_{i+1} }(B_{\sigma_{i+1}r })} &\leq (2 c_3 2^{i+2}/r)^{2/\gamma_i} \,  \frac{1}{{\vartheta}^{-1}_{\epsilon, \om_{r_0}}(V, \epsilon \, c_4 \gamma_i^{-1}  )^{2/\gamma_i}}\| \bu \|_{L^{\gamma_i}(B_{\sigma_i r})}\\
&=: (K_1/r^{2/p})^{1/\chi^i} \, K_2^{i/\chi^i}\,\frac{1}{{\vartheta}^{-1}_{\epsilon, \om_{r_0}}(V,  c_7\, \chi^{-i} )^{2/p \chi^i}}\,\| \bu \|_{L^{ \gamma_{i} }(B_{\sigma_{i}r })},
\end{align*}
where $K_1=(8\,c_3)^{2/p}$ and  $K_2=2^{2/p}$ and $c_7:=\epsilon \,c_4\,p <1$ (we can choose $c_4$ so that $p\, c_4<1$). Iteration of this inequality leads to 
\begin{align}\label{eq:bu-localbound}
 \sup_{B_{r/2}}\bu \leq (K_1r)^{\sum_i \frac{1}{\chi^i}}\, K_2^{\sum_i \frac{i}{\chi^i}}\, \prod_{i=0}^\infty \frac{1}{{\vartheta}^{-1}_{\epsilon,\om_{r_0}}(V, c_7\, \chi^{-i} )^{2/p \chi^i}}\,\| \bu \|_{L^p(B_r)}.
\end{align}

Thus,  since
\begin{align*}
\log \prod_{i=0}^\infty \frac{1}{{\vartheta}^{-1}_{\epsilon,\om_{r_0}}(V, c_7\, \chi^{-i})^{2/p \chi^i}} = -\frac{2}{ \epsilon\, c_4} \sum_{i=0}^\infty \frac{c_7}{\chi^i} \log {\vartheta}^{-1}_{\epsilon,\om_{r_0}}(V, c_7\,\chi^{-i}),
\end{align*}
we may apply Lemma \ref{lem:modulus-dini} for $\tau = \chi^{-1}$ and $c=c_7$, and by Lemma \ref{lem:changevariable},  we obtain
\begin{align*}
-&\frac{2}{ \epsilon\, c_4} \sum_{i=0}^\infty \frac{c_7}{\chi^i}  \log {\vartheta}^{-1}_{\epsilon, \om_{r_0}}(V, c_7\, \chi^{-i}) \leq \frac{2\chi}{(\chi-1)\epsilon\, c_4} \int_0^{{\vartheta}^{-1}_{\epsilon, \om_{r_0}}(V, c_7)} {\vartheta}_{\epsilon, \om_{r_0}}(V, t) \frac{dt}{t}\\
&=\frac{2\chi}{(\chi-1)\epsilon\, c_4} \int_0^{{\vartheta}^{-1}_{\epsilon, \om_{r_0}}(V, c_7)} {\vartheta}_{ \om_{r_0}}(V, t) \frac{dt}{t} +\frac{2\chi\,\epsilon}{(\chi-1)\epsilon\, c_4} {\vartheta}^{-1}_{\epsilon, \om_{r_0}}(V, c_7)\\
&\leq \frac{2\chi}{(\chi-1)\epsilon\, c_4}  \left( C_{| f|, \om_{r_0}}      {\vartheta}_{ \om_{r_0}}(|\tilde  f|, {\vartheta}^{-1}_{\epsilon, \om_{r_0}}(V, c_7)  ) +             C_{| g|^2, \om_{r_0}}          {\vartheta}_{ \om_{r_0}}(|\tilde  g|^2, {\vartheta}^{-1}_{\epsilon, \om_{r_0}}(V, c_7))  \right)\\
&+\frac{2\chi\,\epsilon}{(\chi-1)\epsilon\, c_4} {\vartheta}^{-1}_{\epsilon, \om_{r_0}}(V, c_7)\\
&\leq \frac{2\chi}{(\chi-1)\epsilon\, c_4}  \left(  \big(  C_{| f|, \om_{r_0}}   +   C_{| g|^2, \om_{r_0}}  \big){\vartheta}_{ \om_{r_0}}(V, {\vartheta}^{-1}_{\epsilon, \om_{r_0}}(V, c_7)) + \epsilon{\vartheta}^{-1}_{\epsilon, \om_{r_0}}(V, c_7)\right)\\
&\leq  \frac{2\chi}{(\chi-1)\epsilon\, c_4}   \max\left( \big(C_{| f|, \om_{r_0}}+C_{| g|^2, \om_{r_0}}  \big),1 \right) {\vartheta}_{ \epsilon,\om_{r_0}}(V, {\vartheta}^{-1}_{\epsilon, \om_{r_0}}(V, c_7))\\
&\leq \frac{2\chi\,c_7}{(\chi-1)\epsilon\, c_4}  \max\left( C_{| f|, \om_{r_0}}+C_{| g|^2, \om_{r_0}},1 \right)= \frac{2\chi\,p}{(\chi-1)}  \max\left( C_{| f|, \om_{r_0}}+C_{| g|^2, \om_{r_0}},1 \right),
\end{align*}
where $C_{| f|,\om_{r_0}}$ and $C_{| g|^2,\om_{r_0}}$ stand for the Carleson-Dini constants \eqref{eq:Kato-dini}.

 By the definition of $\bu$, we get
$$
\sup_{B_{r/2}}  u^+_M \leq \sup_{B_{r/2}}  u^+_M+k(r)  \lesssim r^{-n/p}\,\| \bu \|_{L^p(B_r)} \lec r^{-n/p} \| u^+_M \|_{L^p(B_r)} +k(r),
$$
from which,  \eqref{eq:bdry-localbddness-u+} for $r\leq 1$ follows.  Replacing $u^+_M$ by $u^+$, the same argument shows \eqref{eq:int-localbddness-u+} for $r\leq 1$.

To obtain the desired estimates in any ball of arbitrary radius $r>1$ we use a rescaling argument. Note that $u_r= u(rx)$ is a subsolution (resp. supersolution) of the equation 
$$
-\divv ( A_r \nabla w +  b_r w) - c_r \nabla w - d_r w = f_r - \divv g_r,
$$
 where
\begin{align*}
A_r(x)&=A(rx), \quad b_r(x)=r b(rx), \quad c_r(x)=r c(rx), \quad d_r(x)=r^2d(rx),\\ 
f_r(x)& =r^2 f(rx), \quad g_r(x)= r g(rx).
\end{align*}
 If we set $D_{r}= \frac{1}{r}\om_{r_0}$,  by Lemma \ref{lem:changevariable},  we get that
\begin{align*}
 \|b_r+c_r\|_{L^{n,q}(D_{r})} &= \|b+c\|_{L^{n,q}(\om_{r_0})},\\
  {\vartheta}_{D_{r}}(f_r, 1)={\vartheta}_{\om_{r_0}}(f, r) \quad & {\vartheta}_{D_{r}}(|g_r|^2, 1)={\vartheta}_{\om_{r_0}}(|g|^2, r),
 \end{align*}
 and since the Dini condition is scale invariant, we have
 $$
 C_{f_r, D_{r}}=C_{f,\om_{r_0}} \quad  C_{|g_r|^2, D_{r}}=C_{|g|^2,\om_{r_0}}.
 $$
If we apply \eqref{eq:bu-localbound} to $u_r$ in the domain $D_{r}$, by the  change of variables $y=rx$, we obtain the following estimate:
\begin{align*}
 \sup_{B_{r/2}} \bu =\sup_{B_{1/2}}\bu_r \lesssim \| \bu_r \|_{L^p(B_1)} \approx r^{-n/p}\,\| \bu \|_{L^p(B_r)}.
\end{align*}
Remark that the implicit constants do not depend on $r$.

Moreover, if $0<\sigma  < 1/2$, 
\begin{align*}
 \sup_{B_{\sigma r}} \bu \leq  \sup_{B_{r/2}} \bu &\lesssim  r^{-n/p}\,\| \bu \|_{L^p(B_r)}\\
 &\lesssim (1-\sigma)^{-n/p} r^{-n/p}\,\| \bu \|_{L^p(B_r)}.
\end{align*}
and if $1/2<\sigma  < 1$, then for any ball $B(z, (1-\sigma) r) \subset B_{\sigma r}$, we get
\begin{align*}
 \sup_{B(z, (1-\sigma)r)}\bu \lesssim (1-\sigma)^{-n/p} r^{-n/p}\,\| \bu \|_{L^p(B(z, 2(1-\sigma)r)} \leq (1-\sigma)^{-n/p} r^{-n/p}\,\| \bu \|_{L^p(B_r)}.
\end{align*}
Thus, for any $\sigma \in (0,1)$, we have shown that
\begin{align*}
 \sup_{B_{\sigma r}} \bu &\lesssim (1-\sigma)^{-n/p} r^{-n/p}\,\| \bu \|_{L^p(B_r)},
\end{align*}
  which trivially implies \eqref{eq:int-localbddness-u+} and \eqref{eq:bdry-localbddness-u+}. To show \eqref{eq:int-localbddness-u-} and \eqref{eq:bdry-localbddness-u-}  it suffices to notice that $w=-u$ is a subsolution of $Lw = -f + \divv g $ and use  \eqref{eq:int-localbddness-u+} and \eqref{eq:bdry-localbddness-u+} as $\divv b + d \leq 0$ still holds.
  
 Using Remark \ref{rem:c-d-moser} we can prove the same result under either condition (D) or  $-\divv c + d \leq 0$ and $|b+c|^2 \in \khalf(\om_r)$. We omit the details.
  \end{proof}

\vv

We  turn our attention to the weak Harnack inequality.

\begin{theorem}[\bf Weak Harnack inequality]\label{thm:reverse-Holder-bu}
Let $B_r$ be a ball such that $B_r \cap \om \neq \emptyset$, for  $r \leq r_0$, and assume that $f, |g|^2 \in \khalf (\om_{r_0})$. If  $u$ is a supersolution of \eqref{eq:Lu=f-divg} in $B_r \cap \om$ and the condition $\textup{(P)}_r$ or $\textup{(D)}_r$ is satisfied, then for $0< s<p< \chi=n/n-2$, the following hold:
\begin{itemize}
\item[(i)]  if $B_r \subset \om$
\begin{align}
r^{-n/p} \| u \|_{L^p(B_{r/2})}& \lesssim r^{-n/q} \| u \|_{L^s(B_{r})} + k(r), \label{eq:reverse-Holder-int}\\
r^{-n/p}\| u \|_{L^p(B_{r})} &\lesssim \inf_{B_{r/2}} u + k(r/2).\label{eq:weak-Harnack-int}
\end{align}
\item[(ii)] if $ B_r$ is centered at a point $\xi \in \d \om$, 
\begin{align}
r^{-n/p} \| \umm \|_{L^p(B_{r/2})}& \lesssim r^{-n/s} \| \umm  \|_{L^s(B_{r})} + k(r),\label{eq:reverse-Holder-bdry}\\
r^{-n/p}\| \umm  \|_{L^p(B_{r})} &\lesssim \inf_{B_{r/2}} \umm + k(r/2), \label{eq:weak-Harnack-bdry}
\end{align}
\end{itemize}
The implicit constants depend only on $p$, $s$, $\sigma$, $n$, $\lambda$, $\Lambda$, $C_{|f|, \om_{r_0}}, C_{|g|^2, \om_{r_0}}$ and according to our assumptions, on the following:  a) $C_{s,q}$ and $\|b+c\|_{L^{n,q}(\om_{r_0})}$, or  $C_s'$ and $\vartheta_{\om_{r_0}}(|b+c|^2, r)$, b) $C_{|b+c|^2, \om_{r_0}}$, and c) $C_{|b|^2, \om_{r_0}}$, $C_{|d|, \om_{r_0}}$, and either $C_{s,q}$ and $\|c\|_{L^{n,q}(\om_{r_0})}$, or  $C_s'$ and $\vartheta_{\om_{r_0}}(|c|^2, r)$.
\end{theorem}

\begin{proof}
We shall first prove the reverse H\"older inequality for $\bu$.  Recall first that $\gamma= \beta+1$. If $p< \chi$,  there exists $\delta \in (0,1)$ such that $p= \delta \chi$. For any non-negative integer $i$, we let
$$
\gamma_i= \chi^{-i} \,p\quad \textup{and}\quad \sigma_i=1-\frac{1}{2^{i+1}},
$$
and apply \eqref{eq:iterat-moser-pos} (which is still true as $\beta<0$ when $0<\gamma=\beta+1<1$) with $\gamma = \gamma_i$, $\sigma_1=\sigma_{i}$ and $\sigma_2=\sigma_{i+1}$. If we argue as in the proof of the previous theorem we obtain
\begin{align*}
\| \bu \|_{L^{ \gamma_{i} }(B_{\sigma_{i} })} &\leq K_1^{1/\chi^i} \, K_2^{i/\chi^i}\,\frac{1}{{\vartheta}^{-1}_{\epsilon, \om_{r_0}}(V, c_6 \chi^{-i} )^{2/p \chi^i}}\, \| \bu \|_{L^{\gamma_{i+1}}(B_{\sigma_{i+1} })},
\end{align*}
where $K_1=(4\,c_5)^{2/p}$ and  $K_2=2^{2/p}$ and $c_6 <1$. As $q<p$, there exists $i_0 \in \bN$ such that $ \gamma_{i_0-1} \leq q < \gamma_{i_0-2}$. Thus, if we iterate the latter inequality $i_0$ times we get 
\begin{align}\label{eq:bu-RH}
\| \bu \|_{L^p(B_{1/2})} \lesssim   \| \bu \|_{L^q(B_{1})}.
\end{align}

If $u$ is a supersolution,  then \eqref{eq:iterat-moser-b<1} for $r=1$ implies
\begin{align*}
\| \bu \|_{L^{ -q  }(B_{\sigma_2  })} \leq \| \bu \|_{L^{ -\gamma_{i_0-1}  }(B_{\sigma_2  })}  \leq  K_1 ^{1/\chi^{i_0}} \, K_2^{i/\chi^{i_0}}\,\frac{1}{{\vartheta}^{-1}_{\epsilon,  \om_{r_0}}(V, c_6 \chi^{-{i_0}}  )^{2/p \chi^{i_0}}}\, \| \bu \|_{L^{ \gamma_{i_0}}(B_{\sigma_1 })}.
\end{align*}
By a similar iteration argument as above we can show that for any $q\in (0,\chi)$,
\begin{align}\label{eq:bu-WH-p}
\| \bu \|_{L^{-q}(B_1)} &\lesssim  \inf_{B_{1/4}}\bu.
\end{align}

Set now $w = \log \bu$ and let $B_r(x)$ a ball centered at $x$ of radius $r\leq 1/2 $ so that $B_{2r}(x) \subset B_1$. Let also $\eta \in C^\infty_c(B_{2r}(x))$ so that $\eta=1$ in $B_r(x)$, $\eta =0$ outside $B_{2r}(x)$ and $\|\nabla \eta\|_{\infty} \lesssim 1/r$. Then, by Poincar\'e and H\"older inequalities, along with \eqref{eq:bdry-exp-Caccio} or \eqref{eq:inter-exp-Caccio} for $\beta=-1$ and the fact that $|\bar f \leq |\tilde f|$, $| \bar g| \leq |\tilde g|$,  and $k=k(1)$, we get
\begin{align*}
&\int_{B_r(x)} \Big|w - \avint_{B_r} w \Big| \lesssim r \int_{B_r(x)} |\nabla w| \lesssim r r^{n/2} \left( \int_{B_r(x)} |\nabla w|^2 \right)^{1/2}\\
& \leq  r r^{n/2} \left( \int_{B_{2r}(x)} |\eta \nabla w|^2 \right)^{1/2} \lesssim  r r^{n/2}\left[  \int_{B_{2r}(x)} | \nabla \eta|^2  + \int_{B_{2r}(x)}( |\tilde f| + |\tilde g|^2) \right]^{1/2}\\
& \lesssim   r r^{n/2}\left[  \int_{B_{2r}(x)} | \nabla \eta|^2  + r^{n-2} \int_{B_{2r}(x)} \frac{ |\tilde f(y)| + |\tilde g(y)|^2}{|x-y|^{n-2}}\,dy \right]^{1/2}\\
&\lec  r^n \left[ 1 +  \vartheta_{\om_{r_0}}(|\tilde f|,r) +   \vartheta_{\om_{r_0}}(|\tilde g|^2,r) \right]^{1/2}\\
& =  r^n \left[ 1 + \frac{ \vartheta_{\om_{r_0}}(| f|,r)}{k(1)} +  \frac{  \vartheta_{\om_{r_0}}(| g|^2,r)}{k(1)} \right]^{1/2} \leq  2r^n.
\end{align*}
This shows that, $w \in \bmo(B_1)$ and thus, there exists $s\in (0,1)$ such that $e^{s w}$ is in the class of $A_2$ Muckenhoupt weights in $B_1$. That is,
$$
\left(\int_{B_1} \bu^s \right)^{1/s} \lesssim \left(\int_{B_1} \bu^{-s} \right)^{-1/s}.
$$
This, combined with \eqref{eq:bu-RH} and \eqref{eq:bu-WH-p}, implies that, for any $0<p<\chi$,
$$
\| \bu \|_{L^{p}(B_{1/2})} \lesssim \inf_{B_{1/4}}\bu
$$
and so \eqref{eq:reverse-Holder-int}-\eqref{eq:weak-Harnack-bdry} hold for $r=1$. The general case follows by rescaling.
\end{proof}

\vv

\begin{remark}\label{rem:moser-glob}
If we impose global assumptions (e.g. $|c|^2\in \mathcal{K}'(\om)$ and $|b|^2, |d| \in \khalf(\om)$) on the coefficients and the interior data, then we may take $r_0=\infty$ and all of the constants in Theorems \ref{thm:Moser} and \ref{thm:reverse-Holder-bu} are independent of $r$. In particular, the implicit constants depend  on $p$, $\sigma$, $n$, $\lambda$, $\Lambda$, $C_{s,q}$, $C_{|f|, \om}, C_{|g|^2, \om}$ and according to our assumptions, on the following:  a) $C_{s,q}$ and $\|b+c\|_{L^{n,q}(\om)}$, for $q \in [n,\infty)$, or  $C_s'$ and $\vartheta_{\om}(|b+c|^2)$, b) $C_{|b+c|^2, \om}$, and c) $C_{s,q}$ and $\|c\|_{L^{n,q}(\om)}$, for $q \in [n,\infty)$, or  $C_s'$ and $\vartheta_{\om}(|b+c|^2)$, $C_{|b|^2, \om}$, and $C_{|d|, \om}$.
\end{remark}

\vv

\begin{remark}\label{rem:molifier-moser-indep}
Let $\delta>0$, $\psi_\delta$ be as in \eqref{eq:molifier}, and  $\om_\delta=\{x\in \om: \dist(x, \d\om) >\delta\} \cap B(0,\delta^{-1})$. Define $b_\delta= (b {\bf 1}_{\om_\delta}) \ast \psi_\delta$, $c_\delta= (c {\bf 1}_{\om_\delta}) \ast \psi_\delta$, and $d_\delta= (d {\bf 1}_{\om_\delta}) \ast \psi_{\om_\delta}$. Let us also  define $L_\delta u = -\divv A \nabla u - \divv (b_\delta u) - c_\delta \nabla u - d_\delta u$. If \eqref{negativity} (resp. \eqref{negativity2}) holds for  $b$, $c$ and $d$ in $\om$, then  \eqref{negativity} (resp. \eqref{negativity2}) holds in $\om_\delta$. For a proof see  Lemma 6.9 in \cite{KSa}. Moreover, $\| b_\delta+c_\delta\|_{L^{n,q}(\om)}$ is dominated by $2\| b+c\|_{L^{n,q}((\om)}$ and so, all the constants in the theorems of section \ref{sec:caccio} are independent of $\delta$. In the case that \eqref{negativity} holds, everything works exactly as before. On the other hand, if \eqref{negativity2} is satisfied and  $|b+c|^2 \in \khalf(\om)$, we should use Corollary \ref{cor:Kato-emb-glob} in the proof of Lemma \ref{lem:caccio-w-moser} to obtain bounds which are independent of $\delta$. Theorems \eqref{thm:Moser} and \eqref{thm:reverse-Holder-bu}  for subsolutions and supersolutions of $L_\delta$ in $\om_\delta$ will then follow from the same proofs with estimates uniform in $\delta$.
\end{remark}

\vv

\begin{example}\label{ex:c1}
Let us now refer to the counterexample constructed in  \cite[Lemma 7.4]{KSa}. In particular, the authors defined the operator
$$
-\Delta u - \divv( \delta b  u)=0 \,\,\textup{in}\,\,B(0, e^{-1}), 
$$
where $b(x)=- \frac{x}{|x|^2 | \ln|x| |}$ and $\delta>0$, and showed that the  solution $u = | \ln|x| |^\delta  \in Y^{1,2}(B(0, e^{-1}))$ does not satisfy \eqref{eq:int-localbddness-u+} around $0$. They proved that $b \in L^q(B(0, e^{-1}))$ for any $q >n$  but not in $L^n(B(0, e^{-1}))$. It is not hard to see that $|b|^2 \in \mathcal{K}(B(0, e^{-1}))$ but not in $\khalf(B(0, e^{-1}))$, and thus, assuming $|b+c|^2 \in \mathcal{K}(\om)$ does not suffice to establish local boundedness.  A modification of this example shows that  \eqref{eq:weak-Harnack-int} does not hold  when $|b+c|^2 \in  \mathcal{K}(\om)$. It is important to note that, since $\delta$ can be taken as small as we want, this example shows that assuming the norms to be small  is not enough either.
\end{example}

\vv

\begin{example}\label{ex:d>0}
Adjusting the previous example we can find an operator which does not satisfy neither \eqref{negativity} nor \eqref{negativity2}, for which there exists a non-bounded solution in the ball $B(0, e^{-1})$. Indeed, let
\begin{equation}\label{eq:counter}
 -\Delta u - d  u=0 \,\,\textup{in}\,\,B(0, e^{-1}), \quad  \textup{where}\,\,d(x)= \frac{n-2}{|x|^2 | \ln|x| |}.
\end{equation}
 It is not hard to see that $ d \geq 0$ is  in the Lorentz space  $L^{n/2,q}(B(0, e^{-1}))$, for any $q>1$.  But notice that  $u = | \ln|x| |$  is a solution of \eqref{eq:counter} and $u \in Y^{1,2}(B(0, e^{-1}))$. Since $u$  fails to be bounded around $0$,  the necessity of either \eqref{negativity} or \eqref{negativity2} to prove local boundedness is established. It is interesting to see that $d $ is not  in $\mathcal{K}(B(0, e^{-1}))$ (and thus, it is not in $L^{n/2,1}(B(0, e^{-1}))$ either). 
\end{example}

\vvv

\subsection{Interior and boundary regularity}\label{subs:regularity}

\vv

\begin{theorem}
Let $u$ be a supersolution of \eqref{eq:Lu=f-divg} in $\om$ with $\sup_{\om} u<\infty$ and assume that the condition \textup{(P)} or  \textup{(D)} holds. Then  $u$ has a lower semi-continuous representative satisfying
\begin{equation}\label{eq:lsc-repres}
u(x) = \ess \liminf_{y\to x} u(y) = \lim_{r \to 0} \avint_{B(x,r)} u(y)\,dy, \quad \textup{for all} \,\, x \in \om.
\end{equation}
\end{theorem}

\begin{proof}
We follow the proof of \cite[Theorem 3.66]{HKM}. Fix a ball $B_r$ centered at $x \in \om$ so that $B_{2r} \subset \om$ and apply Theorem \ref{thm:reverse-Holder-bu} (i) to $u - m_r$, where $m_r=\ess \inf_{B_r} u$. Then, we have
\begin{align*}
 0 \leq \avint_{B_{r}} (u - m_r)  \leq C( ( m_{r/2} -m_r) +  k(r) ).
\end{align*}
Since $C$ is either a constant independent of $r$ and $(m_{r/2} -m_r)  + k(r) \to 0$ as $r \to 0$, by taking limits in the inequality above as $r \to 0$, we obtain
$$
\lim_{r \to 0}  \avint_{B_{r}} (u - m_r) = \ess \liminf_{y \to x}  (u - m_r) = 0,
$$
which implies \eqref{eq:lsc-repres}.
\end{proof}

\vv

Let us now  introduce some notation that we will use in the rest of this section.  For $r \leq r_0/2$ and $r_0 \in (0, \infty]$,  set
 \begin{align}\label{eq:holdcont-k1}
 k_{\epsilon,1}(r)&:=   \vartheta_{\om_{r_0}}(| f|, r) + \big(\sup_{\om_{r}} |u|\big)\,  \vartheta_{\om_{r_0}}(|d|, r) +\epsilon r,\\
\lim_{\epsilon \to 0} k_{\epsilon,1}(r) &=  k_1(r) := \vartheta_{\om_{r_0}}(| f|, r) + \big(\sup_{\om_{r}} |u|\big)\,  \vartheta_{\om_{r_0}}(|d|, r), \label{eq:lim-k1}\\
 k_{\epsilon,2}(r)&:= \vartheta_{\om_{r_0}}(|g|^2, r)^{1/2} +\big( \sup_{\om_{r}} |u|\big) \,\vartheta_{\om_{r_0}}(|b|^2, r)^{1/2}+\epsilon r,\label{eq:holdcont-k2}\\
 \lim_{\epsilon \to 0} k_{\epsilon,2}(r) &=  k_2(r) :=\vartheta_{\om_{r_0}}(|g|^2, r)^{1/2} + \big(\sup_{\om_{r}} |u|\big)\,  \vartheta_{\om_{r_0}}(|b|^2, r)^{1/2},\label{eq:lim-k2}\\
  k_{\epsilon,3}(r)&:= \vartheta_{\om_{r_0}}(|b|^2, r)^{1/2}+ \vartheta_{\om_{r_0}}(|d|, r)+\epsilon r,\label{eq:holdcont-k3}\\
   \lim_{\epsilon \to 0}   k_{\epsilon,3}(r)&= k_{3}(r):= \vartheta_{\om_{r_0}}(|b|^2, r)^{1/2}+ \vartheta_{\om_{r_0}}(|d|, r),\label{eq:holdcont-k3}\\
     k_{\epsilon,4}(r)&:= \vartheta_{\om_{r_0}}(|f|, r)+\vartheta_{\om_{r_0}}(|g|^2, r)^{1/2}+ \epsilon r,\label{eq:holdcont-k4}\\
   \lim_{\epsilon \to 0}   k_{\epsilon,4}(r)&= k_{4}(r):= \vartheta_{\om_{r_0}}(|g|^2, r)^{1/2}+ \vartheta_{\om_{r_0}}(|f|, r),\label{eq:holdcont-k4}\\
 \wt k_\epsilon(r)&:= k_{\epsilon,1}(r)+  k_{\epsilon,2}(r),\label{eq:holdcont-bar-k}\\
  \wt k(r)&:= k_{1}(r)+  k_{2}(r). \label{eq:holdcont-bar-k}
 \end{align}
 
 \vv
 
If $k$ is defined as in Case \eqref{item:moser2} of \eqref{eq:def-Moser-k(r)-cd}, then $k = k_3+k_4$.  All the functions above with subscript $\epsilon$ are strictly increasing  and from their very definitions we have the following:

\vv

\begin{lemma}
If  $u$ satisfies 
\begin{equation}\label{eq:reg-local-bounds}
 \sup_{\om_r} |u| \lesssim \left( \avint_{\om_{2r}} |u|^2 \right)^{1/2} + k(2r),\quad \textup{for any} \,\,r \leq r_0/2.
 \end{equation}
 then, if $0<r_1 \leq r_0$, 
\begin{equation}\label{eq:bk-k3-u-k}
\wt{k}(r) \lesssim k_3(r)\left[ \left( \avint_{\om_{r_1}} |u|^2 \right)^{1/2}  + k(r_1) \right]+ k_4(r),\quad \textup{for any} \,\,r \leq r_1/2.
\end{equation}
\end{lemma}

\vv

\begin{theorem}[\bf Modulus of continuity in the interior]\label{thm:Holder-cont}
Let $0<r \leq r_0/2$ and  $B_{r}$ be a ball such that $\overline B_{r} \subset \om $. Assume that $|f|$, $|d|$, $|b|^2$, and $|g|^2 \in \khalf(B_{r_0})$, and either  $c \in L^{n,q}(B_{r_0})$, $q \in [ n, \infty)$, or $|c|^2 \in \mathcal{K}(B_{r_0})$.  If $u$ is a solution of \eqref{eq:Lu=f-divg} in $B_{r}$, then for every $\mu\in (0,1)$, there exists $\alpha \in (0,1)$ so that 
\begin{align}\label{eq:Holdercont}
|u(x)-u(y)| \lesssim & \left[ \left( \frac{|x-y|}{r} \right)^\alpha + k_3(|x-y|^\mu r^{1-\mu})  \right] \left[  \left( \frac{1}{r^n} \int_{B_r} |u|^2 \right)^{1/2} + k(r) \right] \\
&+ k_4(|x-y|^\mu r^{1-\mu} ),\notag
\end{align}
for all $x, y \in B_{r/2}$, where $ k_3(r) $  and $ k_4(r) $ are given by \eqref{eq:holdcont-k3}. and  \eqref{eq:holdcont-k4}.   Note that $\alpha$ and the implicit constants depend only on $\lambda$, $\Lambda$,  $C_{|f|, \om_{r_0}}, C_{|g|^2, \om_{r_0}}$ and either on $C_{s,q}$ and $\|c\|_{L^{n,q}(\om_r)}$, or  $C_s'$ and $\vartheta_{\om_{r_0}}(|b+c|^2, r)$.
\end{theorem}

\begin{proof}
Fix $r_1 \in (0,r_0/2)$ such that $B_{r_1} \subset \om $ and assume that $u$ is a weak solution of the equation $Lu = f -\divv g$ in $B_{r_1}$.
It is easy to see that $u$ is also a weak solution of the equation 
\begin{equation}\label{eq:holdcont-wtL}
  \wt Lu =- \divv A  \nabla u - c \nabla u = (f+ du) -\divv (g - b u).
\end{equation}
in $B_{r_1}$. Note that  $\wt L 1 = 0 $ and since  $\wt d = \wt b_i = 0$, $i=1, \dots, n$, we can use Theorems \ref{thm:Moser} and  \ref{thm:reverse-Holder-bu} with $\wt k$ as in \eqref{eq:holdcont-bar-k} to get
\begin{equation}\label{eq:harnack-bu}
\sup_{B_r} (u + \wt k(r)) \lesssim \avint_{B_{2r}} (u + \wt k(r)) \lesssim  \inf_{B_r} (u + \wt k(r)),\quad \textup{for any}\,\, r \leq  r_0/2.
\end{equation}

Now, let
$$
 M_0=\sup_{B_{r_1}} |u|, \quad M_r=\sup_{B_r} u \quad \textup{and}\quad m_r=\inf_{B_r} u,
$$
and since  $M_r-u$ and $u-m_r$ are non-negative solutions of \eqref{eq:holdcont-wtL} in $B_{r_0}$,  by \eqref{eq:weak-Harnack-int} for $r \leq r_0/2$, we obtain
\begin{align*}
\avint_{B_r} (M_r - u) &\leq C \left(M_r-M_{r/2} +\wt k(r/2)\right),\\
\avint_{B_r}  (u - m_r) &\leq C \left(m_{r/2}- m_r +\wt k(r/2)\right).
\end{align*}
Summing those two inequalities we get
$$
(M_r  - m_r) \leq C \left[(M_r- m_r)-(M_{r/2}- m_{r/2})+2\wt k(r/2)\right],
$$
which further implies
$$
(M_{r/2}- m_{r/2}) \leq \frac{C-1}{C} (M_r  - m_r) + 2\wt k(r/2).
$$
If we set $\hm(r)=\osc_{B_r} u= M_r  - m_r$ and $\gamma= 1-C^{-1} \in (0,1)$, the latter inequality can be written
$$
\hm(r/2) \leq \gamma \hm(r) + 2 \wt k(r/2),
$$
which implies that for any $\mu \in (0,1)$ and for $\alpha= -(1- \mu) \log \gamma/\log2 \in (0,1)$,  there exists a constant $C'>0$ depending only on $\gamma$  such that
$$
\hm(r) \lesssim \left(\frac{r}{r_1}\right)^\alpha \hm(r_1) + \wt k(r^\mu \,r_1^{1-\mu}),
$$
which, by \eqref{eq:bk-k3-u-k}, concludes the proof.
\end{proof}

\vv

The last goal of this section is to develop of a Wiener-type criterion for boundary regularity of solutions. We will follow the proof of Theorem 8.30 in \cite{GiTr}. Several modifications are required in our case and in particular, we would like to draw the reader's attention to the iteration argument at the end of the proof. In  \cite{GiTr} it is claimed that the inequality (8.81) on p. 208  can be iterated to produce the desired oscillation bound at the boundary. Unless the CDC is satisfied, it is not clear to us that the second term on the right hand-side of that inequality will converge after infinitely many iterations. In fact, the exact term one picks up after $m$ iterations is 
$$
\left(\chi(r/2^m) + \sum_{k=0}^{m-1} \chi(r/2^k) \prod_{j = k+1}^m (1-\chi(r/2^j)) \right) \osc_{\d \om \cap B_r} u =: S_m \osc_{\d \om \cap B_r} u.
$$
It seems that if we do not have additional information about the behavior of the sequence $a_k =\chi(r/2^k)$, we could choose different sequences $a_k$  so that $S_m$  either converges or diverges or even have multiple limit points. We resolve this issue by incorporating this term into the main oscillation  term. 

\vv

Let us first introduce some definitions. 
\begin{definition}
We say that a set $E$ is {\it thick} at $\xi \in E$ if
\begin{equation}
\int_{0}^1 \frac{\Cap(E \cap  \overline B_{r}(\xi), B_{2r}(\xi))}{r^{n-2}}\frac{dr}{r}= +\infty.
\end{equation}

If $\om \subset \R^n$ is an open set and for $\xi \in \d \om$ it holds that 
$$
{\Cap(\overline B_{r}(\xi) \setminus \om,B_{2r}(\xi) ) } \geq c_0\, {r^{n-2}}, \quad \textup{for all}\,\, r \in (0, \diam \d \om),
$$
for some $c\in (0,1)$ independent of $r$, we say that $\om$ satisfies the {\it capacity density condition \textup{(CDC)} at $\xi$}. If this holds for every  $\xi \in \d \om$ and a uniform constant $c$, we say that $\d \om$  has  the {\it capacity density condition}.
\end{definition}

\vv

\begin{theorem}[\bf Boundary oscillation]\label{thm:wiener-osc}
Let $r \leq r_0/2$ and  $B_{r}$ be a ball centered at $\xi \in \d \om$.  Assume also that $u$ is a  solution of \eqref{eq:Lu=f-divg} and $\vphi \in Y^{1,2}(\om) \cap C(\overline{\om})$ so that $u-\vphi$ vanishes on $\d \om \cap B_r$ in the Sobolev sense.  Then,  the following hold:
\begin{itemize}
\item[(i)] Let $|f|$, $|d|$, $|b|^2$, and $|g|^2 \in \khalf(\om_{r_0})$, and either  $c \in L^{n,q}(\om_{r_0})$, $q \in [ n, \infty)$, or $|c|^2 \in \mathcal{K}(\om_{r_0})$.   If $\om$ satisfies the \textup{(CDC)} at $\xi$,  then
\begin{align}\label{eq:bdry-Holder-cont}
|u(x)-u(y)| \lesssim & \left[ \left( \frac{|x-y|}{r} \right)^\alpha + k_3(|x-y|^\mu r^{1-\mu})  \right] \left[  \left( \frac{1}{r^n} \int_{\om_r} |u|^2 \right)^{1/2} + k(r) \right] \\
&+ k_4(|x-y|^\mu r^{1-\mu} ) + |\vphi(x)-\vphi(y)|,\,\notag
\end{align}
for all $x, y \in B_{r/2}$ and $0<r \leq r_0/2$. Here $k_3 $ and $ k_4(r) $ are given by \eqref{eq:holdcont-k3}. and  \eqref{eq:holdcont-k4},  and the implicit constants depend on the \textup{CDC}  constant $c_0$, $C_{|f|, \om_{r_0}}, C_{|g|^2, \om_{r_0}}$, $C_{|b|^2, \om_{r_0}}$, $C_{|d|, \om_{r_0}}$,  $\lambda$,  $\Lambda$,  and either $C_{s}$ and $\|c\|_{L^{n,q}(\om_r)}$ or  $C_s'$ and $\vartheta_{\om_{r_0}}(|c|^2, r)$.
\item[(ii)] Let $|f|, |d| \in \mathcal{K}_{\textup{Dini}, \delta}(\om_{r_0})$,  $|b|^2, |g|^2 \in \mathcal{K}_{\textup{Dini}, \delta/2}(\om_{r_0})$ for some $\delta\in(0,1)$,  and either  $c \in L^{n,q}(\om_{r_0})$, $q \in [n, \infty)$, or $|c|^2 \in \mathcal{K}(\om_{r_0})$.  For any $0 \leq \rho \leq r/2$, it holds
\begin{align*}
\osc_{B_\rho(\xi) \cap \om } & u \leq  \osc_{\d \om \cap B_{\rho}(\xi) } \vphi \\
&+\exp\left(- \frac{1}{C} \int_{2\rho}^r \frac{\Cap(\overline B_{s}(\xi)  \setminus \om)}{s^{n-2}}\, \frac{ds}{s} \right) \, \left(\osc_{B_{r}(\xi) \cap \om} u + \Big(\wt k(r) +\frac{\wt k(r_0/2)}{r_0/2}  r\Big)\right),\notag
\end{align*}
where $C>0$ depends on $\lambda$,  $\Lambda$,  $k_0$ as defined in \eqref{eq:k0}, $C_{|f|, \om_{r_0},\delta}, C_{|g|^2, \om_{r_0},\delta/2}$, $C_{|b|^2, \om_{r_0},\delta/2}$, $C_{|d|, \om_{r_0},\delta}$, and either $C_{s}$ and $\|c\|_{L^{n,q}(\om_r)}$ or  $C_s'$ and $\vartheta_{\om_{r_0}}(|c|^2, r)$.
\end{itemize}
\end{theorem}

\begin{proof}
If we set $B_r=B_r(\xi)$ we record that  $u$ is a solution of $Lu = f -\divv g$ in $B_{r} \cap \om$ and thus, a solution of \eqref{eq:holdcont-wtL}. Using the same notation as above, one can prove that for $\eta \in C^\infty_c(B_{r})$,
\begin{equation}\label{eq:wiener-caccio}
\|\eta  \nabla \umm \|_{L^{2}(B_{r})}  \lesssim  \| (\eta + |\nabla \eta|)  (\umm+ \wt k_\epsilon)\|_{L^2(B_{r})}.
\end{equation} 
This follows easily by inspection of the proofs of Theorem \ref{thm:bdry-exp-subCaccioppoli1} and Lemma \ref{lem:caccio-w-moser}. 

We fix $\eta$ so that $\eta =1$ in $B_{1/2}$, $0 \leq \eta \leq 1$ and $|\nabla \eta | \leq 2$. If we set $w = \eta (\umm+\wt  k_\epsilon(1))$, by \eqref{eq:wiener-caccio} and (the proof of) \eqref{eq:weak-Harnack-bdry}, we deduce that
\begin{align*}
\|\nabla w \|_{L^2(B_{1})}^2 &\lesssim  \| (\eta + |\nabla \eta|)  (\umm+ \wt k_\epsilon(1))\|_{L^2(B_{1})}^2\\
 &\lesssim  (m+\wt  k_\epsilon(1)) \int_{B_{1}} (\umm+ \wt k_\epsilon(1))\lesssim (m+\wt  k_\epsilon(1))  (\inf_{B_{1/2}} \umm+ \wt k_\epsilon(1/2)).
\end{align*}
If we rescale, the latter inequality is written as
$$
r^{2-n} \|\nabla w \|_{L^2(B_{r})}^2   \lesssim  (m+\wt  k_\epsilon(r)) (\inf_{B_{r/2}} \umm+ \wt k_\epsilon(r/2)).
$$

It is easy to see that $\frac{w}{m+\wt k_\epsilon(r)}$ is a function in the convex set $\mathbb{K}_{\overline B_{r/2} \setminus \om}$ in the definition of capacity. This observation along with the latter inequality implies that
$$
(m+\wt  k_\epsilon(r))^2 \Cap(\overline B_{r/2} \setminus \om) \lesssim r^{n-2} (m+\wt k_\epsilon(r))  (\inf_{B_{r/2}} \umm+ \wt k_\epsilon(r/2)).
$$
Therefore, since $\wt k_\epsilon(r) \geq 0$,
\begin{equation}\label{eq:wiener-cdc}
m \frac{\Cap(\overline B_{r/2} \setminus \om)}{(r/2)^{n-2}} \leq C   \large(\inf_{B_{r/2}} \umm + \wt k_\epsilon(r/2)\large).
\end{equation} 

If we set 
$$
\gamma(r/2)= \frac{\Cap(\overline B_{r/2} \setminus \om)}{C\,(r/2)^{n-2}}, \quad M=\sup_{B_r \cap \d \om} u, \quad \textup{and} \quad m=\inf_{B_r \cap \d \om} u,
$$
we can apply \eqref{eq:wiener-cdc} to the functions $M_r - u$ and $u-m_r$ to obtain
\begin{align*}
(M_r - M)\gamma(r/2)& \leq  M_r-M_{r/2} + \wt k_\epsilon(r/2)=(M_r-M)- (M_{r/2} -M)+ \wt k_\epsilon(r/2),\\
( m-m_r)\gamma(r/2) &\leq  m_{r/2}-m_r + \wt k_\epsilon(r/2)=(m-m_r)- (m-m_{r/2} )+ \wt k_\epsilon(r/2).
\end{align*}
Set
$$
\hm(r)= \osc_{\om \cap B_r} u - \osc_{\d\om \cap B_{r}} u,
$$ 
and sum the above inequalities to get 
\begin{align}\label{eq:wiener-iter-1}
\hm(r/2) \leq (1- \gamma(r/2))\hm(r)+ 2 \wt k_\epsilon(r/2).
\end{align} 

If $ \gamma(r) > c$, for every $r>0$,  we can write  \eqref{eq:wiener-iter-1}  as $\hm(r/2) \leq (1- c)\hm(r)+ 2 \wt k_\epsilon(r/2)$ and take limits as $\epsilon \to 0$. Then, we can repeat the iteration argument  in the proof of Theorem  \ref{thm:Holder-cont} to show \eqref{eq:bdry-Holder-cont}.  

If $ \gamma(r) $ is not uniformly bounded from below, then for $m \in \bN$, \eqref{eq:wiener-iter-1} can be iterated to obtain
\begin{align}\label{eq:wiener-iter-j}
\hm(2^{-m}r)& \leq \prod_{j=1}^{m} (1- \gamma(2^{-j}r))\, \hm(r)  +2 \sum_{j=1}^m\,\wt k_\epsilon(2^{-j} r)  \prod_{\ell=j+1}^{m}  (1- \gamma(2^{-\ell}r))\\
&=:\Sigma_1+\Sigma_2.\notag
\end{align}

To handle $\Sigma_2$ we adjust the argument in \cite[pp. 202-203]{MZ}.  Let us define 
\begin{equation}\label{eq:k0}
k_0^\frac{1}{1-\delta} :=\sup_{t \in (0,r_0)}\frac{\wt k_\epsilon(t)}{\wt k_\epsilon(2t)} <1,
\end{equation}
for some $\delta \in (0,1)$,  where we used Lemma \ref{lem:kato-uniformly<1} to deduce that $k_0 <1 $. Define also
$$
b(r)= \frac{\gamma(r)}{1+\gamma_1},  \quad \textup{where} \quad \gamma_1=   (1-k_0)^{-1}\sup_{r \in (0,r_0)} \gamma(r).
$$
Since $b(r) \leq 1- k_0$ for all $r \in (0,r_0)$,  $1-t \leq e^{-t}$ and  $b(r) \leq \gamma(r)$, we have
\begin{align}\label{eq:wiener-S2} 
\Sigma_2 &\leq 2 \prod_{k=1}^{m} e^{- b(2^{-k}r)}  \sum_{j=1}^m  \wt k_\epsilon(2^{-j } r)  \prod_{\ell=1}^{j}  (1- b(2^{-\ell}r))^{-1}\\
&=2 \prod_{k=1}^{m} e^{- b(2^{-k}r)}  \sum_{j=1}^m   \wt  k_\epsilon(2^{-j } r) k_0^{-j}\notag \\
& \leq  \exp\Big(- \sum_{k=1}^{m} b(2^{-k}r) \Big)  \sum_{j=1}^m   \wt k_\epsilon(2^{-j } r)   \prod_{\ell=1}^{j} \left(\frac{\wt k_\epsilon(2^{-\ell+1} r)}{\wt k_\epsilon(2^{-\ell} r)}\right)^{1-\delta} \notag\\
&=  \wt k_\epsilon(r)^{1-\delta}\, \exp\Big(- \sum_{k=1}^{m} b(2^{-k}r) \Big)  \sum_{j=1}^m   \wt k_\epsilon(2^{-j } r)^{\delta}  \notag\\
& \lesssim \wt k_\epsilon(r)^{1-\delta}\, \exp\Big(- \sum_{k=1}^{m} b(2^{-k}r) \Big)\, \wt k_\epsilon(r/2)^{\delta},\notag
\end{align}
where in the last inequality we used the fact that $|f|, |d| \in \mathcal{K}_{\textup{Dini}, \delta}(\om_{r_0})$ and $|b|^2, |g|^2 \in \mathcal{K}_{\textup{Dini}, \delta/2}(\om_{r_0})$ and the implicit constants depend on the  constants of the relevant Carleson-Dini conditions.  If we choose $\epsilon= \min(2\wt k(r_0/2)/r_0,1)$,  the latter quantity is dominated by 
\begin{equation}\label{eq:wiener-S2-bis}
\left( \wt k(r)+\frac{2\wt k(r_0/2)}{r_0}  r\right)\, \exp\Big(- \sum_{k=1}^{m} b(2^{-k}r) \Big).
\end{equation}
Arguing similarly, we get
\begin{align}\label{eq:wiener-S1}
\Sigma_1 &\leq \exp\Big(- \sum_{k=1}^{m} b(2^{-k}r) \Big) \hm(r).
\end{align}
Therefore, combining \eqref{eq:wiener-iter-j}, \eqref{eq:wiener-S2}, \eqref{eq:wiener-S2-bis}  and \eqref{eq:wiener-S1}, we infer that
\begin{align}\label{eq:wiener-S}
\hm(2^{-m}r) \leq  \left(\hm(r)+ C \Big(\wt k(r) +\frac{2\wt k(r_0/2)}{r_0}  r\Big) \right)\, \exp\Big(- \sum_{k=1}^{m} b(2^{-k}r) \Big).
\end{align}
It is easy to see that 
$$
\int_{2^{-m} r}^r b(s)\, \frac{ds}{s} \leq 2^{n-2} \sum_{j=0}^{m-1} b(2^{-j}r),
$$
which can be used in  \eqref{eq:wiener-S} along with $K_0 \, \gamma(s) := \frac{1-k_0}{1-k_0+c_n} \gamma(s) \leq b(s)$ (using the fact that $\Cap_2(\overline{B(\xi,s)}, B(\xi, 2s))=c_n s^{n-2}$ for any $s>0$) to obtain 
\begin{align}\label{eq:wiener-osc-discrete}
\hm(2^{-m}r) \leq  \left( \hm(r)+ 2 \,  \Big(\wt k(r) +\frac{2\wt k(r_0/2)}{r_0}  r\Big)  \right) \exp\left(-\frac{K_0}{2^{n-2}} \int_{2^{-m} r}^r \frac{\Cap(\overline B_{s} \setminus \om)}{s^{n-2}}\, \frac{ds}{s} \right).
\end{align}

For any $\rho \leq r \leq r_0/2$, there exists $m_0 \in \bN$ such that $2^{-m_0-1} r \leq \rho <2^{-m_0} r$. Thus,  by \eqref{eq:wiener-osc-discrete} we deduce that
\begin{align*}\label{eq:wiener-osc}
&\osc_{B_\rho \cap \om } u  \leq \osc_{\d \om \cap B_{\rho}} u\\
&+\exp\left(- \frac{K_0}{2^{n-2}} \int_{2\rho}^r \frac{\Cap(\overline B_{s} \setminus \om)}{s^{n-2}}\, \frac{ds}{s} \right) \, \left(\osc_{B_{r}\cap \om} u -\osc_{\d \om \cap B_{r}} u + 2 \,  \Big(\wt k(r) +\frac{2\wt k(r_0/2)}{r_0}  r\Big)   \right),
\end{align*}
which, by \eqref{eq:bk-k3-u-k},  concludes the proof of Theorem  \ref{thm:wiener-osc}, since $\osc_{\d \om \cap B_{r}} u \geq 0$ and $u =\vphi$ on $\d \om\cap B_{r}$ in the Sobolev sense.
\end{proof}

\vv


As a corollary of the previous theorem we obtain the following Wiener-type criterion for continuity of solutions up to the boundary as well as  a modulus of continuity under the CDC.

\begin{theorem}[\bf Boundary continuity]\label{thm:wiener-cdc}
Under the assumptions of Theorem \ref{thm:wiener-osc}, if $u$ is the unique solution of \ref{boundvar} the following hold:
\begin{itemize}
\item[(i)] If $\xi \in \d \om$ and $\R^n \setminus \om$ is  thick at $\xi$, then $  \lim_{\om \ni x \to \xi} u(x)= \vphi(\xi)$ continuously.
\item[(ii)]  If $\vphi$ is continuous with a  modulus of continuity and $\d \om$  has  the  CDC, then $u$ is continuous in $\overline{\om}$ with a modulus of continuity depending on the one of $\varphi$ as well as the Stummel-Kato modulus of continuity of the data and the coefficients in the definition of $\wt k$.
\end{itemize}
\end{theorem}


\vvv

\section{Dirichlet and obstacle problems in Sobolev space}\label{sec:dirichlet}

In this section we will need to assume  the following standing (global) assumptions:
\begin{align*}
|b|^2, |c|^2, |d| \in  \mathcal{K}'(\om) \quad \textup{or} \quad b, c \in L^{n,\infty}(\om), d \in L^{\frac{n}{2},\infty}(\om).
\end{align*}

\subsection{Weak maximum principle}

\begin{theorem}\label{thm:maxprin-bd}
Let $\Omega \subset \R^{n}$ be an open and connected set and assume that either $b+c \in L^{n,q}(\om)$, for $q \in [n, \infty)$, or $b+c \in \mathcal{K}'(\om)$. If  $u \in Y^{1,2}(\om)$ is a subsolution of $Lu =0$, then  the following hold:
\begin{itemize}
\item[(i)] If \eqref{negativity} holds then
\begin{equation}\label{eq:maxprin-bd}
\sup_{\om}u \leq \sup_{\partial \Omega}u^+.
\end{equation}
\item[(ii)] If \eqref{negativity2} holds and $u^+ \in Y_0^{1,2}(\om)$, then  
\begin{equation}\label{eq:maxprin-cd}
\sup_{\om}u \leq 0.
\end{equation}
\end{itemize}
\end{theorem}

\begin{proof}
Set $\ell=\sup_{\partial \om}u^+$ and define $w=(u-\ell)^+ \in Y^{1,2}_0(\om)$. We apply Lemma $\ref{lem:main-splitting-Ln}$ to $w$, for $p= n$, $q \in [n. \infty)$, $h= b+c$, and $a=\lambda /2C_{s,q}$, to find $w_i \in Y^{1,2}_0(\om)$ and $\om_i \subset \om$,  $1 \leq i \leq m$, satisfying  \eqref{1exist}--\eqref{6exist}. In light of $\eqref{3exist}$, as $w \geq 0$, we have that $w_i \in Y^{1,2}_0(\om)$ is also non-negative. Recall also that $ \nabla w_i = \nabla u$ in $\om_i$. We will now proceed as usual. Indeed, using that $u$ is a subsolution along with \eqref{eqelliptic1}, \eqref{negativity}, \eqref{6exist}, and \eqref{eq:Lorentz-Sobolev-emb}, we infer
\begin{align*}
\lambda \| \nabla w_i\|_{L^2(\om)}^2 &\leq \int_\om A \nabla w_i \nabla w_i = \int_\om A\nabla u \nabla w_i \leq  \int_\om (b+c) \nabla u w_i \\
&= \sum_{j=1}^i \int_\om (b+c) \nabla w_j w_i \\
& \leq a C_{s,q} \| \nabla w_i\|_{L^2(\om)}^2 + a C_{s,q}\| \nabla w_i\|_{L^2(\om)} \sum_{j=1}^{i-1}\| \nabla w_j\|_{L^2(\om)},
\end{align*}
which implies
$$
\| \nabla w_i\|_{L^2(\om)} \leq \sum_{j=1}^{i-1}\| \nabla w_j\|_{L^2(\om)}.
$$
By the induction argument in the proof of Theorem \ref{thm:subCaccioppoli}, we get that for any $i =1,2,\dots, \kappa$, $\| \nabla w_i\|_{L^2(\om)} =0$, which we may sum in $i$ and use the condition \eqref{4exist} to obtain
$\| \nabla w\|_{L^2(\om)} =0$. Since $w \in Y^{1,2}_0(\om)$, by Lemma \ref{lem:sobolev-u-constant}, $w=0$. Therefore, $u \leq \ell$, which concludes the proof of (i).

To prove of (ii), we argue as above for $w= u^+ \in Y^{1,2}_0(\om)$ (i.e., $\ell=0$) and use \eqref{negativity2} instead of \eqref{negativity}, to get
\begin{align*}
\lambda \| \nabla w_i\|_{L^2(\om)}^2 &\leq  \int_\om (b+c)  u \nabla w_i = \sum_{j=i}^\kappa \int_\om (b+c)  w_j \nabla w_i\\
& \leq a C_{s,q} \| \nabla w_i\|_{L^2(\om)}^2 + a C_{s,q}\| \nabla w_i\|_{L^2(\om)} \sum_{j=i+1}^\kappa \| \nabla w_j\|_{L^2(\om)}.
\end{align*}
Thus, 
$$
\| \nabla w_i\|_{L^2(\om)} \leq \sum_{j=i+1}^{\kappa}\| \nabla w_j\|_{L^2(\om)},
$$
which, by the induction argument in Theorem \ref{thm:Caccioppoli2}, implies $\| \nabla w\|_{L^2(\om)} =0$, and so, \eqref{eq:maxprin-cd} readily follows.

The proof when $b+c \in \mathcal{K}'(\om)$ is analogous and the required adjustments are the same  as in the proof of Theorem \ref{thm:subCaccioppoli}. Details are omitted.
\end{proof}

A direct consequence of the weak maximum principles proved above is the following comparison principle:

\begin{corollary}\label{cor:comparison-princ}
Let $\Omega \subset \R^{n}$ be an open and connected set and assume that  either \eqref{negativity} or \eqref{negativity2} holds.  Assume also either $b+c \in L^{n,q}(\om)$, for $q \in [n, \infty)$, or $b+c \in \mathcal{K}'(\om)$. If $u \in Y^{1,2}(\om)$  is a supersolution of \eqref{eq:Lu=f-divg} and $v \in Y^{1,2}(\om)$ is a subsolution of \eqref{eq:Lu=f-divg} such that
$(v-u)^+ \in Y_0^{1,2}(\om)$, then we have that
$$
v \leq u \,\,\textup{in} \,\,\om.
$$	
\end{corollary}

\begin{proof}
Since $L(v-u) \leq 0$  and  $(v-u)^+ \in Y_0^{1,2}(\om)$, we apply  Theorem \ref{thm:maxprin-bd} (either (i) or (ii)) and obtain
$$\sup_{\om}(v-u) \leq 0,$$
which concludes our proof. 
\end{proof}

\vvv

\subsection{Dirichlet problem}
Let  $f: \om \to \R$, $g: \Omega \to \R^{n} $ and $\vphi: \om \to \R$, such that $f \in L^{2_*}(\Omega)$, $g \in L^2(\Omega)$, and $\vphi \in Y^{1,2}(\Omega)$. In this section we deal with the  {\it Dirichlet problem} 
\begin{equation}\label{boundvar}
\begin{cases}
Lu = f - \divv g  \\
u-\vphi\in Y^{1,2}_0(\Omega).
\end{cases}
\end{equation}
In particular, we show that it is well-posed assuming either \eqref{negativity} or \eqref{negativity2}. In fact, if we set $w=u-\vphi$, then, $w \in Y^{1,2}_0(\om)$, and (in the weak sense) it holds 
\begin{align*}
Lw&= Lu - L \vphi \\
&= ( f - c \nabla \vphi - d \vphi) - \divv(g + A \nabla \vphi  + b \vphi) \\
& =: \hat f -  \divv \hat g.
\end{align*}
Thus, \eqref{boundvar} is readily reduced to the following inhomogeneous Dirichlet problem with zero boundary data:
\begin{equation}\label{interiorvar}
\begin{cases}
Lu = f -  \divv g \\
u\in Y_0^{1,2}(\Omega).
\end{cases}
\end{equation}

\vv

Well-posedness of the Dirichlet problem \eqref{interiorvar}  with solutions $u \in W_0^{1,2}(\Omega)$ instead of $u \in Y_0^{1,2}(\Omega)$ in unbounded domains was shown in \cite[Theorem 1.4]{BM} for data $f, g \in L^2(\Omega)$, but with a stronger negativity assumption than $ \divv b + d \leq 0$. Namely, it was assumed that there exists $\mu < 0$ such that $ \divv b + d \leq \mu$. This was necessary exactly because  they required the solutions to be in $W_0^{1,2}(\Omega)$ as opposed to $Y_0^{1,2}(\Omega)$. It is worth mentioning that \eqref{negativity2} was not treated at all.

In the following theorem we follow the proof of  \cite[Theorem 1.4]{BM} adjusting the arguments to the weaker negativity assumption $ \divv b + d \leq 0$ and the Sobolev space $Y_0^{1,2}(\Omega)$. Moreover, our argument works for Lorentz spaces as well as the Stummel-Kato class.

\begin{theorem}\label{thm:existvar}
Let $\Omega \subset \R^{n}$ be an open and connected set and assume that either $b+c \in L^{n,q}(\om)$, for $q \in [n, \infty)$, or $|b+c|^2 \in \mathcal{K}'(\om)$. If  $g_i \in L^2(\Omega)$ for $1 \leq i \leq n$,  $f \in L^{2_*}(\om)$, and either \eqref{negativity} or \eqref{negativity2} holds, then the Dirichlet problem \eqref{interiorvar} has a unique solution $u \in Y_0^{1,2}(\Omega)$ satisfying
\begin{equation}\label{eq:Dir-quant-bound}
\| u\|_{Y^{1,2}(\om)}\lesssim \|f\|_{L^{2_*}(\om)} + \|g\|_{L^2(\om)},
\end{equation}
where the implicit constant depends only on $\lambda$, $\Lambda$, and either $C_{s,q}$ and $\|b+c\|_{L^{n,q}(\om)}$ or  $C_s'$ and $\vartheta_{\om}(|b+c|^2)$.
\end{theorem}

\begin{proof}

To demonstrate that \eqref{eq:Dir-quant-bound} holds assuming that such a solution exists, it is enough to repeat the  argument  in the proof of Theorem \ref{thm:maxprin-bd} applying Lemma \ref{lem:main-splitting-Ln} to $u \in Y_0^{1,2}(\Omega)$. The difference is that we should  use that $u$ is a solution of \eqref{eq:Lu=f-divg} instead of a subsolution of $Lu =0$ and thus, we  pick up two terms related to the  interior data exactly as in the proofs of Theorems \ref{thm:subCaccioppoli} and \ref{thm:Caccioppoli2}. Similar (but easier) manipulations along with the same induction argument conclude \eqref{eq:Dir-quant-bound}. We omit the details.

\vv

To show that  \eqref{interiorvar} has a unique solution it is enough to apply the comparison principle given in Corollary \ref{cor:comparison-princ}.

 \vv
 
Existence of solutions of \eqref{interiorvar} is also based on \eqref{eq:Dir-quant-bound}.  We first assume that $\om$ is a bounded  domain and solve the variational problem \eqref{interiorvar} in $W^{1,2}_0(\om)$ with interior data $f \in L^2(\om) \cap L^{2_*}(\om)$ and $g \in L^2(\om)$.

Let $u \in W_0^{1,2}(\Omega)$ and note that by \eqref{eqelliptic1} and  $ \divv b + d \leq 0$ we have
\begin{align}\label{eq:s-accr-1}
\LL(u,u) &= \int_{\om}A\nabla u\nabla u + (b-c)u\nabla u -du^2 
\geq \lambda \| \nabla u\|^2_{L^2(\om)} - \int_\om (b+c) \cdot \nabla u \, u.
\end{align}

If $(b+c) \in L^{n,q}(\om)$, for $\delta>0$ sufficiently small to be chosen, we can find $\zeta \in L^{\infty}(\om)$ which support has finite Lebesgue measure, such that $\| (b+c)^2 - \zeta\|_{L^{n,q}(\om)}< \delta$. Thus, by \eqref{eq:Lorentz-Sobolev-emb}, 
\begin{align}\label{eq:s-accr-2}
\int_\om (b+c) \cdot \nabla u \, u &\leq C_{s,q}  \|b+c- \zeta\|_{L^{n,q}(\om)} \| \nabla u \|_{L^2(\om)}   \|  u \|_{L^{2^*}(\om)} +\int_\om \zeta \cdot \nabla u \, u\\
& \leq \delta C_{s,q} \| \nabla u \|^2_{L^2(\om)} +\int_\om \zeta \cdot \nabla u \, u.\notag
\end{align}
If $\ve>0$ small enough to be chosen, then by \eqref{eq:s-accr-1}, \eqref{eq:s-accr-2}, and Young inequality, we infer
$$
\LL(u,u) \geq (\lambda - \delta C_{s,q} - \frac{\ve}{2}) \| \nabla u\|^2_{L^2(\om)} - \frac{1}{2\ve} \int_\om |\zeta|^2 u^2.
$$
We now choose $\ve=\frac{\lambda}{4}$ and $\delta= \frac{\lambda}{4 C_{s,q}}$, and obtain
\begin{align}\label{eq:s-accr-3}
\LL(u,u) \geq \frac{\lambda}{2} \| \nabla u\|^2_{L^2(\om)} - \frac{2\|\zeta\|_{L^\infty(\om)}^2}{\lambda} \|  u \|^2_{L^2(\om)}=: \frac{\lambda}{2} \| \nabla u\|^2_{L^2(\om)} -\sigma  \|  u \|^2_{L^2(\om)}.
\end{align}

If $|b+c|^2 \in \mathcal{K}(\om)$, then we apply Cauchy-Schwarz and \eqref{eq:cor-Kato-emb-glob},
\begin{align}\label{eq:s-accr-2}
\int_\om (b+c)  \nabla u \, u &\leq  \left( \int_\om |b+c|^2 |u|^2 \right)^{1/2} \| \nabla u\|_{L^2(\om)} \\
& \leq \ve \| \nabla u \|^2_{L^2(\om)} + C_\ve  \| \nabla u\|_{L^2(\om)}  \| u\|_{L^2(\om)}\\
& \leq 2\ve \| \nabla u \|^2_{L^2(\om)} + C_\ve'\| u\|_{L^2(\om)}^2.\notag
\end{align}
If we choose $\ve=\frac{\lambda}{4}$, we get
$$
\LL(u,u) \geq \frac{\lambda}{2} \| \nabla u\|^2_{L^2(\om)} - C_\ve'\  \|  u \|^2_{L^2(\om)}=: \frac{\lambda}{2} \| \nabla u\|^2_{L^2(\om)} -\sigma  \|  u \|^2_{L^2(\om)}.
$$

Let us denote $H=L^2(\om)$, $V=W^{1,2}_0(\om)$ and its dual $V^*=W^{-1,2}(\om)$ and  define 
$$
L_\sigma w:= Lw +\sigma w.
$$
By  \eqref{eq:s-accr-3}, its associated bilinear form is clearly coercive and bounded in $V$. As $f \in  H$ and $g \in H$,  by Lax-Milgram theorem, there exists a unique solution to the problem 
\begin{equation}\label{boundvar-W12}
\begin{cases}
L_\sigma u = f - \divv g  \\
u \in V.
\end{cases}
\end{equation}
and so, $L_\sigma$ has a bounded inverse $L_\sigma^{-1}: V^* \to V$.

If $J: V \to V^*$ is an embedding  given by 
\begin{equation}\label{eq:Embedd}
J v = \int_\om u v, \quad v \in V,
\end{equation}
$I_2: V \to H$ is the natural embedding and $I_1: H\to V^*$ is an embedding given also by \eqref{eq:Embedd}, we can write  $J= I_1\circ I_2$.  It is clear that $J$ is compact as $I_2$ is compact and $I_1$ is continuous. 

The interior data naturally induces a linear functional on $V$ by
$$
F(v) = \int_{\om}fv + g \cdot \nabla v, \quad \textup{for}\,\, v \in V,
$$
so we wish to solve the equation $Lu =F$. This is  is equivalent to $L_\sigma u - \sigma  J u = F$, which in turn, can be written as
\begin{equation}\label{eq:inhom-Fredholm} 
u - \sigma L_\sigma^{-1} J u = L_\sigma^{-1}  F.
\end{equation}
But $L_\sigma^{-1} J$ is compact as $J$ is compact and $L_\sigma^{-1}$ is continuous. Thus, by the Fredholm alternative,  \eqref{eq:inhom-Fredholm} has a unique solution if and only if $w=0$ is the unique function in $V$ satisfying $w - \sigma L_\sigma^{-1} J w =0$ (or else $Lw=0$). But this readily follows from the weak maximum principle  in Theorem \ref{thm:maxprin-bd} and thus, a solution of \eqref{interiorvar} exists in bounded domains.

If $\om$ be an unbounded domain, we can find a sequence of function $f_k \in C^\infty_c(\om)$ such that $f_k \to f$ in $L^{2_*}(\om)$, and then for $j \in \bN$  define
$$
\om_j := \{ x \in \om \cap B(0, j): \dist(x, \d \om) > j^{-1} \}.
$$
Since $f_k \in L^2(\om)\cap L^{2_*}(\om)$ and $\om_j$ is a bounded open set,  by \eqref{boundvar-W12}, there exists $u_{k,j} \in W^{1,2}_0(\om_j) = Y^{1,2}_0(\om_j)$ such that $L u_{k,j} = f_k - \divv g$ in $\om_j$. If we extend $u_{k,j} $ by zero outside $\om_j$,  by \eqref{eq:Dir-quant-bound},  we will have
$$\| u_{k,j} \|_{Y^{1,2}(\om)} \lesssim  \|f_k\|_{L^{2_*}(\om)} + \|g\|_{L^2(\om)},$$
that is, $u_{k,j}$ is a uniformly bounded sequence in $Y_0^{1,2}(\om)$ with bounds independent of $j$ and $k$. Thus, since $Y_0^{1,2}(\om)$ is weakly compact, there exists a subsequence $\{u_{k, j_m}\}_{m \geq 1}$ converging weakly to a function $u_k \in Y^{1,2}_0(\om)$. Notice also that if $\vphi \in C^\infty_c(\om)$, then for $j$ large enough, it also holds $\vphi \in C^\infty_c(\om_j)$. Therefore, since $L u_{k, j} = f_k - \divv g$ in $\om_j$ for any $j \geq 0$, and $u_{k,j_m} \to u_k$ weakly in $Y_0^{1,2}(\om)$ as $m \to \infty$, we obtain
\begin{equation}\label{eq:var-Luj-zero-limit}
\langle f_k, \vphi\rangle +\langle g, \nabla \vphi \rangle =\LL(u_{k,j_m}, \vphi) \xrightarrow{m \to \infty} \LL(u_k, \vphi),\quad \textup{for all}\,\,  \vphi \in C^\infty_c(\om),
\end{equation}
i.e., $L u_{k} = f_k - \divv g$ in $\om$. In addition, since $u_k$ is the  weak limit of $u_{k,j_m}$, for $k$ large enough, it satisfies
$$\| u_{k} \|_{Y^{1,2}(\om)} \lesssim  \|f_k\|_{L^{2_*}(\om)} + \|g\|_{L^2(\om)} \lesssim  \|f\|_{L^{2_*}(\om)} + \|g\|_{L^2(\om)},$$
with implicit contants independent of $k$. Once again by the weak compactness of $Y_0^{1,2}(\om)$, we can find a subsequence $\{u_{k_m}\}_{m \geq 1}$ converging weakly to a function $u \in Y^{1,2}_0(\om)$. Thus, since $L u_{k} = f_k - \divv g$ in $\om$,  $u_{k_m} \to u$ weakly in $Y_0^{1,2}(\om)$ and $f_{k_m} \to f$ in $L^{2_*}(\om)$-norm, we obtain
$$
\LL(u, \vphi) =\langle f, \vphi \rangle +\langle\nabla g, \vphi \rangle, \quad \textup{for all}\,\,  \vphi \in C^\infty_c(\om).
$$
The proof is now concluded.
\end{proof}

An immediate corollary of the last  theorem in light of the considerations at the beginning of this section is the following:

\begin{theorem}\label{thm:Dirichlet-non-zero-bdrydata}
Let $\Omega \subset \R^{n}$ be an open and connected set and assume that either $b+c \in L^{n,q}(\om)$, for $q \in [n, \infty)$, or $|b+c|^2 \in \mathcal{K}'(\om)$. If  $\vphi \in Y^{1,2}(\Omega)$, $g_i \in L^2(\Omega)$ for $1 \leq i \leq n$,  $f \in L^{2_*}(\om)$, and either \eqref{negativity} or \eqref{negativity2} holds,  then the Dirichlet problem $\eqref{boundvar}$ has a unique solution $u \in Y^{1,2}(\Omega)$ satisfying 
\begin{equation}\label{eq:Dir-estimate-nonzero}
\| u\|_{Y^{1,2}(\om)}\leq  \|\vphi\|_{Y^{1,2}(\om)}+\|f\|_{L^{2_*}(\om)} + \|g\|_{L^2(\om)},
\end{equation}
with the implicit constant depending only on $\lambda$, $\Lambda$, and either $C_{s,q}$ and $\|b+c\|_{L^{n,q}(\om)}$  or  $C_s'$ and $\vartheta_{\om}(|b+c|^2)$.
\end{theorem}

\vvv

\subsection{Obstacle problem}\label{sec:obstacle}

In this subsection, we let   $\Omega$ be a bounded and open set, and assume that either \eqref{negativity} or \eqref{negativity2} is satisfied, and also that either $b+c \in L^{n,q}(\om)$, for $q \in [n, \infty)$, or $|b+c|^2 \in \mathcal{K}'(\om)$ holds. 


\begin{definition}
Let  $\psi, \phi \in W^{1,2}(\Omega)$ such that $\phi \geq \psi$ on $\d \om$ in the $W^{1,2}$ sense. Let us also define the convex set
\begin{equation}\label{eq:convex-obstacle}
\bK:=\{v \in W^{1,2}(\Omega): v \geq \psi \,\,\textup{on} \,\,\Omega \,\textup{ in the $W^{1,2}$ sense and} \,\, v - \phi \in W^{1,2}_0(\Omega)\}.
\end{equation}
We say that $u$ is a \textit{solution to the obstacle problem} in $\Omega$ with obstacle $\psi$ and boundary values $\phi$ and we write $u \in  \mathcal{K}_{\psi, \phi}(\om)$, if $u \in \bK$ and 
\begin{equation}\label{eq:solution-obstacle}
\mathcal{L}(u, v-u) \geq 0, \,\, \textup{for all}\, v \in \bK.
\end{equation}
\end{definition}

This problem can be reduced to the one with zero boundary data as follows: Let us define the convex set
\begin{equation}\label{eq:convex-obstacle-zerobdry}
\bK_0:=\{w \in W^{1,2}_0(\Omega): w \geq \psi-\phi \,\,\textup{on} \,\, \om \,\textup{  in the $W^{1,2}$ sense}\}.
\end{equation}
Suppose that $u \in \mathcal{K}_{\psi, \phi}(\om)$ and write
\begin{align*}
u&= u_0 + \phi, \quad\textup{for}\,\,v_0 \in \bK_0\\
v&= v_0 + \phi,\quad\textup{for} \,\,v_0 \in \bK_0.
\end{align*}
Thus, 
$$ \mathcal{L}(u_0,v_0 -u_0) \geq \langle  f ,v_0 -u_0 \rangle - \mathcal{L}(\phi,v_0 -u_0),$$
and since $\langle F,  \eta  \rangle:=\langle  f, \eta \rangle - \mathcal{L}(\phi,\eta)$,  $\eta \in W^{1,2}_0(\om)$, defines an element $F \in W^{-1,2}(\om)$, it is enough to prove the following theorem:

\begin{theorem}\label{thm:v.i.-unilateral-lower}
Let  $ \psi$ be measurable such that $\psi \leq 0$ on $\d \om$ in the $W^{1,2}$ sense. Define
\begin{equation}\label{eq:convex-psi-obstacle-zerobdry}
\bK_\psi:=\{w \in W^{1,2}_0(\Omega): w \geq \psi \,\,\textup{in} \,\, \om  \,\textup{  in the $W^{1,2}$ sense}\}.
\end{equation}
Given $F \in W^{-1,2}(\om)$, there exists a unique $u \in \bK_\psi$ such that
\begin{equation}\label{eq:v.i.-super}
\mathcal{L}(u, v-u) \geq \langle F,  v-u \rangle, \quad \textup{for all}\,\, v \in \bK_\psi.
\end{equation}
Moreover, $u$ is the minimal among all $w \in W^{1,2}(\om)$ that are supersolutions of $L w =F$ and satisfy $w \geq \psi$ in $\om$ and $w \geq 0 $ on $\d \om$ in the $W^{1,2}$ sense.
\end{theorem}
\begin{proof}
By  the weak maximum principle proved in Theorem \ref{thm:maxprin-bd}, our theorem  follows from Theorem 4.27 in \cite{Tr} and the Corollary right after it.
\end{proof}

An important consequence of this theorem is the following:
\begin{corollary}\label{cor:min(u,v)}
Let $\om \subset \R^n$ be an open set (not necessarily bounded). If $u$ and $v$ are  supersolutions of $Lw =F$ in $\om$, then $\min(u,v)$ is a supersolution of the same equation.
\end{corollary}

\begin{proof}
If $\om$ is bounded, the proof is a consequence of Theorem \ref{thm:v.i.-unilateral-lower} and can be found in  \cite[Chapter II, Theorem 6.6]{KrS}. Let $\om$ be an unbounded open set and assume that $u$ and $v$ are  supersolutions of $Lw =F$ in $\om$. Since they are supersolutions of the same equation  in any bounded open set $D \subset \om$, $\min(u,v)$ is a supersolution in any such $D$ as well. Using a partition of unity, this yields that $\min(u,v)$ is a supersolution in $\om$.
\end{proof}


The proof of the following theorem can be found for instance  in  \cite[Chapter II, Theorem 6.9]{KrS}.
\begin{theorem}\label{thm:coinc-set}
Let $u$ be the unique solution obtained in Theorem \ref{thm:v.i.-unilateral-lower} for $\psi \in W^{1,2}(\om)$.  Then there exists a non-negative Radon measure so that
\begin{equation}\label{eq:cap-meas}
Lu = f+ \mu, \quad \textup{in} \,\, \om,
\end{equation}
with 
$$
\supp (\mu) \subset I:= \om \setminus \{x \in \om: u(x) > \psi(x)\}.
$$
In particular, 
$$
Lu = f \quad \textup{in}\,\, \{x \in \om: u(x) > \psi(x)\}.
$$
\end{theorem}

\vvv


\section{Green's Functions in unbounded domains}\label{sec:green}

Here we construct  the Green's function associated with an elliptic operator given by $\eqref{operator}$ satisfying  either negativity assumption following the approach of Hofmann and Kim  \cite{HK} along with its variation due to Kang and Kim \cite{KK}.

\subsection{Construction of Green's functions} 

Before we start, we should mention  that the equation formal  adjoint operator of $L$ is given by
 \begin{equation}\label{adjoint-operator}
 L^t u = -\divv (A \cdot \nabla u - c u) + b \cdot \nabla u - du=0,
\end{equation}
with corresponding bilinear form
\begin{equation}\label{adjoint-bilinear}
\LL^t(u, \vphi ) = \int_{\om}(A^t\nabla u -cu)\nabla \vphi  -(du-b\nabla u)\vphi.
\end{equation}
Moreover,  if $\LL$ satisfies \eqref{negativity}, then its adjoint satisfies \eqref{negativity2} and vice versa.

In the current section, we will require the following conditions to hold:
\begin{align*}
|b|^2, |c|^2, |d| \in  \mathcal{K}'(\om) \quad \textup{or} \quad b, c \in L^{n,\infty}(\om), d \in L^{\frac{n}{2},\infty}(\om).
\end{align*}

\begin{theorem}\label{thm:green-contruction}
Let $\om \subset \R^n$ be an open and connected set and $L$ be an operator given by \eqref{operator} so that \eqref{negativity2} holds. For a fixed $y \in \om$, there exists the Green's function $G(x, y)\geq 0$ for a.e. $x \in \om \setminus\{y\}$ with the following properties:
\begin{enumerate}
\item $G(\cdot, y) \in Y^{1,2}(\om \setminus B_r(y))$ for  all  $r>0$ and vanishes on $\d \om$.\label{item:green2}
\item If $f \in L^{\frac{n}{2},1}(\om) $ and  $g \in  L^{n,1}(\om)$, we have that
\begin{equation}\label{eq:Green-1}
u(x)=\int_\om G(y,x) \,f(y) \,dy + \int_\om \nabla_y G(y, x) \,g(y) \,dy,
\end{equation}
is a solution of $L^t u =f - \divv g$  in $\om$ and $u \in Y^{1,2}_0(\om)$ satisfying $\| u\|_{L^\infty(\om)}\lesssim \|f\|_{L^{\frac{n}{2},1}(\om) } + \|g\|_{L^{n,1}(\om) } $.\label{item:green3}
\item For any other Green's function $\widehat G(x,y)$ satisfying (3), it holds $G(x,y)=\widehat G(x,y)$ for a.e. $ x \in \om \setminus \{y\}$.\label{item:green4}
\item $G(\cdot, y) \in W^{1,1}_\loc(\om)$ and  for any $\eta_y \in C^\infty_c(B_r(y))$ such that  $\eta_y=1$ in $B_{r/2}(y)$, for $r>0$, it holds that \label{item:green1}
\begin{equation}\label{eq:green-phi}
\mathcal{L}(G(\cdot, y), (1-\eta_y)\vphi)=0, \quad \textup{for any}\,\, \vphi \in C^\infty_c(\om).
\end{equation}
\end{enumerate}
If we set $d_y= \dist(y, \d \om)$ \textup{(}$d_y=\infty$ if $\om=\R^n$\textup{)}, the following bounds are satisfied:
\begin{align}
\| G(\cdot, y) \|_{Y^{1,2}(\om \setminus B_r(y))} \lesssim r^{1-\frac{n}{2}},&\quad\textup{for any}\,\, r>0,\label{eq:Green-2}\\
\| G(\cdot, y) \|_{L^{p}(B_r(y))} \lesssim_p  r^{2-n+\frac{n}{p}}, &\quad\textup{for all}\,\, r<d_y \,\,\textup{and}\,\, p \in \big[1, \frac{n}{n-2} \big),\label{eq:Green-3}\\
\| \nabla G(\cdot, y) \|_{L^{p}(B_r(y))} \lesssim_p  r^{1-n+\frac{n}{p}}, &\quad\textup{for all}\,\, r<d_y , \,\,\textup{and}\,\, p \in \big[1, \frac{n}{n-1} \big),\label{eq:Green-3-bis}\\
|\{x\in \om: G(x,y) >t \}| \lesssim t^{-\frac{n}{n-2}}, &\quad \textup{for all}\,\, t>0,\label{eq:Green-4}\\
|\{x\in \om: \nabla_x G(x,y) >t \}| \lesssim t^{-\frac{n}{n-1}}, &\quad \textup{for all}\,\, t> 0,\label{eq:Green-5}
\end{align}
The implicit constants depend only on $\lambda$,  $\Lambda$, and either $C_{s,q}$ and $\|b+c\|_{L^{n,q}(\om)}$, or $C_s'$ and $\vartheta_\om(|b+c|^2)$.
If we also assume  $|b+c|^2 \in \khalf(\om)$, then 
\begin{equation}
G(x, y) \lesssim \frac{1}{|x-y|^{n-2}},\quad\textup{for all} \,\,x \in \om\setminus \{y\}. \label{eq:Green-pointwise}
\end{equation}
where the implicit constant depends also on $C_{|b+c|^2, \om}$.
\vv

If $|b+c|^2 \in \khalf(\om)$, we can  construct the Green's function $G^t(x, y)$ associated with the operator $L^t$ which is non-negative for a.e.  $x\in \om \setminus \{ y\}$ and  satisfies the analogous properties \eqref{item:green2}-\eqref{item:green1} and the bounds \eqref{eq:Green-2}-\eqref{eq:Green-pointwise}. The implicit constants depend on $\lambda$,  $\Lambda$, $C_s'$ and $C_{|b+c|^2, \om}$, and, in the pointwise bounds, on $\|b+c\|_{L^{n,q}(\om)}$, or $C_s'$ and $\vartheta_\om(|b+c|^2)$ as well.  Moreover, if $b, c \in L^{n,q}(\om)$, $d \in L^{\frac{n}{2},q}(\om)$, for  $q \in [n, \infty)$, or $|b|^2, |c|^2, |d| \in  \mathcal{K}'(\om)$,  it  holds that
\begin{align}\label{eq:Green-symmetry}
G^t(x,y)=G(y,x), \quad \textup{ for a.e.}\,\,  (x,y) \in \om^2\setminus \{x \neq y\},
\end{align}
and
\begin{equation}\label{eq:Green-1-bis}
u(x)=\int_\om G^t(x, y) \,f(y) \,dy + \int_\om \nabla_y G^t(x, y) \,g(y) \,dy,  \,\,\textup{for all} \, x \in \om.
\end{equation}
\end{theorem}

\begin{proof}
Given a point $y \in \om$, if $\om_\rho(y)= \om \cap B_\rho(y)$, we define 
$$
f_{\rho}(x,y) =|B_{\rho}(y)|^{-1} {\bf 1}_{\om_\rho(y)}(x), \quad x\in \om.
$$
Since $L$ satisfies \eqref{negativity2} and $f_\rho(\cdot, y) \in L^\infty(\om)$ with bounded support, we may apply Theorem \ref{thm:existvar} (ii) to find a function $G_\rho(\cdot, y) \in Y^{1,2}_0(\om)$ so that 
\begin{equation}\label{eq:green-rho-var}
\mathcal{L}(G_\rho(\cdot, y), \varphi)= \int f_{\rho}(\cdot,y) \, \varphi, 
\end{equation}
for any $\vphi \in C^\infty_c(\om)$, with global bounds
\begin{align}\label{eq:Green-rho-Dirbds}
\|G_\rho(\cdot, y)\|_{Y^{1,2}(\om)} &\lesssim |B_\rho(y)|^{\frac{2-n}{2n}}.
\end{align}
Note that $G_{\rho}(\cdot,y) \in Y^{1,2}_0(\om)$ and is an $L$-supersolution. If we apply the maximum principle given in Theorem \ref{thm:maxprin-bd} (ii),we get that  $G_{\rho}(\cdot,y) \geq 0$ in $\om$. 

Let now $f \in L^\infty(\om)$ and $g \in L^\infty(\om)$ so that $|\supp(f)|+|\supp(g)|<\infty$. Then, by Theorem \ref{thm:existvar}, there exists $u \in Y^{1,2}_0(\om)$ such that 
\begin{equation}\label{eq:green-dual-var}
\mathcal{L}^t(u, \psi)= \int f \psi +\int g \nabla \psi \quad  \textup{for all}\,\, \psi \in C^\infty_c(\om),
\end{equation}
satisfying 
\begin{align}\label{eq:green-dual-u-bds}
\|u\|_{Y^{1,2}(\om)} \lesssim \| f \|_{L^{2_*}(\om)} &+\| g \|_{L^{2}(\om)}\\
& \leq |\supp(f)|^{\frac{n+2}{2n}} \|f\|_{L^\infty(\om)}+ |\supp(g)|^{\frac{1}{2}} \|g\|_{L^\infty(\om)}.\notag
\end{align}

Remark here that, by the density of $C^\infty_c(\om)$ in $Y^{1,2}_0(\om)$, both \eqref{eq:green-rho-var} and \eqref{eq:green-dual-var} can be extended to test functions $\varphi \in Y^{1,2}_0(\om)$. So, if we set $\varphi= u$ in \eqref{eq:green-rho-var} and $\psi = G_{\rho}(\cdot,y)$ in \eqref{eq:green-dual-var}, we obtain that
\begin{equation}\label{eq:green-rho-dual-average}
 \int G_{\rho}(x,y) f(x)\,dx +\int \nabla_x G_{\rho}(x,y) g(x)\,dx= \avint_{\om_\rho(y)} u(x)\,dx.
\end{equation}

For $r>0$ fixed, assume that $\supp(f) \subset \om_r(y)$, $g=0$, and let $\rho < r/2$. Since $u_f$  is in $Y^{1,2}(\om_r(y))$, vanishes on $B_r(y) \cap \d \om$, and satisfies $L^t u_f = f$ in $\om_r(y)$, by  Theorem  \ref{thm:Moser} (1) with $M=0$, we obtain
\begin{align*}
\|u_f\|_{L^\infty(\om_{\frac{r}{2}}(y))} &\lesssim  r^{-\frac{n}{2}} \|u_f \|_{L^2(\om_r(y))} + r^2 \|f\|_{L^\infty(\om_r(y))}\lesssim  r^2 \|f\|_{L^\infty(\om_r(y))},
\end{align*}
where in the penultimate inequality we used H\"older inequality and \eqref{eq:green-dual-u-bds}. Similarly, if $f=0$,  $\supp(g) \subset \om_r(y)$, and $\rho < r/2$, since $u_g \in Y^{1,2}(\om_r(y))$ that vanishes on $B_r(y) \cap \d \om$ and  $L^t u_g = -\divv g$ in $\om_r(y)$, 
\begin{align*}
\|u_g\|_{L^\infty(\om_{\frac{r}{2}}(y))} &\lesssim  r^{-\frac{n}{2}} \|u_g \|_{L^2(\om_r(y))} + r \|g\|_{L^\infty(\om_r(y))}\lesssim  r \|g\|_{L^\infty(\om_r(y))}.
\end{align*}

By \eqref{eq:green-rho-dual-average}, duality considerations, and the latter two estimates,  we have that  for all  $r>0$ and $\rho < r/2$,
\begin{align}
\|G_\rho(\cdot, y)\|_{L^1(\om_r(y))} &\lesssim  r^2,\label{eq:green-rho-L1}\\
\|\nabla G_\rho(\cdot, y)\|_{L^1(\om_r(y))} &\lesssim  r.\label{eq:green-rho-nabla-L1}
\end{align}
In fact, arguing similarly, we can prove that for all  $r>0$, $\rho < r/2$, and $q \in [1, \frac{n}{n-2})$,
 \begin{align}
\|G_\rho(\cdot, y)\|_{L^q(\om_r(y))}& \lesssim  r^{2-n +\frac{n}{q}},\label{eq:green-rho-Lq}\\
\|\nabla G_\rho(\cdot, y)\|_{L^q(\om_r(y))}& \lesssim  r^{1-n +\frac{n}{q}}.\label{eq:green-rho-Lq}
\end{align}

To avoid an early use of the pointwise bounds and thus, of the assumption $|b+c|^2 \in \khalf(\om)$, we will need the following auxiliary lemma. 
\begin{lemma}\label{lem:reverseHold}
 Let $\Omega \subset \R^{n}$ be an open set and $L$ be the operator given by \eqref{operator} that satisfies either \eqref{negativity} or \eqref{negativity2}. Let $B_s=B(x,s)$ be a ball of radius $s$ centered at $x \in \om$ such that $3B_s \subset \om$ and $u \in Y^{1,2}(\om \setminus B_s)$ be a solution of $Lu =0$ in $\om \setminus B_s$  that vanishes on $\d \om$. Then for any $r \geq 4s$ we have
 \begin{equation}\label{eq:ReverseHolder}
 \int_{\om \cap (B_{2r} \setminus B_{r/3})} |u|^2 \lesssim\frac{1}{r^n} \left( \int_{\om \cap (B_{3r} \setminus B_{r/4})} |u| \right)^2,
 \end{equation}
 where the implicit constants depend only on $\lambda$,  $\Lambda$, $\|b+c\|_{L^n(\om;\, \R^n)}$, and $C_{s,q}$.
\end{lemma}

\begin{proof}
The proof can be found in \cite[Lemma 3.19]{KSa} with the difference that we use Theorems \ref{thm:boundCaccioppoli} instead of  \cite[Lemma 3.18]{KSa} that only holds for $r \leq 1$.
\end{proof}

For fixed $r >0$ and $\rho \in (0, r/6)$ we let $\eta \in C^\infty(\R^n)$ so that 
$$
0 \leq \eta \leq 1, \quad \eta \equiv 1\,\, \textup{on}\,\, \R^n \setminus B_r(y), \quad  \eta \equiv 0\,\, \textup{on}\,\, B_{r/2}(y), \quad \textup{and}\quad  |\nabla \eta| \leq \frac{4}{r}.
$$
Thus, by Theorem \ref{thm:boundCaccioppoli}, since $L G_\rho(\cdot,y) = 0$, in $\om \setminus B_{r/2}(y)$,
\begin{align}\label{eq:green-rho-gradL2-away}
\|\nabla G_\rho (\cdot, y) \|^2_{L^{2}(\om \setminus B_r(y))} &\leq \int_\om |\eta \nabla G_\rho(\cdot, y) |^2  \overset{\eqref{eq:bCacciop}}{\lesssim} \int_\om |  G_\rho(\cdot, y) \nabla \eta |^2 \\
&\lesssim \frac{1}{r^2} \int_{\om \cap (B_{r}(y) \setminus B_{r/2}(y))} G_\rho(\cdot, y)^2 \notag\\
&\overset{\eqref{eq:ReverseHolder}}{\lesssim}\frac{1}{r^{n+2}} \left(\int_{\om \cap (B_{2r}(y) \setminus B_{r/4}(y))} G_\rho(\cdot, y) \right)^2 \overset{\eqref{eq:green-rho-L1}}{\lesssim} r^{2-n},\notag
\end{align}
which, in turn, by Sobolev embedding theorem, implies that for $0<\rho < r/6$,
\begin{equation}\label{eq:green-rho-L2*-away}
\|G_\rho(\cdot, y) \|_{L^{2^*}(\om \setminus B_r(y))} \leq \|G_\rho(\cdot, y)  \eta\|_{L^{2^*}(\om)} \lesssim \|\nabla(G_\rho(\cdot, y)  \eta)\|_{L^{2^*}(\om)} \lesssim r^{1-\frac{n}{2}}.
\end{equation}

On the other hand, for $\rho \geq r/6$,  by \eqref{eq:Green-rho-Dirbds}, we have that
\begin{align}
&\|G_\rho (\cdot, y) \|_{Y^{1,2}(\om \setminus B_r(y))} \leq \|G_\rho(\cdot, y) \|_{Y^{1,2}(\om)} \lesssim |B_{\rho/6}(y)|^{\frac{2-n}{n}} \lesssim r^{2-n}.
\label{eq:green-rho-Y12-away2}
\end{align}

Therefore, if we apply \eqref{eq:green-rho-gradL2-away}, \eqref{eq:green-rho-L2*-away}, and \eqref{eq:green-rho-Y12-away2}, we obtain that for any $r>0$, there exists a constant $C(r)$  depending on $r$ so that  
\begin{align*}
&\|G_\rho (\cdot, y) \|_{Y^{1,2}(\om \setminus B_r(y))} \leq C(r),
\end{align*}
uniformly in  $\rho>0$. So, by a diagonalization argument and weak compactness of $Y^{1,2}_0$, there exists a sequence  $\{\rho_m\}_{m=1}^\infty$ that converges to zero as $m \to \infty$ such that for all $r>0$,
\begin{align}
&G_{\rho_m}(\cdot,y) \rightharpoonup G(\cdot, y) \,\, \textup{in}\,\, Y_0^{1,2}(\om \setminus B_{r}(y)), \quad \textup{as}\,\,m \to \infty,\label{eq:green-convergY12}
\end{align}
where $G(\cdot,y)  \in Y_0^{1,2}(\om \setminus B_r(y))$. Moreover, by \eqref{eq:green-rho-Y12-away2},
\begin{align}
&\|G(\cdot, y) \|_{Y^{1,2}(\om \setminus B_{r}(y))} \lesssim r^{2-n},\quad\textup{for all}\,\, r>0\label{eq:green-boundY12}.
\end{align}

If we follow the proof of inequalities (3.21) and (3.23)  in \cite{HK} using the the same considerations that lead to the proof of the estimates for $G_\rho(\cdot, y)$ away from the pole, we can show that 
\begin{align}\label{eq:green-rho-Lorentz1}
&|\{x\in \om: G_\rho(x,y) >s \}| \lesssim s^{-\frac{n}{n-2}}, \quad \textup{for all}\,\, s> 0,\\
&|\{x\in \om: \nabla_x G_\rho(x,y) >s \}| \lesssim s^{-\frac{n}{n-1}}, \quad \textup{for all}\,\, s>0,\label{eq:green-rho-Lorentz2}
\end{align}
uniformly in $\rho>0$.  This yields that $G_\rho(\cdot,y) \in L^{\frac{n}{n-2}, \infty}(\om)$  and $\nabla G_\rho(\cdot,y) \in L^{\frac{n}{n-1}, \infty}(\om)$ with bounds independent of $\rho$. 

Moreover, in light of  \eqref{eq:green-rho-Lorentz1} and \eqref{eq:green-rho-Lorentz2}, we can mimic the proof of inequalities (3.24) and (3.26)  in \cite{HK} and infer that for any $\rho>0$ and $r<d_y $,
\begin{align}
&\| G_{\rho}(\cdot,y)\|_{L^p(B_r(y))} \lesssim r^{2-n+\frac{n}{p}}, \quad p \in (0, \tfrac{n}{n-2}),\label{eq:green-rho-Lp}\\
&\| \nabla G_{\rho}(\cdot,y)\|_{L^p(B_r(y))} \lesssim r^{1-n+\frac{n}{p}},  \quad p \in (0, \tfrac{n}{n-1}).\label{eq:green-rho-nabla-Lp}
\end{align}
In particular, 
\begin{align}\label{eq:green-rho-boundW1p}
&\|G_\rho(\cdot,y)\|_{W^{1,p}(B_r(y))} \leq C(r,p),\quad r<d_y , \,\,p \in [1, \tfrac{n}{n-1}),
\end{align}
uniformly in $\rho>0$. Thus,  fixing $p \in (1, \tfrac{n}{n-1})$, by a diagonalization argument, we can find  a subsequence of $\rho_m$ in \eqref{eq:green-convergY12} (which we still denote by $\rho_m$ for simplicity) so that
\begin{align}\label{eq:green-convergW1p}
G_{\rho_m}(\cdot,y) &\rightharpoonup \tilde G(\cdot, y) \,\, \textup{in}\,\,W^{1,p}(B_r(y))\quad \textup{as}\,\,m \to \infty,
\end{align}
for all $ r<d_y $. We also have that  $\tilde G(\cdot, y)$ satisfies \eqref{eq:Green-3} and \eqref{eq:Green-3-bis} for this particular $p$. Since $G(\cdot, y)=\tilde G(\cdot, y)$ in $B(y, d_y) \setminus B(y, d_y/2)$, we can extend $\tilde G(\cdot, y)$ by $ G(\cdot, y)$ to the entire $\om$ by setting $G(\cdot, y)=\tilde G(\cdot, y)$. 

Let $\om_t= \{x \in \om: G(x,y) > t \}$, $p=\frac{n}{n-2}$, $\ve \in (0, p-1)$. If we apply Chebyshev inequality, and then use that the $L^p$-norms are weakly lower semicontinuous and $|\om_t|<\infty$, by \eqref{eq:Green-2} and \eqref{eq:Green-3}, we have
\begin{align*}
t^{p-\ve} |\om_t| &\lesssim  \| G(\cdot, y)\|^{p-\ve}_{L^{p-\ve}(\om_t)} \leq \liminf_{m \to \infty} \| G_{\rho_m}(\cdot, y)\|^{p-\ve}_{L^{p-\ve}(\om_t)}\\
&\leq \liminf_{m \to \infty}   \frac{p}{\ve} | \om_t|^{\frac{\ve}{p}}\| G_{\rho_m}(\cdot, y)\|^{p-\ve}_{L^{p,\infty}(\om)} \overset{\eqref{eq:green-rho-Lorentz1}}{\leq}   \frac{p}{\ve} | \om_t|^{\frac{\ve}{p}} C^{p -\ve}.
\end{align*}
Letting $ \ve \to p-1$,  we get $|\om_t|^{\frac{1}{p}} \lesssim 1$ which proves \eqref{eq:Green-4}. A similar reasoning  proves  \eqref{eq:Green-5}. Moreover,
\begin{align}\label{eq:green-convergLorentz}
G_{\rho_m}(\cdot,y) &\overset{*}{\rightharpoonup}  G(\cdot, y) \,\, \textup{in}\,\,L^{\frac{n}{n-2}, \infty}(\om)\quad \textup{as}\,\,m \to \infty,\\
\nabla G_{\rho_m}(\cdot,y) &\overset{*}\rightharpoonup  \nabla  G(\cdot, y) \,\, \textup{in}\,\,L^{\frac{n}{n-1}, \infty}(\om)\quad \textup{as}\,\,m \to \infty.\label{eq:green-grad-convergLorentz}
\end{align}

Therefore, by \eqref{eq:green-rho-var} and \eqref{eq:green-rho-dual-average}, in view of \eqref{eq:green-convergLorentz}, \eqref{eq:green-grad-convergLorentz}, and \eqref{eq:green-convergY12},  we can prove \eqref{eq:green-phi} and also,  \eqref{eq:Green-1} for $f \in L^\infty(\om)$ and $g \in L^\infty(\om)$ so that $|\supp(f)|+|\supp(g)|<\infty$ (a detailed but more involved argument can be found after equation \eqref{Green:bk2-conv}). To show  that \eqref{eq:Green-1} holds in general, it is enough to use that simple functions are dense in $L^{p,q}(\om)$ if $q \neq \infty$ along with \eqref{eq:Green-4} and  \eqref{eq:Green-5}. Details are left to the reader.

The proof of inequalities (3.30) and (3.31)  in \cite{HK} gives us  \eqref{eq:Green-3} and  \eqref{eq:Green-3-bis} for any $p$ (in the stated range).

We will now demonstrate  that for a fixed $y \in \om$, $G(\cdot, y)  \geq 0$ a.e.  in $\om\setminus \{y\}$. Assume that $\sigma_n$ is the sequence converging to zero for which $G_{\sigma_n}(\cdot, y)$ converge to $G(\cdot, y)$ in the sense of  \eqref{eq:green-convergY12} and \eqref{eq:green-convergW1p}. If necessary, we can pass to a subsequence so that $\sigma_n <  \min(|x-y|, d_y)/10$. Fix $x \in \om$ so that $x \neq y$ and let $\rho_m$ be a sequence converging to zero  so that $\rho_m \leq \min(|x-y|, d_x)/10$. Therefore, since $ G_{\sigma_n}(\cdot, y) \geq 0$ in $ \om$, we have that
\begin{equation*}
0 \leq \avint_{B_{\rho_m}(x)} G_{\sigma_n}(\cdot, y) \longrightarrow \avint_{B_{\rho_m}(x)}G(\cdot, y), \,\, \textup{as}\,\, n \to \infty,
\end{equation*}
where we used \eqref{eq:green-convergY12} in the case $B_{\rho_m}(y) \subset \om \setminus B_r(x)$ for some $r>0$ and \eqref{eq:green-convergW1p} in the case  $B_{\rho_m}(x) \cap B_{\sigma_n}(y) \neq \emptyset$.
By Lebesgue differentiation theorem, if we let $m \to \infty$, we infer that $G(x, y) \geq 0$ for a.e. $x \in \om\setminus\{y\}$.

To prove uniqueness of the Green's function, we assume that $ \widehat G(\cdot, y)$ is another Green's function for the same operator. Then for $f \in C^\infty_c(\om)$ and $g=0$, we have that for fixed $y \in \om$, 
$$
 \int_\om \widehat G(\cdot, y) \, f =\widehat u(y) \in Y^{1,2}_0(\om) \quad \textup{and}\quad L^t \widehat u =f.
$$
By the comparison principle Corollary \ref{cor:comparison-princ}, $u = \widehat u$  in $\om$ and so,
$$
\int_\om G(\cdot, y) \, f =\int_\om \widehat G(\cdot, y) \, f.
$$
Since $f \in C^\infty_c(\om)$ is arbitrary, this  readily implies that $ G(x, y)=\widehat G(x, y)$  for a.e.  $x \in\om \setminus \{y\}$. 

So far, we have not used the local boundedness of solutions of $L^t u=0$ and thus, the assumption  $|b+c|^2 \in \khalf(\om)$. It is only for the pointwise bounds we will need it.  Indeed, let $x,y \in \om$, $x \neq y$ and set $r= |x-y|/4$. Then,  \eqref{eq:green-phi} yields that $LG(\cdot,y)=0$ away from $y$. So, by Theorem \ref{thm:Moser} and \eqref{eq:Green-2} for $ p=2$, we obtain
\begin{align}\label{eq:green-ptwse}
|G(x,y)| &\leq \sup_{\om_r(x)} |G(\cdot,y)| \lesssim r^{-n/2}  \|G(\cdot,y)\|_{L^2(\om_r(x))}\\
&\lesssim r^{-n/2} r^{2-n/2} \approx |x-y|^{2-n}.\notag
\end{align}

Notice that, under the additional assumption $|b+c|^2 \in \khalf(\om)$, we can apply the  previous considerations  to  construct the Green's function $G^t(\cdot, y)$ associated with the operator $L^t$ with all the properties above. The only thing that remains to be shown is that $G^t(x,y)= G(y,x)$ for a.e. $(x,y) \in \om^2\setminus \{x=y\}$.  We will first prove it in the case that solutions of $Lu=0$ and  $L^t u=0$ are locally H\"older continuous in $\om\setminus\{x\}$ and $\om \setminus\{y\}$ respectively. In this case,  all the properties that hold a.e. in  $\om \setminus \{ \textup{pole}\}$, because of  the continuity therein, will actually hold everywhere in $\om \setminus \{ \textup{pole}\}$. 

To this end, let $\sigma_n$ and $\rho_m$  be the sequences converging to zero for which $G_{\sigma_n}(\cdot, x)$ and $G^t_{\rho_m}(\cdot, y)$ converge to $G(\cdot, x)$ and $G^t(\cdot, y)$ in the sense of  \eqref{eq:green-convergY12}, \eqref{eq:green-convergW1p}, and \eqref{eq:green-convergLorentz}. If necessary, we may further pass to subsequences so that 
$$\sigma_n < \min(|x-y|, d_x)/10 \quad \textup{and}\quad\rho_m \leq \min(|x-y|, d_y)/10.$$
Because $G_{\sigma_n}(\cdot, x)$ and $G_{\rho_m}^t(\cdot, y)$ are locally H\"older continuous in $\om \setminus\{x\}$ and $\om \setminus\{y\}$ respectively, with constants uniform in $\sigma_n$ and $\rho_m$ and, by Theorem \ref{thm:Moser}, they are uniformly bounded on compact subsets of the respective domains, we may pass to subsequences so that 
\begin{align}\label{eq:Green-unif-conv}
G_{\sigma_n}(\cdot, x) &\to G(\cdot ,x) \,\,  \textup{unifomly on compact subsets of }\, \om\setminus\{x\},\\
G^t_{\rho_m}(\cdot, y) &\to G^t(\cdot ,y) \,\,  \textup{unifomly on compact subsets of }\, \om\setminus\{y\}.\label{eq:Green-t-unif-conv}
\end{align}

We now use $G^t_{\rho_m}(\cdot, y)$ and  $G_{\sigma_n}(\cdot, x)$ as test functions in their very definitions to obtain
\begin{align*}
 \avint_{B_{\sigma_n}(x)}  G^t_{\rho_m}(\cdot, y)&= \mathcal{L}( G_{\sigma_n}(\cdot, x), G^t_{\rho_m}(\cdot, y))\\ &= \mathcal{L}^t(  G^t_{\rho_m}(\cdot, y), G_{\sigma_n}(\cdot, x))=  \avint_{B_{\rho_m}(y)} G_{\sigma_n}(\cdot, x).
\end{align*}
By Lebesgue's differentiation theorem and continuity of $G_{\sigma_n}(\cdot, x)$ in $\om \setminus \{x\}$, 
$$
\lim_{m \to \infty}  \avint_{B_{\rho_m}(y)} G_{\sigma_n}(\cdot, x) = G_{\sigma_n}(y, x),
$$
which, in view of \eqref{eq:Green-unif-conv}, yields that 
$$
\lim_{n \to \infty}  \lim_{m \to \infty}  \avint_{B_{\rho_m}(y)} G_{\sigma_n}(\cdot, x) = G(y, x) \quad \textup{for all}\,y \in \om \setminus \{x\}.
$$
On the other hand, the weak convergence of $G^t_{\rho_m}(\cdot, y)$ in $Y^{1,2}(\om \setminus B_r(y))$ for any $r>0$ implies
$$
\lim_{m \to \infty}  \avint_{B_{\sigma_n}(x)}  G^t_{\rho_m}(\cdot, y) =  \avint_{B_{\sigma_n}(x)}  G^t(\cdot, y),
$$
from which, by Lebesgue differentiation theorem and the continuity of $G^t(\cdot, y)$ in $\om \setminus \{y\}$, we deduce that
$$
\lim_{n \to \infty} \lim_{m \to \infty}  \avint_{B_{\sigma_n}(x)}  G^t_{\rho_m}(\cdot, y) =  G^t(x, y) \quad \textup{for all}\,x \in \om \setminus \{y\}.
$$
Therefore, $G(x,y)=G^t(y,x)$ for all  $(x, y)  \in \om^2 \setminus \{x=y\}$, which, combined with \eqref{eq:Green-1}, implies \eqref{eq:Green-1-bis}. 

We are now ready to remove the H\"older continuity assumption. 
Set
$$
\om_k = \{ x\in\om: d(x,\d\om)>k^{-1}\} \cap B(0,k),
$$
which are  open sets such that $\cup_{k \geq 1} \om_k = \om$. Let $\psi \in C^\infty_c(\R^n)$ so that 
$$
0 \leq \psi \leq 1,\,\, \psi = 0\,\, \textup{in}\,\, \R^n \setminus B(0,1) \,\, \textup{and}\,\, \int \psi =1.
$$
For $k \in \N$, set $\psi_k(x)= k^{n} \psi(k x)$ and define $b_k = (b\,{\bf 1}_{\om_k}) \ast \psi_k$, $c_k = (c\,{\bf 1}_{\om_k}) \ast \psi_k$ and $d_k = (d\,{\bf 1}_{\om_k}) \ast \psi_k$.

Define 
$$
L_k u = -\divv A\nabla u - \divv (b_k u) - c_k \nabla u- d_k u.
$$

If we fix $x \neq y \in \om$, there exists $k_0$ large enough such that $x, y \in \om_k$ for every $k \geq k_0$ and in particular, $x$ and $y$ are in the same connected component of $\om_k$.  Therefore, Remark \ref{rem:molifier-moser-indep} applies, and since, for such $k$, Theorem \ref{thm:Moser} holds for $L_k $ in $\om_k$ with bounds independent of $k$, we can  construct the Green's functions $G_k(\cdot, y)$ and $G_k^t(\cdot, x)$ associated with $L_k$  and $L^t_k$ in $\om_k$ as above, with the additional property that $G_k(\cdot,x)$ and $G_k^t(\cdot, y)$ are locally H\"older continuous away from $x$ and $y$ respectively. In the last part we used Theorem \ref{thm:Holder-cont}, which applies in this situation, since $b_k, c_k, d_k \in L^\infty$ with compact support and thus, $|b_k|^2, |c_k|^2, |d_k| \in \khalf(\om_k)$ (with implicit constants depending in the domain). Extend both  $G_k(\cdot,x)$ and $G_k^t(\cdot, y)$ by zero outside $\om_k$ and note that \eqref{eq:Green-2}-\eqref{eq:Green-5} hold in $\om$ with constants independent of $k$ (see  Remark \ref{rem:molifier-moser-indep}). Therefore, repeating essentially the arguments concerning the convergence of $G_\rho$ and the inheritance of the bounds from  $G_\rho$, we can find ${G}(\cdot, y)$ which is non-negative a.e. in $\om \setminus \{y\}$ and vanishes on $\d \om$. Additionally, it satisfies \eqref{eq:Green-2}-\eqref{eq:Green-5}, and, after passing to a subsequence, 
\begin{align}
G_k(\cdot,y)& \rightharpoonup  {G}(\cdot, y) \,\, \textup{in}\,\, Y^{1,2}(\om \setminus B_{r}(y))\,\, \text{for all}\,\,r>0, \label{eq:green-k-convergY12-r}\\
G_k(\cdot,y) &\rightharpoonup {G}(\cdot, y) \,\, \textup{in}\,\,W^{1,p}(B_{r}(y)),\,\, \text{for all}\,\,r<d_y, \label{eq:green-k-convergW1p-r}\\
G_{k}(\cdot,y) &\overset{*}{\rightharpoonup}    G(\cdot, y) \,\, \textup{in}\,\,L^{\frac{n}{n-2}, \infty}(\om),\label{eq:green-k-convergLorentz}\\
\nabla G_{k}(\cdot,y) &\overset{*}\rightharpoonup  \nabla    G(\cdot, y) \,\, \textup{in}\,\,L^{\frac{n}{n-1}, \infty}(\om)\label{eq:green-k-grad-convergLorentz},\\
G_{k}(\cdot,y) &\rightarrow G(\cdot,y) \,\, \textup{a.e. in } \, \om.\label{eq:green-k-convergpointwise}
\end{align}
The  considerations  above apply to $G_k^t$ as well.

Let  $f \in L^\infty(\om)$ and  $g \in L^\infty_\om)$ which supports have finite Lebesgue measure. Thus, by virtue of \eqref{eq:Green-1}, we have that   
\begin{equation}\label{eq:green-uk-gtk}
u_k(y) = \int_{\om} {G}_k( \cdot, y)\, f+\int_{\om} \nabla {G}_k(\cdot,y)\, g.
\end{equation}
Since $u_k \in Y^{1,2}_0(\om_k)$, we can extend it by $0$ outside $\om_k$. Recall that $u_k$ satisfies $L_k^t u_k=f-\divv g$ in $\om_k$ and also
$$
\|u_k\|_{Y^{1,2}(\om)} =\|u_k\|_{Y^{1,2}(\om_k)} \lesssim \|f\|_{L^{2_*}(\om_k)}+\|g\|_{L^{2}(\om_k)}\leq \|f\|_{L^{2_*}(\om)}+\|g\|_{L^{2}(\om)},
$$
where the implicit constant is independent of $k$.  If we take limits in \eqref{eq:green-uk-gtk} as $k \to \infty$ and use \eqref{eq:green-k-convergLorentz} and \eqref{eq:green-k-grad-convergLorentz} for $G_k^t(\cdot, y)$, we can show that for all $y \in \om$,
\begin{align}\label{eq:green-uk-u}
\lim_{k \to \infty} u_k(y)&=\lim_{k \to \infty} \int_{\om} {G}_k(x,y)\, f(x)\,dx +\lim_{k \to \infty} \int_{\om} \nabla {G}_k(x,y)\, g(x)\,dx \\
&= \int_{\om} {G}( x, y)\, f(x)\,dx +\int_{\om} \nabla {G}(x,y)\, g(x)\,dx =: u(y).\notag
\end{align}
Therefore, since $u_k \to u$ pointwisely in $\om$ and $u_k$ is a uniformly bounded sequence in $Y_0^{1,2}(\om)$, it holds that $u_k \rightharpoonup u\,\,\textup{in}\,\,  Y^{1,2}(\om)$ and $u \in Y^{1,2}_0(\om)$. For a proof see for instance  \cite[Theorem 1.32]{HKM}.  We will show that   $u$ is the unique solution of the Dirichlet problem $L^t u =f$ and $u \in Y^{1,2}_0(\om)$. If $\vphi \in C^\infty_c(\om)$, there exists $k_1 \geq k_0$  such that $\vphi \in C^\infty_c(\om_k)$ for every $k \geq k_1$. Thus,
$$
\mathcal{L}^t_{k, \om} (u_k, \vphi)= \mathcal{L}^t_{k, \om_k} (u_k, \vphi) = \int_{  \om_k} f \vphi +  \int_{\om_k} g \nabla \vphi=\int_{  \om} f \vphi +  \int_{ \om} g \nabla \vphi.
$$

To pass to the limit, we need to  treat each of the terms of the bilinear form separately. We first write
\begin{equation*}
\int_\om b_k \nabla u_k \phi=\int_\om (b_k-b) \nabla u_k \phi+ \int_\om b \nabla u_k \phi = I_{b,1}^k+I_{b,2}^k.
\end{equation*}
If $b\in L^{n,q}(\om)$, by Lemma \ref{Lorentz-molifier} we have that $b_k \to b$ in $L^{n,q}(\om)$,  which, combined with  \eqref{eq:Lorentz-Sobolev-emb} and the uniform $Y^{1,2}$-bound of $u_k$, yields that $ \lim_{k \to \infty} I_{b,1}^k =0$. To prove that 
\begin{equation}\label{Green:bk2-conv}
 \lim_{k \to \infty} I_{b,2}^k = \int_\om b \nabla u \phi,
\end{equation}
  it is enough to notice that, by H\"older inequality in Lorentz spaces and Lemma \ref{lem:Lor-Sobolev-emb}, $b \phi \in L^2(\om)$, and then use that $\nabla u_k \rightharpoonup \nabla u$ in $L^2(\om)$. If $|b|^2\in \mathcal{K}'(\om)$, we combine Cauchy-Schwarz inequality, Lemma \ref{lem:Kato-emb-grad-om}, the uniform $Y^{1,2}$-bound of $u_k$, and Lemma \ref{lem:kato-approx-0}, to show $I_{b,1}^k \to 0$. By \eqref{eq:Kato-emb-grad-om}, we have that  $b \phi \in L^2(\om)$, and thus, \eqref{Green:bk2-conv} follows from the weak-$L^2$ convergence of $\nabla u_k$ to $\nabla u$.   Let us now prove the limit for the one involving $d_k$. To this end, write
  \begin{equation*}
\int_\om d_k  u_k \phi=\int_\om (d_k-d) u_k \phi+ \int_\om d  u_k \phi = I_{d,1}^k+I_{d,2}^k.
\end{equation*}
 If $d \in L^{\frac{n}{2},q}(\om)$,  $d_k \to d$ in $L^{\frac{n}{2},q}(\om)$,  which, by  H\"older inequality for Lorentz spaces, \eqref{eq:Lor-inclusion}, \eqref{eq:Lor-Sob-uw}, and the uniform $Y^{1,2}$-bound of $u_k$, yields that $ \lim_{k \to \infty} I_{d,1}^k =0$. Moreover, as $u_k \to u$ pointwisely, we can apply the dominated convergence theorem to obtain
 \begin{equation}\label{Green:dk2-conv}
 \lim_{k \to \infty} I_{d,2}^k = \int_\om d u \phi.
\end{equation}
If $ |d| \in \mathcal{K}'(\om)$, we first apply Cauchy-Schwarz inequality, and then use Lemma \ref{lem:Kato-emb-grad-om} and the uniform $Y^{1,2}$-bound of $u_k$. Finally, in view of Lemma \ref{lem:kato-approx-0}, we can take limits as $k \to \infty$ to conclude that  $\lim_{k \to \infty} I_{d,1}^k$. The  proof of \eqref{Green:dk2-conv} follows by dominated convergence. The integral involving $c_k$  can be treated very similarly and the details are left to the reader. We have thus proved that 
$$
\mathcal{L}^t_{ \om} (u, \vphi)=\lim_{k \to \infty} \mathcal{L}^t_{ k,\om} (u_k, \vphi)=\int_{  \om} f \vphi +  \int_{ \om} g \nabla \vphi,
$$
which, in turn, yields that $u$ is the unique solution of the Dirichlet problem $L^t u =f - \divv g$ and $u \in Y^{1,2}_0(\om)$.

Let us now recall thatfrom the first part of the proof (before the approximation) we can construct a Green's function $\widehat G(\cdot, y)$ associated with $L$ so that the function
$$
\widehat u(y)=\int_\om \widehat G(x,y) \,f(x) \,dx + \int_\om \nabla_x \widehat G(x, y) \,g(x) \,dx,
$$
is also a solution of the Dirichlet problem $L^t \widehat u =f - \divv g$ and $\widehat u \in Y^{1,2}_0(\om)$. But since there is only one such solution we must have $u = \widehat u$, which, as we showed before,  implies that $G(x,y)=\widehat G(x,y)$, for a.e. $x \in \om \setminus \{y\}$. As we have shown that \eqref{eq:Green-2} holds for $\widehat G(x,y)$, it also holds for $G(x,y)$.

The same arguments are  valid if we replace ${G}$ by ${G}^t$ and $L$ by $L^t$ (and vice versa), implying that
\begin{align}\label{eq:green-uk-t-u}
\lim_{k \to \infty} u_k^t(x)&=\lim_{k \to \infty} \int_{\om} {G^t}_k(y, x)\, f(y)\,dy  + \lim_{k \to \infty}  \int_\om \nabla_y {G^t}_k(y, x)\, f(y)\,dy\\
&= \int_{\om} {G^t}( y, x)\, f(y)\,dy+ \int_\om \nabla_y {G^t}(y, x)\, f(y)\,dy=: u^t(x),
\end{align}
and after passing to a subsequence, $u^t_k \rightharpoonup u^t\,\,\textup{in}\,\,  Y^{1,2}(\om)$, $u^t \in Y^{1,2}_0(\om)$, and $L u^t =f$ in $\om$.

For $f, g \in C^\infty_c(\om)$  we set 
\begin{align*}
u_{f,k}(y)&= \int G_k(x,y) \,f(x)\,dx \quad \textup{and}\quad u^t_{g,k}(x)= \int G^t_k(y,x) \,g(y)\,dy;\\
u_{f}(y)&= \int {G}(x,y) \,f(x)\,dx \quad \textup{and}\quad u^t_{g}(x)= \int {G}^t(y,x) \,g(y)\,dy.
\end{align*}
Recall that
\begin{equation*}
u_{f,k} \rightharpoonup u_f\,\,\textup{in}\,\,  Y^{1,2}(\om)\quad \textup{and}\quad u_f \in Y^{1,2}_0(\om),
\end{equation*}
and 
\begin{equation*}
u^t_{g,k} \rightharpoonup u_g^t\,\,\textup{in}\,\,  Y^{1,2}(\om)\quad \textup{and}\quad u^t_g \in Y^{1,2}_0(\om).
\end{equation*}
By Fubini theorem and $G^t_k(x,y)=G_k(y,x)$ for all $(x,y) \in \om^2 \setminus \{x=y\}$,  we have that
\begin{align}\label{eq:green-ukfg=uktgf}
\int u_{f,k}(y) \,g(y)\,dy &= \int g(y) \int G_k(x,y) \,f(x)\,dx\,dy \\
&= \int f(x) \int G^t_k(y,x) \,g(y)\,dy\,dx= \int u^t_{g,k}(x) \,f(x)\,dx.\notag
\end{align}
If we take limits as $k \to \infty$ in \eqref{eq:green-ukfg=uktgf},
$$
\int u_{f}(y) \,g(y)\,dy = \int u^t_{g}(x) \,f(x)\,dx,
$$
which  implies
$$
\int \int {G}(x,y) \,f(x) \,g(y)\,dx\,dy = \int \int {G}^t_(y,x) \,g(y) \,f(x)\,dy \,dx.
$$
Since $f, g \in C^\infty_c(\om)$ are arbitrary, we conclude that ${G}^t(x,y) ={G}(y,x) $ for a.e. $(x,y) \in \om^2 \setminus \{x=y\}$. 

Once we have that \eqref{eq:Green-3} holds, the proof of \eqref{eq:Green-pointwise} is the same as in \eqref{eq:green-ptwse}, while  \eqref{eq:Green-1} follows by  density.
\end{proof}

\begin{remark}
If  $\vphi \in C^\infty_c(\om)$ and it holds that $b \nabla \vphi \in L^{\frac{n}{2},1}(\om)$, $c  \vphi \in L^{n,1}(\om)$, and $d  \vphi \in L^{\frac{n}{2},1}(\om)$, then we can show that 
$$
\mathcal{L}(G(\cdot, y), \vphi)= \vphi(y).
$$
This is straightforward if we use \eqref{eq:Green-4} and  \eqref{eq:Green-5}.
\end{remark}

\vv

Finally, we can prove that, under certain restrictions, the Green's function has pointwise  lower bounds as well.

\begin{lemma}
Let $\Omega \subset \R^{n}$ be an open and connected set and $Lu=-\divv (A\nabla u + b u)$ be an elliptic operator so that  $b \in \khalf(\om)$. Let $x,y \in \om$, $x \neq y$,  such that $2|x-y| < \dist(\{x,y\}, \partial \om)$. If we set $r= |x-y| /4$, then the Green's functions $G$  constructed in Theorem \ref{thm:green-contruction} satisfy the following lower bound: 
\begin{align}
G(x,y) &\gtrsim \frac{1}{|x-y|^{n-2}},\label{eq:green-lowerbd}\\
G^t(x,y) &\gtrsim \frac{1}{|x-y|^{n-2}}. \label{eq:green-lowerbd}
  \end{align}
  \end{lemma}

\begin{proof}
Let us fix $x, y \in \om$ with $x \neq y$. If we set $r = \frac{|x-y|}{4}$ and let $\eta \in C_0^{\infty}(B_r(y))$ be a bump function so that 
$$
0 \leq \eta \leq 1, \quad \eta =1 \,\, \textup{in}\,\, B_{\frac{r}{2}}(y),\,\, \textup{and}\quad |\nabla \eta| \lesssim \frac{1}{r}.
$$ 
Then using it as a test function we have that
\begin{align*}
1=\eta(y)&= \mathcal{L}(G(\cdot,y), \eta) = \int_\om A \nabla G(\cdot,y) \nabla \eta + \int_\om b G(\cdot,y) \nabla \eta\\
& \lesssim \frac{1}{r} \|\nabla G(\cdot,y)\|_{L^1(B_r(y) \setminus B_{\frac{r}{2}}(y))} + \frac{1}{r} \|b\|_{L^n(\om)} \|G(\cdot,y)\|_{L^{\frac{n}{n-1}}(B_r(y) \setminus B_{\frac{r}{2}}(y))}\\
& \lesssim \frac{1}{r^2} \| G(\cdot,y)\|_{L^1(B_{2r}(y) \setminus B_{\frac{r}{8}}(y))},
\end{align*}
where we used H\"older, Sobolev and Caccioppoli inequality, along with Lemma \ref{lem:reverseHold}. Thus, from \eqref{eq:weak-Harnack-int}, we have that $G(x,y)\gtrsim \frac{1}{|x-y|^{n-2}}$.

Let $v\in Y^{1,2}(\om)$ be a nonnegative function such that $Lv =0$ and $v(y) > 0$, and let  $\eta$ be the bump function defined above. Then, if we assume $\rho \leq \min\left(\frac{|x-y|}{10}, \frac{d_y}{10}, \frac{d_x}{10}\right)$,
\begin{align*}
\avint_{B_\rho(y)} \eta \,v&= \LL^t(G^t_\rho(\cdot, y),\eta \,v )\\
&= \int_{\om} A^t \nabla G^t_\rho(\cdot,y)\,\nabla \eta \,v - A^t\nabla \eta \nabla v \,G^t_\rho(\cdot,y)+ A\nabla v \nabla(G^t_\rho(\cdot,y) \eta)\\
&+ \int_\om b\,\nabla v G^t_\rho(\cdot,y) \eta - \int_{\om}  b\, \nabla \eta \,G^t_\rho(\cdot,y) \,v \\
& = \int_{\om}A^t \nabla G^t_{\rho}(\cdot,y)\nabla \eta\, v
-A^t\nabla \eta\nabla v \,G^t_{\rho}(\cdot,y)-b\,\nabla \eta \,G^t_{\rho}(\cdot,y)\,v\\
 &=:I_1-I_2-I_3,
\end{align*}
where we used that $G^t_{\rho}(\cdot,y) \eta$ is a test function and $Lv=0$. We will only estimate $I_3$ since $I_1$ and $I_2$ can be handled similarly.
\begin{align*}
|I_3| &\lesssim \frac{1}{r}\|b+c\|_{L^n(B_r(y) \setminus B_{\frac{r}{2}}(y))}\, \|G^t_{\rho}(\cdot,y)\|_{L^2(B_r(y) \setminus B_{\frac{r}{2}}(y))}\,\|v\|_{L^{2^*}(B_r(y) \setminus B_{\frac{r}{2}}(y))}\\
&\lesssim \frac{1}{r^{2}}\, \|G^t_{\rho}(\cdot,y)\|_{L^2(B_{r}(y) \setminus B_{\frac{r}{2}}(y))}\,\| v\|_{L^{2}(B_{\frac{3r}{2}}(y) \setminus B_{\frac{3r}{8}}(y))},
\end{align*}
where in the first inequality we used H\"older inequality and in the second one the local bonudedness of $v$. If $\rho_m$ is the sequence obtained in \eqref{eq:green-convergY12}, then by Rellich-Kondrachov theorem and a diagonalization argument, we may pass to a subsequence so that 
$$
 G^t_{\rho_m}(\cdot, y) \to G^t(\cdot, y), \quad \textup{strongly in}\,\, L^2(B_r(y) \setminus B_{\frac{r}{2}}(y)).
$$
Thus, if we take $m \to \infty$, by  Lemma \ref{lem:reverseHold}, for a.e. $y \in \om$,
\begin{align*}
v(y) &= \eta(y) \, v(y)=\lim_{m \to \infty} \avint_{B_{\rho_m}(y)} \eta \,v \\
  &\lesssim  \lim_{m \to \infty} \frac{1}{r^{2}}\, \|G^t_{\rho_m}(\cdot,y)\|_{L^2(B_{r}(y) \setminus B_{\frac{r}{2}}(y))}\,\| v\|_{L^{2}(B_{\frac{3r}{2}}(y) \setminus B_{\frac{3r}{8}}(y))}\\
&=\frac{1}{r^{2}}\, \|G^t(\cdot,y)\|_{L^2(B_{r}(y) \setminus B_{\frac{r}{2}}(y))}\,\| v\|_{L^{2}(B_{\frac{3r}{2}}(y) \setminus B_{\frac{3r}{8}}(y))}\\
&\lesssim  \frac{1}{r^{n+2}}\, \|G^t(\cdot,y)\|_{L^1(B_{2r}(y) \setminus B_{\frac{r}{4}}(y))}\,\| v\|_{L^{1}(B_{2r}(y) \setminus B_{\frac{r}{4}}(y))}.
\end{align*}
So, by \eqref{eq:weak-Harnack-int} and Remark \eqref{rem:c-d-moser}, we get
$$
v(y)  \lesssim   |x-y|^{n-2}\,G^t(x,y) \, v(y),
$$
which implies \eqref{eq:green-lowerbd}.
\end{proof}

\end{document}